\documentclass[reqno, 11pt, letterpaper ]{amsart}

\usepackage{amsmath}
\usepackage{amsfonts}
\usepackage{amssymb}
\usepackage{graphicx}
\usepackage{amsthm,color,yfonts,cite} \usepackage{paralist}
\usepackage{hyperref}

\usepackage{subfig}

\DeclareMathOperator*{\esssup}{ess\,sup}

\newtheorem{theorem}{Theorem}

\newtheorem{proposition}[theorem]{Proposition}
\newtheorem{lemma}[theorem]{Lemma}
\newtheorem{corollary}[theorem]{Corollary}

\newtheorem{example}[theorem]{Example}
\theoremstyle{remark}
\newtheorem{remark}[theorem]{Remark}

\makeatletter
\newcommand*{\rom}[1]{\expandafter\@slowromancap\romannumeral #1@}
\makeatother

\newcommand{\R}{\mathbb{R}}

\newcommand{\C}{\mathbb{C}}

\usepackage{wrapfig}
\usepackage{tikz}
\usetikzlibrary{arrows,calc,decorations.pathreplacing}
\definecolor{light-gray1}{gray}{0.90}
\definecolor{light-gray2}{gray}{0.80}
\definecolor{light-gray3}{gray}{0.60}

\numberwithin{equation}{section}

\numberwithin{theorem}{section}

\numberwithin{table}{section}

\numberwithin{figure}{section}

\ifx\pdfoutput\undefined
  \DeclareGraphicsExtensions{.pstex, .eps}
\else
  \ifx\pdfoutput\relax
    \DeclareGraphicsExtensions{.pstex, .eps}
  \else
    \ifnum\pdfoutput>0
      \DeclareGraphicsExtensions{.pdf}
    \else
      \DeclareGraphicsExtensions{.pstex, .eps}
    \fi
  \fi
\fi

\title[NLS with a toroidal trap]{On the nonlinear Schr\"odinger equation with a toroidal trap in the strong confinement regime}

\date{\today}
\author[Y. Hong]{Younghun Hong}
\address{Department of Mathematics, Chung-Ang University, Seoul 06974, Korea}
\email{yhhong@cau.ac.kr}

\author[S. Jin]{Sangdon Jin}
\address{Department of Mathematics, Chung-Ang University, Seoul 06974, Korea}
\email{sdjin@cau.ac.kr}

\begin{document}

\begin{abstract}
We consider the 3D cubic nonlinear Schr\"odinger equation (NLS) with a strong toroidal trap. In the first part, we show that as the confinement is strengthened, a large class of global solutions to the time-dependent model can be described by 1D flows solving the 1D periodic NLS (Theorem \ref{dimension reduction for tNLS}). In the second part, we construct a steady state as a constrained energy minimizer, and prove its dimension reduction to the well-known 1D periodic ground state (Theorem \ref{thm: existence} and \ref{drth}). Then, employing the dimension reduction limit, we establish the local uniqueness and the orbital stability of the 3D ring soliton (Theorem \ref{thm: unique}). These results justify the emergence of stable quasi-1D periodic dynamics for Bose-Einstein condensates on a ring in physics experiments.
\end{abstract}

\maketitle


\section{Introduction}

\subsection{Bose-Einstein condensate in a toroidal trap}

A Bose-Einstein condensate (BEC) is a state of matter where almost all particles have the same quantum state. It is produced by cooling a dilute Bose gas down to near absolute $0$K. This exotic material form was first proposed in 1924-25 by Bose and Einstein. Many decades later, in 1995, Cornell and Wieman in JILA group\cite{AE} and Ketterle's group \cite{DM} independently succeeded in creating BECs in laboratory using alkali vapor. For weakly interacting BECs, via the mean-field approximation, the particle dynamics is described by a single wave function governed by the cubic nonlinear Schr\"odinger equation (NLS) also known as the Gross-Pitaevskii equation \cite{ES, ES1, ES2, ES3, LSS}.

In most experiments, BECs are confined in a trapping potential, and different trapping geometries determine different shapes of condensates. In particular, controlling confining frequencies in an anisotropic harmonic trap, quasi-lower dimensional cigar-shaped and disc-shaped condensates can be produced. In this way, BECs for attractive particles are possibly stabilized in a quasi-lower dimensional regime \cite{BBJV, BM, BM1, CCR, Khaykovich, RA}. For various dimension reduction processes, we refer to the survey articles \cite{BaoCai, CF} and to \cite{Bossmann 1d, Bossmann 2d, Bossmann-Teufel, CH 2d 2013, CH, CH 2d, HT, HJ2} for rigorous proofs in various settings.

In this paper, we are particularly concerned with a BEC confined in a toroidal trap. In experiments, a toroidal trap can be realized by optical and/or magnetic potentials \cite{RA, W}. For instance, the NIST atomic physics group constructed a potential of the form\footnote{In \cite{RA}, a fully anisotropic potential $\frac{m}{2}(\omega_{y_1}^2y_1^2+\omega_{y_2}^2y_2^2+\omega_z^2 z^2)+V_0e^{-2(y_1^2+y_2^2)/w_0^2}$ is employed with $\frac{\omega_{y_1}}{2\pi}=36$, $\frac{\omega_{y_2}}{2\pi}=51$ and $\frac{\omega_{z}}{2\pi}=25$ (Hz), where $\omega_{y_1}$, $\omega_{y_2}$ and $\omega_z$ are harmonic trapping frequencies, $V_0$ is the maximum optical potential and $w_0$ is the Gaussian beam waist. However, in this article, we set $\omega_{y_1}=\omega_{y_2}$ for mathematical simplicity.}
\begin{equation}\label{NIST potential}
\frac{m}{2}\big(\omega_r^2|y|^2+\omega_z^2 z^2\big)+V_0\exp\left(-\frac{2|y|^2}{w_0^2}\right),\quad (y,z)\in\mathbb{R}^2\times\mathbb{R}
\end{equation}
plugging a repulsive Gaussian laser beam in the $z$-direction in the middle of a magnetic harmonic trap  \cite{ RA}. An important property of this toroidal trap potential is that it has the minimum value on a ring provided that the Gaussian beam is strong enough. Precisely, if $V_0>V_{\text{crit}}:=\frac{m\omega_r^2w_0^2}{4}$, then the potential attains its minimum on the ring $\{(y,0): |y|=r_*\}$ with $r_*=\frac{w_0}{\sqrt{2}}\big(\ln(\frac{V_0}{V_{\text{crit}}})\big)^{1/2}$, and it is formally expanded as 
$$V_{\text{crit}}\left\{1+\ln\left(\frac{V_0}{V_{\text{crit}}}\right)\right\}+m\omega_r^2\ln\left(\frac{V_0}{V_{\text{crit}}}\right)\big(|y|-r_*\big)^2+\frac{m}{2}\omega_z^2z^2+\cdots$$
near the ring.\footnote{If the Gaussian beam is weak ($V_0\leq V_{\textup{crit}}$), then  the potential can be approximated by a harmonic + quartic potential $V_0+\frac{2(V_{\textup{crit}}-V_0)}{w_0^2}|y|^2+\frac{2V_0}{w_0^4}|y|^4+\frac{m}{2}\omega_z^2z^2$ near the origin  \cite{BS, FZ, F}.}  Thus, a trapped BEC can be placed on a ring-shaped region if the potential is strong enough or if a Bose gas is sufficiently diluted. As a consequence, a quasi-1D periodic dynamics may arise from the unbounded 3D system, and a BEC becomes easier to analyze as well as it exhibits rich dynamics. In addition, in a toroidal trap, a persistent flow of a BEC with 10-second lifetime is observed \cite{RA,WY}; however, quantum vortices will not be discussed in this article.   

\subsection{Model description}
Let $\omega\geq1$ be sufficiently large. Throughout the article, we assume that  a smooth function
$$U_\omega=U_\omega(s):[-\sqrt{\omega},\infty)\to [0,\infty)$$
is close to $s^2$ near the origin, and that it has global quadratic lower and upper bounds;
$$\begin{aligned}
&&\textbf{\textup{(H1)}}&\quad |U_\omega(s)-s^2|\lesssim \frac{1}{\sqrt{\omega}}s^2&&\textup{on }[-\sqrt{\omega}, \sqrt{\omega}]\\
&&\textbf{\textup{(H2)}}&\quad U_\omega(s)\sim s^2 &&\textup{for all }s\geq-\sqrt{\omega}
\end{aligned}$$
where the implicit constants are independent of large $\omega\geq 1$ (see Section \ref{sec: notation} for notations). Then, we consider a general strong toroidal trap potential of the form
$$\omega U_\omega(\sqrt{\omega}(|y|-1))+\omega^2z^2,\quad (y,z)\in\R^2\times \R.$$
Note that by the assumptions, 
$$\omega U_\omega(\sqrt{\omega}(|y|-1))+\omega^2z^2\underset{\omega\to\infty}\approx\omega^2\big((|y|-1)^2+z^2\big)$$
on the disk $\{(y,0): |y|\leq 2\}$ and it has a lower bound $\sim\omega^2((|y|-1)^2+z^2)$ in $\mathbb{R}^3$.
\begin{example}
\begin{enumerate}[$(i)$]
\item If $U_\omega(s)=s^2$, then
$$\omega U_\omega(\sqrt{\omega}(|y|-1))+\omega^2z^2=\omega^2((|y|-1)^2+z^2).$$
\item $U_\omega(s)=\frac{m\omega}{2}\{(1+\frac{s}{\sqrt{\omega}})^2+me^{\frac{1}{m}(1-(1+\frac{s}{\sqrt{\omega}})^2)}-(m+1)\}$ satisfies \textbf{\textup{(H1)}} and \textbf{\textup{(H2)}}. In this case, the toroidal potential is given by  
$$\omega U_\omega\big(\sqrt{\omega}(|y|-1)\big)+\omega^2z^2=\frac{m}{2}\left(\omega^2|y|^2+\frac{2\omega^2}{m}z^2\right)+\frac{m^2\omega^2}{2}e^{\frac{1}{m}(1-|y|^2)}-\frac{m(m+1)\omega^2 }{2},$$
which is, up to constant addition, the trapping potential \eqref{NIST potential} in the experiment \cite{RA} with $\omega_r^2=\omega^2$, $\omega_z^2=\frac{2\omega^2}{m}$, $m=\frac{w_0^2}{2}$, $V_{\textup{crit}}=\frac{m^2\omega^2}{2}$ and $V_0=V_{\textup{crit}}e^{\frac{1}{m}}$. Here, the choice of physical parameters seems rather restricted, but that is because the trap potential is chosen to be asymptotically $\omega^2((|y|-1)^2+z^2)$ just for simplicity. Indeed, the potential is allowed to be close to $\omega^2(c_r^2(|y|-r_*)^2+c_z^2 z^2)$ without any essential change in analysis.  Then, we have three additional degree of freedom for physical coefficients. 
\end{enumerate}
\end{example}

Suppose that a BEC is confined in the above toroidal trap. Then, its mean-field dynamics is described by the nonlinear Schr\"odinger equation (NLS)
\begin{equation}\label{NLS0}
i\partial_t\psi_\omega=-\Delta \psi_\omega+\omega U_\omega\big(\sqrt{\omega}(|y|-1)\big)\psi_\omega+\omega^2z^2\psi_\omega+\frac{\kappa}{\omega}|\psi_\omega|^2\psi_\omega, 
\end{equation}
where $\psi_\omega=\psi_\omega(t,x): I(\subset\mathbb{R})\times\mathbb{R}^3\to\mathbb{C}$ and $\omega\gg1$ represents the strength of the confinement and the weakness of particle interactions. The normalized coefficient $\kappa=\pm1$ determines the repulsive/attractive pair particle interaction in the condensate; the equation \eqref{NLS0} is called \textit{defocusing} (resp., \textit{focusing}) when $\kappa=1$ (resp., $\kappa=-1$).

By the assumptions, low energy states would be concentrated on the ring $\{(y,0)\in\mathbb{R}^2\times\mathbb{R}: |y|=1\}$. In order to observe the quasi-1D dynamics in the strong confinement regime, we consider the wave function of the form
\begin{equation}\label{ansatz}
\psi_\omega(t,x)=e^{-it\omega\Lambda_\omega}\omega^{\frac{3}{4}}u_\omega(t,\sqrt{\omega}x),
\end{equation}
where $\Lambda_\omega$ is the lowest eigenvalue for the Schr\"odinger operator
\begin{equation}\label{H_omega}
H_\omega=-\Delta+U_\omega(|y|-\sqrt{\omega})+z^2,
\end{equation}
and deduce the rescaled NLS  
\begin{equation}\label{NLS}
\boxed{\quad i\partial_t u_\omega=\omega(H_\omega-\Lambda_\omega) u_\omega+\kappa\sqrt{\omega}|u_\omega|^{2}u_\omega.\quad}
\end{equation}
The Cauchy problem for the equation \eqref{NLS} is locally well-posed in the energy space $\Sigma=\{u\in H^1(\mathbb{R}^3): (1+|x|^2)^{1/2} u\in L^2(\mathbb{R}^3)\}$ equipped with the weighted Sobolev norm
$$\|u\|_{\Sigma}:=\left\{\|u\|_{L^2(\mathbb{R}^3)}^2+\|\nabla u\|_{L^2(\mathbb{R}^3)}^2+\|xu\|_{L^2(\mathbb{R}^3)}^2\right\}^{1/2}.$$
Moreover, solutions preserve the mass
\begin{equation}\label{eq: mass}
M[u]=\int_{\mathbb{R}^3}|u|^2dx
\end{equation}
and the energy 
\begin{equation}\label{eq: energy}
E_\omega[u]= \frac{\omega}{2}\big\langle ( H_\omega-\Lambda_\omega) u, u\big\rangle_{L^2(\mathbb{R}^3)}+\frac{\kappa\sqrt{\omega}}{4}\int_{\R^3}|u|^{4}dx.
\end{equation}

\begin{remark}
By the phase shift (see \eqref{ansatz}), the lowest eigenvalue $\Lambda_\omega$ is subtracted in the definition of the energy \eqref{eq: energy}. Otherwise, even the energy of the lowest eigenfunction for the Schr\"odinger operator $H_\omega$ blows up as $\omega\to \infty$.
\end{remark}

The structure of a toroidal trap naturally leads us to decompose the Schr\"odinger operator into the radial-axial part and the angular part in cylindrical coordinates,
$$H_\omega=H_\omega^\perp-\frac{1}{r^2}\partial_\theta^2,$$
where 
\begin{equation}\label{H perp}
H_\omega^\perp:=-\partial_r^2-\frac{1}{r}\partial_r-\partial_z^2+U_\omega(r-\sqrt{\omega})+z^2
\end{equation}
acts on the partially radial class $L_{rad}^2(\mathbb{R}^2)\times L^2(\mathbb{R})$. We note that by the assumptions, the lowest eigenfunction for  $H_\omega^\perp$ is localized on a large ring $r=|y|=\sqrt{\omega}$ so that $-\frac{1}{r}\partial_r$ vanish effectively as $\omega\to\infty$. As a consequence, the lowest eigenvalue $\Lambda_\omega$ and the corresponding eigenfunction are approximated by $2$ and $\omega^{-\frac14}\Phi_0(|y|-\sqrt{\omega},z)$ respectively, where $\Phi_0=\frac{1}{\sqrt{\pi}}e^{-\frac{s^2+z^2}{2}}$ is the normalized lowest eigenfunction for the 2D hermite operator (see Section \ref{sec: linear analysis}). Therefore, we may expect that as the potential is strengthened, small energy solutions to the 3D NLS \eqref{NLS} get concentrated on $\omega^{-\frac14}\Phi_0(|y|-\sqrt{\omega},z)$. As a consequence, two degrees of freedom (for $|y|$ and $z$) are frozen and only one degree of freedom (for $\theta$) remains.

The purpose of this paper is to justify the emergence of the 1D periodic motion in the aforementioned physical experiment by answering the two problems. 

\begin{enumerate}
\item Find a large class of global-in-time solutions to the NLS \eqref{NLS}, including 
all solutions near the ring soliton in (2). Then, we show that they are concentrated on the ring, and their angular dynamics can be approximately described by the time-dependent 1D periodic NLS.
\item Construct an orbitally stable ring soliton concentrated on the ring whose angular profile is a ground state for the 1D periodic NLS.
\end{enumerate}
Both (1) and (2) concern the dimension reduction to the 1D periodic NLS. The dimension reduction by anisotropic harmonic potentials is verified theoretically in various settings, for example, for cigar-shaped and disc-shaped condensates \cite{Bossmann 1d, Bossmann 2d, Bossmann-Teufel, CH 2d 2013, CH, CH 2d, HT, HJ2}. To authors' knowledge, the dimension reduction for the toroidal trap model (both time-dependent and independent cases) and construction of a stable ring soliton have not been studied in the literature. 

\subsection{Dimension reduction for the time-dependent equation}

The first part of this paper is devoted to the dimension reduction for a class of global-in-time solutions to the \textit{time-dependent} rescaled NLS \eqref{NLS}. Indeed, in the defocusing case $\kappa=1$, the conservation laws yield global existence of all solutions in the energy space (see Proposition \ref{prop: 3D tNLS}). On the other hand, in the focusing case $\kappa=-1$, the equation has blow-up solutions (see \cite{C} for instance). However, when the trap is strong enough ($\omega\gg 1$), we have a simple criteria for global existence, that is,
$$\textbf{\textup{(H3)}}\quad \big\langle (H_\omega^\perp-\Lambda_\omega)u_{\omega,0},u_{\omega,0}\big\rangle_{L^2(\R^3)}\leq\delta\omega$$
for some sufficiently small $\delta>0$ (see Proposition \ref{global existence; focusing case}). 

\begin{remark}
The assumption \textbf{\textup{(H3)}} says that the portion of high energy states for $H_\omega^\perp$ in the initial data $u_{\omega,0}$ is not too large. This restriction vanishes as $\omega\to\infty$.

\end{remark}

Our first main result establishes the emergence of the effective 1D periodic dynamics for such global solutions. It is stated in two parts. First, we show that a 3D global NLS flow $u_\omega(t,x)$ can be approximated by a 1D periodic global flow $v_{\omega,\parallel}(t,\theta)$. Secondly, we prove that the 1D flow $v_{\omega,\parallel}(t,\theta)$ weakly converges to a solution to the periodic 1D cubic NLS
\begin{equation}\label{1d cubic tNLS}
i\partial_tw=-\partial_\theta^2w+\frac{\kappa}{2\pi}|w|^2w,
\end{equation}
where $w = w(t,\theta) : I(\subset \mathbb{R}) \times \mathbb{S}^1 \rightarrow \C$.

\begin{theorem}[Dimension reduction for the time-dependent 3D NLS \eqref{NLS}]\label{dimension reduction for tNLS}
Let $\kappa=\pm1$, and fix $m, E>0$. Suppose that \textbf{\textup{(H1)}} and \textbf{\textup{(H2)}} holds, and that initial data $u_{\omega,0}\in\Sigma$ satisfies
\begin{equation}\label{uniformly bounded energy assumption}
M[u_{\omega,0}]=m\quad\textup{and}\quad E_\omega[u_{\omega,0}]\leq E.
\end{equation}
In the focusing case $\kappa=-1$, we further assume that \textbf{\textup{(H3)}} holds for all sufficiently large $\omega\gg 1$ and for some small constant $\delta=\delta(m,E)>0$ independent of large $\omega\gg 1$ (see Corollary \ref{forbidden region}). Let $u_\omega(t)\in C([0,\infty); \Sigma)$ be the global solution to the 3D NLS \eqref{NLS} with initial data $u_{\omega,0}$. Then, the following hold.
\begin{enumerate}[$(i)$] 
\item (Dimension reduction to a 1D periodic flow)
There exists a global 1D flow $v_{\omega,\parallel}(t)\in C([0,\infty); H^1(\mathbb{S}^1))$ such that
\begin{equation}\label{t higher state strong convergence}
\sup_{t\geq0}\left\|u_\omega(t,x)-v_{\omega,\parallel}(t,\theta)\left(  \omega^{-\frac14}\Phi_0(|y|-\sqrt{\omega},z)\right)\right\|_{L^2(\mathbb{R}^3) }\lesssim\omega^{-\frac{1}{2}},
\end{equation}
where $\theta$ denotes the angle from the $y_1$-axis on the $y=(y_1,y_2)$-plane  and $\Phi_{0}(s,z)=\frac{1}{\sqrt{\pi}}e^{-\frac{s^2+z^2}{2}}$.
\item (Derivation of the 1D periodic NLS)
By the assumptions on the 3D initial data $u_{\omega,0}$, there exist $w_{\infty,0}\in H^1(\mathbb{S}^1)$ and a sequence $\{\omega_j\}_{j=1}^\infty$, with $\omega_j\to\infty$, such that 
\begin{equation}\label{tNLS initial data}
u_{\omega_j,0}(x)-w_{\infty,0}(\theta)\left(  \omega_j^{-\frac14}\Phi_0(|y|-\sqrt{\omega_j},z)\right)\rightharpoonup 0\quad\textup{in }H^1(\mathbb{R}^3)
\end{equation} 
(see Section \ref{sec: initial data} for the construction of $w_{\infty,0}$). Then, for any fixed $T>0$, we have the weak-$^*$ convergence 
$$w^*\textup{-}\lim_{j\to\infty} v_{\omega_j,\parallel}(t)=w_\infty(t)\quad\textup{in }L^\infty([0,T]; H^1(\mathbb{S}^1)),$$
where $w_\infty(t)\in C([0,\infty); H^1(\mathbb{S}^1))$ is the global solution to the 1D NLS \eqref{1d cubic tNLS} with initial data $w_{\infty,0}$.
\end{enumerate}
\end{theorem}

\begin{remark}
\begin{enumerate}[$(i)$]
\item The dimension reduction (Theorem \ref{dimension reduction for tNLS} $(i)$) is proved in strong sense. However, we are currently able to derive the 1D NLS only in weak sense due to a technical difficulty mentioned in Remark \ref{remark: why weak convergence?}.
\item In Theorem \ref{dimension reduction for tNLS} $(ii)$, the 1D NLS is derived sub-sequentially. That is just because the 1D initial data is prepared as a sub-sequential limit (see \eqref{tNLS initial data}) under the general assumptions on the 3D initial data (\textbf{\textup{(H3)}} and \eqref{uniformly bounded energy assumption}). Indeed, if we assume more that  
$$u_{\omega,0}(x)-w_{\infty,0}(\theta)\left( \omega^{-\frac14}\Phi_0(|y|-\sqrt{\omega},z)\right)\underset{\omega\to\infty}\rightharpoonup 0\quad\textup{in }H^1(\mathbb{R}^3),$$
for instance, if $u_{\omega,0}$ is factorized as $w_{\infty,0}(\theta)(  \omega^{-\frac14}\Phi_0(|y|-\sqrt{\omega},z))$  for some fixed profile $w_\infty\in H^1(\mathbb{S}^1)$, then we have
$$v_{\omega,\parallel}(t)\underset{\omega\to\infty}{\rightharpoonup^*} w_\infty(t)\quad\textup{in }L^\infty([0,T]; H^1(\mathbb{S}^1)).$$
\end{enumerate}
\end{remark}

\subsection{Construction of constrained energy minimizers and their dimension reduction}
In the second part of the paper, we aim to construct stable low energy steady states and to prove their dimension reduction. Before presenting our second main result, we note that in the defocusing case $\kappa=1$, a ground state under a mass constraint can be constructed by standard calculus of variation techniques, and its key properties,  such as uniqueness, orbital stability and dimension reduction, can be proved relatively easily. Thus, for readers' convenience, we put aside the discussion on the defocusing case in Appendix \ref{sec: defocusing}. From now on, we only consider the focusing case $\kappa=-1$.

On the other hand, in the focusing case, a mass constraint energy minimization problem is not properly formulated, because the 3D cubic NLS is mass-supercritical and the energy is not bounded from below. Therefore, motivated by earlier works \cite{BBJV, HJ1, HJ2} on dimension reduction in different settings, we instead consider the following \textit{constrained} energy minimization problem 
\begin{equation}\label{variational problem}
\boxed{\quad J_\omega^{(3D)}(m):=\min\Big\{E_\omega[u]: \  u\in\Sigma,\ M[u]=m, \langle (H_\omega^\perp-\Lambda_\omega)u,u\rangle_{L^2(\R^3)}\leq\delta\omega \Big\}\quad}
\end{equation}
where the mass and the energy are given by \eqref{eq: mass} and \eqref{eq: energy} with $\kappa=-1$ respectively, and $\delta>0$ is a sufficiently small number. Note that we here impose the constraint for the time-dependent problem (see \textup{\textbf{(H3)}}). Similarly as before, since we consider sufficiently large $\omega\geq 1$ but $\delta>0$ will be chosen independently of large $\omega\geq 1$, this additional constraint restricts the portion of high energy states with respect to $H_\omega^\perp$ but the restriction gets weaker as $\omega\to\infty$.

For this modified energy minimization problem, we construct a minimizer.

\begin{theorem}[Existence of a constrained energy minimizer; focusing case]\label{thm: existence}
Let $\kappa=-1$ and let $\delta>0$ be a small number chosen in Corollary \ref{forbidden region}. We assume that $\omega\gg 1$ is sufficiently large. Then, for any minimizing sequence $\{u_n\}_{n=1}^\infty$ to the variational problem \eqref{variational problem}, there exist $\{\mathcal{O}_n\}_{n=1}^\infty\subset \textup{SO}(2)$ and $\{\gamma_n\}_{n=1}^\infty\subset \mathbb{R}$ such that passing to a subsequence, 
$$\lim_{n\to\infty}\|e^{i\gamma_n}u_n(\mathcal{O}_n y,z)-Q_\omega\|_{\Sigma}=0$$
for some positive minimizer $Q_\omega$. Moreover, $Q_\omega$ solves the nonlinear elliptic equation
$$\omega(H_\omega-\Lambda_\omega) Q_\omega-\sqrt{\omega}Q_\omega^3=-\mu_\omega Q_\omega$$
with a Lagrange multiplier $\mu_\omega\in \R$.
\end{theorem}

For the constrained minimizer in Theorem \ref{thm: existence}, we are concerned with its dimension reduction and orbital stability. Indeed, analogously to the time-dependent model, the 3D problem \eqref{variational problem} is closely related to the 1D energy minimization problem 
\begin{equation}\label{circle minimization}
J_\infty^{(1D)}(m)=\inf \left\{E_\infty(w): w\in H^1(\mathbb{S}^1)\textup{ and }\|w\|_{L^2(\mathbb{S}^1)}^2=m\right\},
\end{equation}
where 
$$E_\infty[w]= \frac{1}{2}\|\partial_\theta w\|_{L^2(\mathbb{S}^1)}^2-\frac{1}{8\pi}\|w\|_{L^4(\mathbb{S}^1)}^4.$$
We recall from \cite{GLT} that the variational problem $J_\infty^{(1D)}(m)$ occupies a positive minimizer $Q_\infty$ with $Q_\infty(0)=\max_{\theta\in\mathbb{S}^1} Q_\infty(\theta)$, solving the Euler-Lagrange equation
$$-\partial_\theta^2Q_\infty-\frac{1}{2\pi}Q_\infty^3=-\mu_\infty Q_\infty.$$
Moreover, $Q_\infty$ is a unique minimizer up to phase shift and spatial translation. Indeed, $Q_\infty$ is given by a dnoidal function when $m>2\pi^2$ while it is constant when $0<m\le 2\pi^2$ (see Proposition \ref{prop: circle ground state}). More detailed properties of the ground state is provided in Appendix \ref{sec: Energy minimization for the cubic NLS on the unit circle}. 

The next theorem establishes the dimension reduction from the 3D to the 1D minimizer. 
\begin{theorem}[Dimension reduction for a constrained minimizer]\label{drth}
Under the assumption in Theorem \ref{thm: existence}, let $Q_\omega$ be the positive minimizer for the constrained problem $J_\omega^{(3D)}(m)$. Then, there exists $\{\mathcal{O}_\omega\}_{\omega\gg1}\subset \textup{SO}(2)$ such that  
$$\left\|Q_\omega(\mathcal{O}_\omega y,z)-Q_\infty(\theta)\left(\chi(|y|)\omega^{-\frac14}\Phi_0(|y|-\sqrt{\omega},z)\right) \right\|_{\Sigma}+|\mu_\omega-\mu_\infty|\to0\quad\textup{as }\omega\to\infty,$$
where $\Phi_{0}(s,z)=\frac{1}{\sqrt{\pi}}e^{-\frac{s^2+z^2}{2}}$  and $\chi:[0,\infty)\to[0,1]$ is a smooth function such that $\chi\equiv0$ on $[0,1]$ and $\chi\equiv1$ on $[2,\infty)$.
\end{theorem}

By the variational characterization together with global existence for the time-dependent NLS \eqref{NLS} (Proposition \ref{global existence; focusing case}), it immediately follows that the set of 3D minimizers in Theorem \ref{thm: existence} is orbitally stable  (see \cite{CL}). However, for the orbital stability of a minimizer by itself, the possibility of transforming one minimizer to another should be eliminated.

Our last main result asserts that if the confinement is strong enough, this can be done by proving uniqueness up to symmetries.

\begin{theorem}[Uniqueness and orbital stability of a constrained minimizer]\label{thm: unique}
Let $\omega\gg 1$ be sufficiently large, and suppose that $m\neq 2\pi^2$. Under the assumption in Theorem \ref{thm: existence}, let $Q_\omega$ be the minimizer for the constrained problem $J_\omega^{(3D)}(m)$. 
\begin{enumerate}[$(i)$]
\item (Uniqueness) The minimizer $Q_\omega$ is unique up to a rotation on the plane and phase shift. In other words, any minimizer for the problem $J_\omega^{(3D)}(m)$ can be expressed as $e^{i\gamma}Q_\omega(\mathcal{O} y,z)$ for some $\mathcal{O}\in\textup{SO}(2)$ and  $\gamma\in\mathbb{S}^1$.  
\item (Orbital stability) For any $\epsilon>0$, there exists $\delta>0$ such that if $\|u_0- Q_\omega\|_{\Sigma}<\delta$, then the global solution $u_\omega(t)\in C_t(\R,\Sigma)$ to 3D NLS \eqref{NLS} with the initial data $u_0$ satisfies
$$\inf_{\mathcal{O}\in \textup{SO}(2), \gamma\in\mathbb{S}^1}\|u_\omega(t,x) -e^{i\gamma}Q_\omega(\mathcal{O}y,z) \|_{\Sigma}<\epsilon \quad\textup{for all }t\geq 0.$$
\end{enumerate}
\end{theorem}

\begin{remark}
The assumption $m\neq 2\pi^2$ is from that we do not know the desired coercivity estimate for the linearized operator at the 1D periodic ground state $Q_\infty$ (see Remark \ref{1D coercivity; easy case}).
\end{remark}

\subsection{Ideas of the proofs, and the outline of the paper}

Since the dimension reduction by a toroidal trap model has been less explored, some considerable efforts are taken into developing analysis tools from the basic ones (Section \ref{sec: reformulation}-\ref{subsec: modified GN inequality}). First of all, due to the geometry of the trapping potential, it is natural to employ cylindrical coordinates and translate in the axial direction so that localization is achieved at the origin in the new coordinates. In Section \ref{sec: reformulation}, we introduce the reformulated problem and function spaces accordingly. 

Next, we note that in our setting, the lowest eigenstate for the operator $H_\omega^\perp$ has a different $\omega\to\infty$ behavior from other higher eigenstates (see the energy functional \eqref{eq: energy} and its reformulation \eqref{modified energy'}). Thus, for the proof, we need detailed properties for the lowest eigenstate and the spectral projection on it (as well as its orthogonal complement). An important technical remark is that when we take the projection to the lowest eigenstate, we need to truncate out near the $z$-axis (or $s=-\sqrt{\omega}$ in the new coordinates), because the projected state may have infinite energy (see Remark \ref{why truncation?}). Fortunately, it turns out that this truncation is not very harmful, because the lowest eigenstate enjoys Gaussian decay. In Section \ref{sec: linear analysis}, we provide the properties of the eigenstates including its decay and convergence to the lowest eigenstate to the hermite operator. In Section \ref{subsec: truncated projection}, we introduce the truncated projection operator, and present its useful various mapping properties. 

In Section \ref{subsec: modified GN inequality}, based on the previous two sections, we introduce our key analysis tool, that is, the refined Gagliardo-Nirenberg inequality (Proposition \ref{GN} and Corollary \ref{GN1}); by \eqref{u-v relation}, it is  equivalent to 
\begin{equation}\label{refined GN R^3}
\begin{aligned}
\|u\|_{L^4(\mathbb{R}^3)}^4&\lesssim\|u\|_{L^2(\mathbb{R}^3)}\Big\|\sqrt{H_\omega^\perp-\Lambda_\omega} u\Big\|_{L^2(\mathbb{R}^3)}^2\Big\{\big\|\sqrt{H_\omega^\perp-\Lambda_\omega} u\big\|_{L^2(\mathbb{R}^3)}^2+\|\partial_\theta u\|_{L^2(\mathbb{R}^3)}^2\Big\}^{\frac{1}{2}}\\
&\quad+\|u\|_{L^2(\mathbb{R}^3)}^3\|\partial_\theta u\|_{L^2(\mathbb{R}^3)}+\frac{1}{\sqrt{\omega}}\|u\|_{L^2(\mathbb{R}^3)}^4.
\end{aligned}
\end{equation}
The inequality \eqref{refined GN R^3} is improved in that compared to the standard inequality $\|u\|_{L^4(\mathbb{R}^3)}^4\lesssim \|u\|_{L^2(\mathbb{R}^3)}\|\nabla u\|_{L^2(\mathbb{R}^3)}^3$,  the degree of the angular derivative term $\|\partial_\theta u\|_{L^2(\mathbb{R}^3)}$ is reduced in the upper bound. Therefore, under the assumptions \eqref{uniformly bounded energy assumption} and \textbf{\textbf{(H3)}}, a sub-quadratic bound
$$\|u\|_{L^4(\mathbb{R}^3)}^4\lesssim \sqrt{\delta m\omega}\big\|\sqrt{H_\omega-\Lambda_\omega}u\big\|_{L^2(\mathbb{R}^3)}^2+m^{\frac{3}{2}}\|\partial_\theta u\|_{L^2(\mathbb{R}^3)}+\frac{m^2}{\sqrt{\omega}}$$
becomes available. By the inequality, the sub-critical nature can be captured under the constraint \textbf{\textbf{(H3)}} for the super-critical problem. Note also that the inequality immediately implies concentration to the lowest eigenstate for $H_\omega^\perp$ localized on a ring (see Corollary \ref{forbidden region} for the forbidden region in the function space due to \textbf{\textbf{(H3)}}). Indeed, similar refined inequalities and their consequences have been employed in different settings \cite{HJ1, HJ2}.

After preparing tools, in Section \ref{global}, we give a proof of the first main result (Theorem \ref{dimension reduction for tNLS}). As mentioned above, as an almost direct consequence of the refined Gagliardo-Nirenberg inequality, we prove global existence for solutions considered in the theorem and the dimension reduction (Theorem \ref{dimension reduction for tNLS} $(i)$), and obtain a uniform-in-$\omega$ bounds for global solutions. Then, using the uniform bounds, we derive the time-dependent 1D periodic NLS (Theorem \ref{dimension reduction for tNLS} $(ii)$). 

The last two sections (Section \ref{sec: existence} and \ref{sec: uniqueness}) are devoted to the constrained minimization problem \eqref{variational problem}. In Section \ref{sec: existence}, we construct a minimizer (Theorem \ref{thm: existence}) and prove its dimension reduction (Theorem \ref{drth}). As mentioned above, the energy minimization problem is super-critical without the constraint $\langle (H_\omega^\perp-\Lambda_\omega)u,u\rangle_{L^2(\R^3)}\leq\delta\omega$. However, with the additional constraint, it has sub-critical nature can be captured by the refined Gagliardo-Nirenberg inequality \eqref{refined GN R^3} so that a minimizer is constructed by concentration-compactness principle. Moreover, we can show that higher eigenstates of an energy minimizer vanish. Consequently, together with the energy convergence, the dimension reduction limit to the 1D periodic ground state (Theorem \ref{drth}) follows. Finally, in Section \ref{sec: uniqueness}, we establish the local uniqueness of a minimizer (Theorem \ref{thm: unique}). In the proof, an important step is to obtain a coercivity property of a  linearized operator at a 3D minimizer (Proposition \ref{coc1}). That can be done by transferring the coercivity of the linearized operator at the 1D periodic ground state via the dimension reduction limit. 

In Appendix \ref{sec: Energy minimization for the cubic NLS on the unit circle}, we review the properties of the 1D periodic ground state, and give a  proof of the coercivity of its linearized operator, which is a key ingredient for the coercivity of the linearized operator at a 3D minimizer. In Appendix \ref{sec: defocusing}, we present the analogous results in the defocusing case.

\subsection{Notations}\label{sec: notation}
We denote $A\lesssim B$ (resp., $A\gtrsim B$) ignoring the implicit constant $C>0$ in the inequality $A\leq CB$ (resp., $A\geq CB$) unless there is a confusion. In the same context, $\frac{1}{C}B\leq A\leq CB$ is expressed as $A\sim B$. In this article, we employ various Lebesgue spaces of the form $L^p(S,g(x) dx)$ with $S\subset\mathbb{R}^d$ and weighted measure $g(x)dx$ for some non-negative function $g$. If the domain and the variables for the given function space are clearly given in the context, $L^p(S,g(x) dx)$ is abbreviated to $L^p(g dx)$.

\subsection{Acknowledgement}
This work was supported by the National Research Foundation of Korea (NRF) grant funded by the Korea government (MSIT) (NRF-2020R1A2C4002615).

\section{Reformulation of the problem}\label{sec: reformulation}

For a clearer picture for the dimension reduction, we derive a reformulated problem making suitable changes of variables via the cylindrical coordinates.
\subsection{Modified mass and energy}

Since we are concerned with wave functions localized on a ring of radius $\sqrt{\omega}$, we convert everything into the cylindrical coordinates and then translate in the axial distance axis, introducing the $(s,z,\theta)$-coordinates given by
$$(x,y,z)=\left((s+\sqrt{\omega})\cos\theta, (s+\sqrt{\omega})\sin\theta, z\right): [-\sqrt{\omega},\infty)\times\mathbb{R}\times \mathbb{S}^1\to\mathbb{R}^3.$$
Note that the $(s,z,\theta)$-variable domain formally converges to $\mathbb{R}^2\times\mathbb{S}^1$ as $\omega\to\infty$. In the new coordinates, the 3D toroidal potential $U_\omega(|y|-\sqrt{\omega})+z^2$ becomes effectively a 2D harmonic potential $s^2+z^2$.

Accordingly, we change the variables as
\begin{equation}\label{u-v relation}
\boxed{\quad v(s,z,\theta)=\omega^{\frac{1}{4}}u(s+\sqrt{\omega},z, \theta),\quad |y|=s+\sqrt{\omega},\quad}
\end{equation}
and we define the modified mass
\begin{equation}\label{modified mass}
\mathcal{M}_\omega[v]:= \int_{\mathbb{S}^1}\int_{-\infty}^\infty\int_{-\sqrt{\omega}}^\infty |v(s,z,\theta)|^2\sigma_\omega(s)dsdzd\theta
\end{equation}
and the modified energy
\begin{equation}\label{modified energy}
\begin{aligned}
\mathcal{E}_\omega[v]&:=\frac{\omega}{2}\int_{\mathbb{S}^1}\int_{-\infty}^\infty\int_{-\sqrt{\omega}}^\infty \Big\{|\nabla_{(s,z)}v|^2+\big(U_\omega(s)+z^2\big)|v|^2-\Lambda_\omega|v|^2\Big\} \sigma_\omega(s)dsdzd\theta \\
&\quad +\frac{1}{2}\int_{\mathbb{S}^1}\int_{-\infty}^\infty\int_{-\sqrt{\omega}}^\infty \frac{|\partial_\theta v|^2}{\sigma_\omega(s)}dsdzd\theta+\frac{\kappa}{4}\int_{\mathbb{S}^1}\int_{-\infty}^\infty\int_{-\sqrt{\omega}}^\infty |v|^{4}\sigma_\omega(s)dsdzd\theta
\end{aligned}
\end{equation}
with the weight function $\sigma_\omega:[-\sqrt{\omega},\infty)\to[0,\infty)$ is given by
\begin{equation}\label{weight}
\boxed{\quad\sigma_\omega=\sigma_\omega(s)=1+\frac{s}{\sqrt{\omega}}.\quad}
\end{equation}

\begin{remark}
The mass and the energy, given \eqref{eq: mass} and  \eqref{eq: energy},  remain same by the above modification and the transformation \eqref{u-v relation}; $M[u]=\mathcal{M}_\omega[v]$ and $E_\omega[u]=\mathcal{E}_\omega[v]$.
\end{remark}
\begin{remark}
Our reformulation requires to deal with functions on the rather uncommon domain $[-\sqrt{\omega},\infty)\times\mathbb{R}\times\mathbb{S}^1$ with the weight $\sigma_\omega(s)=1+\frac{s}{\sqrt{\omega}}$, but we may analyze the problem invoking that they formally converge to $\mathbb{R}^2\times\mathbb{S}^1$ and $1$, respectively.
\end{remark}

\subsection{Function spaces}
From now on, we minimize the modified energy under the modified constraints. Thus, it is convenient to employ the following function spaces which fit better into our reformulated problem.

For notational convenience, given an $s$-variable weight function $g=g(s):[-\sqrt{\omega},\infty)\to[0,\infty)$ and $1\leq p\leq \infty$, we abbreviate the weighted space $L^p\big([-\sqrt{\omega},\infty)\times\mathbb{R}\times \mathbb{S}^1, g(s)dsdzd\theta\big)$ to $L^p(g)$ equipped with the norm\footnote{With abuse of notation, we denote $L^\infty(g)\equiv L^\infty(1)$.}
\begin{equation}\label{3d weighted norm}
\|v\|_{L^p(g)}=\left\{\begin{aligned}
&\bigg\{\int_{\mathbb{S}^1}\int_{-\infty}^\infty\int_{-\sqrt{\omega}}^\infty |v(s,z,\theta)|^pg(s)dsdzd\theta\bigg\}^{\frac{1}{p}}&&\textup{if } 1\le p<\infty,\\
& \esssup_{(s,z,\theta)\in[-\sqrt{\omega},\infty)\times\mathbb{R}\times \mathbb{S}^1} |v(s,z,\theta)| &&\textup{if }p=\infty.
\end{aligned}\right.
\end{equation}
In particular, $L^2(g)$ is the Hilbert space with the inner product 
$$\langle v_1,v_2\rangle_{L^2(g)} =\int_{\mathbb{S}^1}\int_{-\infty}^\infty\int_{-\sqrt{\omega}}^\infty v_1(s,z,\theta)\overline{v_2(s,z,\theta)} g(s)dsdzd\theta.$$
To be specific, we will mostly employ the two weight functions $\sigma_\omega(s)$ and $\frac{1}{\sigma_\omega(s)}$. Note that both $L^p(\sigma_\omega)$ and $L^p(\frac{1}{\sigma_\omega})$ formally converge to $L^p(\mathbb{R}^2\times\mathbb{S}^1)$, since $\sigma_\omega(s)=1+\frac{s}{\sqrt{\omega}}\to 1$.

We define the differential operator
\begin{equation}\label{2d differential operator}
\mathcal{H}_{\omega}^{(2D)}:=-\partial_s^2-\frac{1}{\sqrt{\omega}\sigma_\omega}\partial_s  -\partial_z^2 +U_\omega(s)+z^2
\end{equation}
as a quadratic form acting on the weighted space $L^2(\sigma_\omega dsdz)$. The spectral properties of the operator will be presented in Section \ref{sec: linear analysis}. Indeed, it will be shown that the lowest eigenvalue $\Lambda_\omega$ is simple and it converges to $2$ as $\omega\to\infty$ (see Corollary \ref{2d eigenstate}).

Referring to the first term in the modified energy \eqref{modified energy} as well as the additional constraint, we introduce the semi-norm\footnote{By integration by parts, $$\|v\|_{{\dot{\Sigma}}_{\omega;(s,z)}}^2=\int_{\mathbb{S}^1}\int_{-\infty}^\infty\int_{-\sqrt{\omega}}^\infty\Big\{|\nabla_{(s,z)}v|^2+\big(U_\omega(s)+z^2\big)|v|^2-\Lambda_\omega|v|^2\Big\} \sigma_\omega(s)dsdzd\theta,$$
provided that the both sides are finite.}
\begin{equation}\label{seminorm}
\|v\|_{{\dot{\Sigma}}_{\omega;(s,z)}}:=\big\langle (\mathcal{H}_\omega^{(2D)}-\Lambda_\omega)v,v\big\rangle_{L^2(\sigma_\omega)}^{\frac{1}{2}},
\end{equation}
and define 
\begin{equation}\label{seminorm2}
\|v\|_{{\dot{\Sigma}}_{\omega}}:=\left\{\|v\|_{{\dot{\Sigma}}_{\omega;(s,z)}}^2+\frac{1}{\omega}\|\partial_\theta v\|_{L^2(\frac{1}{\sigma_\omega})}^2\right\}^{\frac{1}{2}}.
\end{equation}
Finally, collecting all quadratic terms in the modified mass and the energy, we denote by $\Sigma_\omega$ the Hilbert space with the norm\footnote{By \eqref{u-v relation},  $$ \|v\|_{ \Sigma_\omega}^2=\int_{\R^3}|\nabla u|^2+\big(U_\omega(|y|-\sqrt{\omega})+z^2+1\big)|u|^2dx.$$ } 
\begin{equation}\label{norm}
\|v\|_{ \Sigma_\omega}:=\left\{(1+\Lambda_\omega)\|v\|_{L^2(\sigma_\omega)}^2+\|v\|_{{\dot{\Sigma}_\omega}}^2\right\}^{\frac{1}{2}}.
\end{equation}

\subsection{Reformulated variational problem}
By the definitions of the function spaces and the norms introduced in the previous subsection, the modified mass and the modified energy can be expressed in a compact form as
\begin{equation}\label{modified mass'}
\mathcal{M}_\omega[v]=\|v\|_{L^2(\sigma_\omega)}^2
\end{equation}
and
\begin{equation}\label{modified energy'}
\mathcal{E}_\omega[v]=\frac{\omega}{2}\|v\|_{{\dot{\Sigma}}_{\omega}}^2+\frac{\kappa}{4}\|v\|_{L^4(\sigma_\omega)}^4=\frac{\omega}{2}\|v\|_{{\dot{\Sigma}}_{\omega; (s,z)}}^2+\frac{1}{2}\|\partial_\theta v\|_{L^2(\frac{1}{\sigma_\omega})}^2+\frac{\kappa}{4}\|v\|_{L^4(\sigma_\omega)}^4.
\end{equation}
Consequently, by \eqref{u-v relation}, the energy minimization problem \eqref{variational problem} is rephrased as
\begin{equation}\label{modified variational problem}
\boxed{\quad\mathcal{J}_\omega^{(3D)}(m):=\min\Big\{\mathcal{E}_\omega[v]: \  v\in\Sigma_\omega,\ \mathcal{M}_\omega[v]=m\textup{ and } \|v\|_{{\dot{\Sigma}}_{\omega;(s,z)}}\leq\delta\sqrt{\omega} \Big\}.\quad}
\end{equation}

\begin{remark}\label{formal derivation}
The minimization problem \eqref{modified variational problem} is in essence nothing but a reformulation of the original problem \eqref{variational problem}, but it has many advantages to observe the strong confinement effect and to catch a sub-critical nature of the super-critical problem. Indeed, the energy $\mathcal{E}_\omega(v)$ is formally approximated by
$$\frac{\omega}{2}\left\langle(-\Delta_{(s,z)}+s^2+z^2-2)v,v\right\rangle_{L^2(\mathbb{R}^2\times\mathbb{S}^1)}+\frac{1}{2}\|\partial_\theta v\|_{L^2(\mathbb{R}^2\times\mathbb{S}^1)}^2+\frac{\kappa}{4}\|v\|_{L^4(\mathbb{R}^2\times\mathbb{S}^1)}^4.$$
Thus, an energy minimizer is expected to be concentrated at a factorized state $\Phi_\infty(s,z)w(\theta)$, where $\Phi_\infty(s,z)=\frac{1}{\sqrt{\pi}}e^{-\frac{s^2+z^2}{2}}$ is the lowest eigenstate for the 2d hermite operator $-\Delta_{(s,z)}+s^2+z^2$ (with eigenvalue 2). Subsequently, the minimum energy is further reduced to 
$$\mathcal{E}_\omega[w(\theta)\Phi_\infty(s,z)]\underset{\omega\to\infty}\approx  \frac{1}{2}\int_{\mathbb{S}^1}|\partial_\theta w|^2 +\frac{\kappa}{8\pi}\int_{\mathbb{S}^1} |w|^{4}d\theta=E_\infty[w].$$
On the other hand, the additional constraint $\|v\|_{{\dot{\Sigma}}_{\omega;(s,z)}}\leq\delta\sqrt{\omega}$ gets weaker. Therefore, a ground state on the unit circle is derived in the strong confinement regime. 
\end{remark}

Accordingly, the main results in the second part of this paper are reworded as follows.

\begin{theorem}[Existence of a constrained energy minimizer; reformulated]\label{reformulated main theorem}
Let $\kappa=-1$ and let $\delta>0$ be a small number chosen in Corollary \ref{forbidden region}. We assume that $\omega\gg 1$ is sufficiently large. For any minimizing sequence $\{v_n\}_{n=1}^\infty$ of the variational problem \eqref{modified variational problem}, there exists $\{\gamma_n\}_{n=1}^\infty\subset\mathbb{R}$ and $\{\theta_n\}_{n=1}^\infty\subset\mathbb{S}^1$ such that passing to a subsequence, 
$$e^{i\gamma_n}v_n(s,z,\theta-\theta_n)\to \mathcal{Q}_\omega>0 \textup{ in } \Sigma_\omega$$
and $\mathcal{Q}_\omega$ solves 
$$\omega( \mathcal{H}_\omega^{(2D)}-\Lambda_\omega)\mathcal{Q}_\omega-\frac{1}{\sigma_\omega^2}\partial_\theta^2 \mathcal{Q}_\omega- \mathcal{Q}_\omega^3=- \mu_\omega \mathcal{Q}_\omega.$$
\end{theorem}

\begin{theorem}[Dimension reduction; reformulated]\label{thm: dimension reduction, reformulated}
A minimizer $\mathcal{Q}_\omega$ constructed in Theorem \ref{reformulated main theorem} satisfies 
$$
\lim_{\omega\to \infty}\|\mathcal{Q}_\omega(s,z,\theta)-\mathcal{Q}_\infty(\theta) \chi_\omega(s) \Phi_0(s,z) \|_{\Sigma_\omega}=0,
$$
where $\Phi_{0}(s,z)=\frac{1}{\sqrt{\pi}}e^{-\frac{s^2+z^2}{2}}$  and $\chi_\omega=\chi(\cdot+\sqrt{\omega})$ for a smooth cut-off $\chi:[0,\infty)\to[0,1]$ such that $\chi\equiv0$ on $[0,1]$ and $\chi\equiv1$ on $[2,\infty)$.
\end{theorem}

\begin{theorem}[Uniqueness; reformulated]
Let $\omega\gg 1$ be sufficiently large, and suppose that $m\neq 2\pi$. For the problem \eqref{modified variational problem}, every minimizer is of the form $e^{i\gamma}\mathcal{Q}_\omega(s,z,\theta-\theta_0)$ for some $\gamma\in\mathbb{R}$ and $\theta_*\in\mathbb{S}^1$, where $\mathcal{Q}_\omega$ is a positive minimizer obtained in Theorem \ref{reformulated main theorem}.
\end{theorem}

\section{Spectral properties of the Schr\"odinger operator $\mathcal{H}_{\omega}^{(2D)}$}\label{sec: linear analysis}

Before entering into our main nonlinear problem, we study  spectral properties of the differential operator
\begin{equation}\label{1d differential operator}
\mathcal{H}_{\omega}^{(1D)}=-\partial_s^2-\frac{1}{\sqrt{\omega}\sigma_\omega}\partial_s+U_\omega(s)\quad\left(\textup{resp., }\mathcal{H}_{\omega}^{(2D)}=\mathcal{H}_{\omega}^{(1D)}-\partial_z^2+z^2\right)
\end{equation}
acting on the Hilbert space 
$$L^2(\sigma_\omega ds)=L^2([-\sqrt{\omega},\infty),\sigma_\omega ds) \ (\textup{resp.,} L^2(\sigma_\omega dsdz)=L^2([-\sqrt{\omega},\infty)\times\mathbb{R},\sigma_\omega dsdz)).
$$

\begin{remark}\label{linear differential operator remarks}
\begin{enumerate}[$(i)$]
\item The properties of the 2D operator immediately follows from those of the 1D operator, because the 2D operator is separated as $\mathcal{H}_{\omega}^{(2D)}=\mathcal{H}_{\omega}^{(1D)}\otimes I_{L^2(\mathbb{R}, dz)}+I_{L^2(\sigma_\omega ds)}\otimes\mathcal{H}_\infty^{(1D)}$ and the 1D hermite operator $\mathcal{H}_\infty^{(1D)}=-\partial_z^2+z^2$ is well-known.
\item By the relation \eqref{u-v relation}, the operator $\mathcal{H}_{\omega}^{(2D)}$ is equivalent to the Schr\"odinger operator $-\Delta_x+U_\omega(|y|-\sqrt{\omega})+z^2$ acting on the partially radial class $L_{rad}^2(\mathbb{R}_y^2)\times L^2(\mathbb{R}_z)$.
\item By $(i)$, $\mathcal{H}_\omega^{(2D)}$ has a positive normalized ground state $\Phi_\omega=\phi_\omega(s)\phi_\infty(z)$, where $\phi_\omega(s)$ (resp., $\phi_\infty(z)$) is the $L^2$-normalized ground state for $\mathcal{H}_{\omega}^{(1D)}$ (resp., $\mathcal{H}_\infty^{(1D)}$) corresponding to the lowest eigenvalue $\lambda_\omega$ (resp., $\lambda_\infty$). By $(ii)$, referring to the equivalent radially symmetric $\mathbb{R}^2$ problem (see \cite[Section XIII.11 and XIII.12]{RS4} for instance), we have that the normalized ground state $\phi_\omega$ is unique up to phase shift, and $\phi_\omega$ and $\phi_\omega'$ decrease exponentially. Moreover, $\phi_\omega(-\sqrt{\omega})>0$, which corresponds to the fact that the ground state for $-\Delta_y+U_\omega(|y|-\sqrt{\omega})$ in radial class is not zero at the origin. Our goal here is to obtain properties of eigenfunctions which hold uniformly in $\omega\gg1$.
\end{enumerate}
\end{remark}

The main result of this section provides the asymptotics of the lowest eigenvalue $\lambda_\omega$ and the  ground state $\phi_\omega$.

\begin{proposition}[1D ground state]\label{1d eigenstate}
Suppose that $U_\omega$ satisfies \textbf{\textup{(H1)}} and \textbf{\textup{(H2)}}. Let $\lambda_\omega<\lambda_\omega'$ be the two lowest eigenvalues for the Schr\"odinger operator $\mathcal{H}_\omega^{(1D)}$, and let $\phi_\omega$ be its unique positive $L^2(\sigma_\omega ds)$ normalized ground state. Then there exists $\alpha_0>0$ such that, for any $c\in(0,\alpha_0)$, there exists $\omega_c\geq 1$ such that if $\omega\geq \omega_c$, the following hold.
\begin{enumerate}[$(i)$]
\item (Lowest eigenvalue asymptotic)\footnote{By integration by parts, $$\langle \mathcal{H}_{\omega}^{(1D)}v,v\rangle_{L^2(\sigma_\omega ds)}= \int_{-\sqrt{\omega}}^\infty\Big\{|\partial_s v|^2+ U_\omega(s) |v|^2 \Big\} \sigma_\omega(s)dsdzd\theta.$$}
$$\lambda_\omega= \min \left\{ \mathcal{E}_\omega^{(1d)}(v):=\frac12\langle \mathcal{H}_{\omega}^{(1D)}v,v\rangle_{L^2(\sigma_\omega)}:\ \|v\|_{L^2(\sigma_\omega ds)}=1 \right\}=1+O(\omega^{-\frac{1}{2}}).$$
\item (Exponential decay)
\begin{equation}
\phi_\omega(s)+|\phi_\omega'(s)|\lesssim e^{-cs^2}\quad\textup{for all }s\geq-\sqrt{\omega}.\label{1d exponential decay}
\end{equation}
Moreover, we have
\begin{equation}\label{tue3}
\|(1-\chi_\omega)\phi_\omega\|_{L^2(\sigma_\omega ds)}\lesssim e^{-c\omega},\quad \|\chi_\omega\phi_\omega\|_{L^2(\frac{1}{\sigma_\omega}ds)}=1+O(\omega^{-\frac{1}{2}}),
\end{equation}
where $\chi_\omega$ is a cut-off such that $\chi_\omega(s)=0$ on $[-\sqrt{\omega},-\sqrt{\omega}+1]$ and $\chi_\omega(s)=1$ for $s\geq-\sqrt{\omega}+2$, precisely given in Theorem \ref{thm: dimension reduction, reformulated}.
\item (Convergence of the 1D ground state)
\begin{equation}\label{1d lowest energy convergence}
\int_{-\sqrt{\omega}}^\infty \Big\{(\phi_\omega-\phi_\infty)^2+ (\phi_\omega'-\phi_\infty')^2+ s^2(\phi_\omega-\phi_\infty)^2\Big\}ds=O(\omega^{-1})
\end{equation}
and
\begin{equation}\label{1d l4conv}
\int_{-\sqrt{\omega}}^\infty  (\phi_\omega-\phi_\infty)^4ds=O(\omega^{-2}),
\end{equation}
where $\phi_\infty(z)=\frac{1}{\pi^{1/4}}e^{-\frac{z^2}{2}}$.
\item (Spectral gap asymptotic)
$$\lambda_\omega'-\lambda_\omega=2+O(\omega^{-\frac{1}{2}}).$$
\end{enumerate}
\end{proposition}

By Remark \ref{linear differential operator remarks} $(i)$, the following properties of the 2D operator are immediately obtained. 

\begin{corollary}[2D ground state]\label{2d eigenstate}
Suppose that $U_\omega$ satisfies \textbf{\textup{(H1)}} and \textbf{\textup{(H2)}}.  Let $\Lambda_\omega<\Lambda_\omega'$ be the first two lowest eigenvalues for the Schr\"odinger operator $\mathcal{H}_\omega^{(2D)}$, and let $\Phi_\omega$ be its unique positive $L^2(\sigma_\omega ds)$ normalized ground state. Then, we have
\begin{equation}\label{secei}
\Phi_\omega(s,z)=\phi_\omega(s)\phi_\infty(z),\quad \Lambda_\omega=\lambda_\omega+1=2+O(\omega^{-\frac{1}{2}})\quad\textup{and}\quad\Lambda_\omega'=4+O(\omega^{-\frac{1}{2}}).
\end{equation}
\end{corollary}

\begin{proof}[Proof of Corollary \ref{2d eigenstate} assuming Proposition \ref{1d eigenstate}]
By the non-degeneracy of ground state solution for $\mathcal{H}_\omega^{(2D)}$ (see \cite[Theorem XIII.47]{RS4}), it is obvious that $\Lambda_\omega=\lambda_\omega+1$ and $\Phi_\omega(s,z)=\phi_\omega(s)\phi_\infty(z)$. Moreover, by \cite[Theorem XIII.47]{RS4} and separation of the variables (see Remark \ref{linear differential operator remarks}), we have $\Lambda_\omega'=\min\{\lambda_\omega'+1, \lambda_\omega+3\}\to 4$, because $3$ is the second lowest eigenvalue for $-\partial_z^2+z^2$.
\end{proof}

The proof of the proposition \ref{1d eigenstate} will be broken into several steps. First, we obtain an upper bound on the lowest energy level.

\begin{lemma}[Upper bound on the ground state energy $\lambda_\omega$]\label{1d energy upper bound}
In Proposition \ref{1d eigenstate}, we have
$$\lambda_\omega\leq 1+O(\omega^{-\frac{1}{2}}).$$
\end{lemma}

\begin{proof}
Let $\chi_\omega$ be a cut-off given in Theorem \ref{thm: dimension reduction, reformulated}. Then, direct calculation with $\sigma_\omega=1+\frac{s}{\sqrt{\omega}}$ yields $\|\chi_\omega\phi_\infty\|_{L^2(\sigma_\omega ds)}^2 =1+O(\omega^{-\frac{1}{2}})$. Similarly but by the assumptions on $U_\omega$, one can show that $\langle \mathcal{H}_\omega^{(1D)}(\chi_\omega\phi_\infty),\chi_\omega\phi_\infty\rangle_{L^2(\sigma_\omega ds)}=\langle(-\partial_s^2+s^2)\phi_\infty,\phi_\infty\rangle_{L^2(\mathbb{R})}+O(\omega^{-\frac{1}{2}})=1+O(\omega^{-\frac{1}{2}})$. Thus, we conclude that $\lambda_\omega\leq \mathcal{E}_\omega^{(1D)}(\frac{\chi_\omega\phi_\infty}{\|\chi_\omega\phi_\infty\|_{L^2(\sigma_\omega ds)}})=1+O(\omega^{-\frac{1}{2}})$.
\end{proof}

Next, following the argument of Agmon \cite{Agmon}, we prove the decay of the ground state.

\begin{proof}[Proof of Proposition \ref{1d eigenstate} $(ii)$]
We claim that there exists $\alpha_0>0$ such that for any $c\in(0,\alpha_0)$, if $\omega\geq 1$ is large enough, then
\begin{equation}\label{1d elliptic eq lemma claim}
\phi_\omega(s)\lesssim \sigma_{\omega}(s)^{-\frac{1}{2}}e^{-cs^2}\quad\textup{for all }s>-\sqrt{\omega}. 
\end{equation}
For the proof, we take $c\in(0,\alpha_0)$, and let $g$ be a bounded smooth function on $[-\sqrt{\omega},\infty)$ with $g(s)=1$ for large $s>1$, where $\alpha_0$ and $g$ will  be chosen later. Then, a direct calculation with $\mathcal{H}_\omega^{(1D)}\phi_\omega=\lambda_\omega\phi_\omega$   yields 
\begin{equation}\label{1d elliptic eq lemma claim proof}
\begin{aligned}
\big\langle (\mathcal{H}_\omega^{(1D)}-\lambda_\omega) (\phi_\omega g), \phi_\omega g\big\rangle_{L^2(\sigma_\omega ds)}&=\int_{-\sqrt{\omega}}^\infty(-2\phi_\omega\phi_\omega' gg'-\phi_\omega^2gg'')\sigma_\omega-\frac{1}{\sqrt{\omega}}\phi_\omega^2 gg' ds\\
&=\int_{-\sqrt{\omega}}^\infty (\phi_\omega g')^2\sigma_\omega ds.
\end{aligned}
\end{equation}
where in the second identity, we did integration by parts for $-\phi_\omega^2gg''\sigma_\omega$. Hence, we have 
$$0\geq-\int_{-\sqrt{\omega}}^\infty \big\{(\phi_\omega g)'\big\}^2\sigma_\omega ds=\int_{-\sqrt{\omega}}^\infty \big\{(U_\omega-\lambda_\omega)g^2-(g')^2\big\}\phi_\omega^2\sigma_\omega ds,$$
where we used integration by parts (with Remark \ref{linear differential operator remarks} $(iii)$).
For large $L\geq \sqrt{\omega}$, we take $g=e^{f_L}$, where $f_L=cs^2\eta(\frac{s}{L})$ and $\eta$ is a smooth cutoff such that $\eta$ is supported on $[-\frac{1}{2},\frac{1}{2}]$ and $\eta\equiv1$ on $[-\frac{1}{4},\frac{1}{4}]$. Then, it follows that
$$
0\ge-\|(e^{f_L}\phi_\omega )'\|_{L^2(\sigma_\omega ds)}^2=\int_{-\sqrt{\omega}}^\infty \big\{U_\omega-\lambda_\omega-(f_L')^2\big\}\big(e^{f_L}\phi_\omega\big)^2\sigma_\omega ds.
$$
Note that by the assumption \textbf{\textup{(H2)}}, $U_\omega$ has a quadratic lower bound. Thus, if $\alpha_0>0$ small and $\omega\geq 1$ is large enough, there exists $R\in(1,\sqrt{\omega})$, independent of $\omega\geq 1$  and $L\ge \sqrt{\omega}$, such that $U_\omega-\lambda_\omega-(f_L')^2\geq 1$ for all $|s|\geq R$, while there exists $C_{R}>0$, independent of $\omega\geq 1$ and $L\ge \sqrt{\omega}$, such that $|U_\omega-\lambda_\omega-(f_L')^2|e^{2f_L}\leq C_R^2<\infty$ for all $|s|\leq R$. Thus, it follows that 
$$\begin{aligned}
\|\mathbf{1}_{|s|\geq R}e^{f_L}\phi_\omega\|_{L^2(\sigma_\omega ds)}^2&\leq \int_{-\sqrt{\omega}}^\infty \mathbf{1}_{|s|\geq R} \left\{U_\omega-\lambda_\omega-(f_L')^2\right\}\big(e^{f_L}\phi_\omega\big)^2\sigma_\omega ds\\
&\leq -\int_{-\infty}^\infty \mathbf{1}_{|s|\leq R} \left\{U_\omega-\lambda_\omega-(f_L')^2\right\}\big(e^{f_L}\phi_\omega\big)^2\sigma_\omega ds\leq C_R^2.
\end{aligned}$$
Subsequently, taking $L\to\infty$, we obtain $\|\mathbf{1}_{|s|\geq R}e^{cs^2}\phi_\omega\|_{L^2(\sigma_\omega ds)}\leq C_R$. On the other hand, we have 
$\|\mathbf{1}_{|s|\leq R}e^{cs^2}\phi_\omega\|_{L^2(\sigma_\omega ds)}\leq e^{cR^2}\|\phi_\omega\|_{L^2(\sigma_\omega ds)}=e^{cR^2}$. Since $R$ does not depend on large $\omega\geq 1$, we prove that 
\begin{equation}\label{1d elliptic eq lemma claim proof 2}
\|e^{cs^2}\phi_\omega\|_{L^2(\sigma_\omega ds)}\lesssim1.
\end{equation}
On the other hand, inserting $g=e^{cs^2\eta(\frac{s}{L})}$ in \eqref{1d elliptic eq lemma claim proof}, one can show that $\|(e^{cs^2}\phi_\omega)'\|_{L^2(\sigma_\omega ds)}\lesssim1$, because by \eqref{1d elliptic eq lemma claim proof 2} with $c_0\in (c,\alpha_0)$, $\|\phi_\omega g'\|_{L^2(\sigma_\omega ds)}$ is bounded uniformly in $\omega\geq 1$  and $L\ge \sqrt{\omega}$. Therefore, the claim \eqref{1d elliptic eq lemma claim} follows from   \eqref{radial GN inequ}.

For \eqref{1d exponential decay}, we first estimate the derivative $\phi_\omega'(s)$. For $-\sqrt{\omega}\leq s\leq 0$, by the fundamental theorem of calculus and the equation $\mathcal{H}_\omega^{(1D)}\phi_\omega=\lambda_\omega\phi_\omega$, we have
$$\begin{aligned}
(\sigma_\omega \phi_\omega')(s)&=(\sigma_\omega \phi_\omega')(s)-(\sigma_\omega \phi_\omega')(-\sqrt{\omega})=\int_{-\sqrt{\omega}}^s (\sigma_\omega \phi_\omega')'ds_1\\
&=\int_{-\sqrt{\omega}}^s \Big(\sigma_\omega \phi_\omega''+\frac{1}{\sqrt{\omega}} \phi_\omega'\Big) ds_1=\int_{-\sqrt{\omega}}^s \sigma_\omega (U_\omega-\lambda_\omega)\phi_\omega ds_1.
\end{aligned}$$
Thus, applying \eqref{1d elliptic eq lemma claim} with $c_0\in(c,\alpha_0)$, by the assumption on $U_\omega$, we prove that 
$$\begin{aligned}
|(\sigma_\omega \phi_\omega')(s)|&\lesssim \sigma_\omega(s) \int_{-\sqrt{\omega}}^s  |U_\omega(s_1)-\lambda_\omega|\frac{1}{\sqrt{\sigma_{\omega}(s_1)}}e^{-c_0{s_1}^2} ds_1\lesssim \sigma_\omega(s)e^{-cs^2},
\end{aligned}$$
and consequently, $|\phi_\omega'(s)|\lesssim e^{-cs^2}$. For $s\geq0$, we write
$$(\sigma_\omega \phi_\omega')(s)=(\sigma_\omega \phi_\omega')(s)-(\sigma_\omega \phi_\omega')(\infty)=-\int_s^\infty (\sigma_\omega \phi_\omega')'ds_1.$$
Then, by the same way, we can prove the bound $|\phi_\omega'(s)|\lesssim e^{-cs^2}$. It remains to estimate $\phi_\omega(s)$. Indeed, for $s\geq -\frac{\sqrt{\omega}}{2}$, it is obvious that \eqref{1d elliptic eq lemma claim} with $c_0\in (c,\alpha_0)$ implies $\phi_\omega(s)\lesssim \sigma_{\omega}(-\frac{\sqrt{\omega}}{2})^{-\frac{1}{2}}e^{-cs^2}\sim e^{-cs^2}$. On the other hand, when $-\sqrt{\omega}+1\leq s\leq-\frac{\sqrt{\omega}}{2}$, we use \eqref{1d elliptic eq lemma claim} with $c_0\in (c,\alpha_0)$ to obtain $\phi_\omega(s)\lesssim \sigma_{\omega}(-\sqrt{\omega}+1)^{-\frac{1}{2}}e^{-c_0s^2}=\omega^{\frac{1}{4}}e^{-c_0s^2}\lesssim e ^{-cs^2}$. Finally, if $-\sqrt{\omega}\leq s\leq -\sqrt{\omega}+1$, by the fundamental theorem of calculus and the derivative estimate, we prove that 
$$\begin{aligned}
\phi_\omega(s)\leq \phi_\omega(-\sqrt{\omega}+1)+\int_{-\sqrt{\omega}}^{-\sqrt{\omega}+1} |\phi_\omega'(\tau)|d\tau\lesssim e^{-c\omega}\sim e ^{-cs^2}.
\end{aligned}$$

By fast decay \eqref{1d exponential decay}, the former inequality in \eqref{tue3} immediately follows. For the other one, we write $|\|\chi_\omega\phi_\omega\|_{L^2(\frac{1}{\sigma_\omega}ds)}-1|=|\|\frac{\chi_\omega}{\sigma_\omega}\phi_\omega\|_{L^2(\sigma_\omega ds)}-\|\phi_\omega\|_{L^2(\sigma_\omega ds)}|\leq \|(\frac{\chi_\omega}{\sigma_\omega}-1)\phi_\omega\|_{L^2(\sigma_\omega ds)}$. We note that 
$$\left\{\begin{aligned}
&|\tfrac{\chi_\omega}{\sigma_\omega}-1|=1 &&\textup{if }-\sqrt{\omega}\leq s\leq-\sqrt{\omega}+1,\\
&|\tfrac{\chi_\omega}{\sigma_\omega}-1|\leq  \tfrac{1}{\sigma_\omega}+1\leq  \tfrac{1}{\sigma_\omega(-\sqrt{\omega}+1)}+1\sim \sqrt{\omega} &&\textup{if }-\sqrt{\omega}+1\leq s\leq -\tfrac{\sqrt{\omega}}{2},\\
&|\tfrac{\chi_\omega}{\sigma_\omega}-1|= |\tfrac{1}{\sigma_\omega}-1|=\tfrac{|s|}{\sqrt{\omega}+s}\lesssim\tfrac{|s|}{\sqrt{\omega}}&&\textup{if }s\geq-\tfrac{\sqrt{\omega}}{2}.
\end{aligned}\right.$$
Thus, it follows from \eqref{1d exponential decay} that $\|(\frac{\chi_\omega}{\sigma_\omega}-1)\phi_\omega\|_{L^2(\sigma_\omega ds)}\lesssim\omega^{-\frac{1}{2}}$.
\end{proof}

\begin{proof}[Proof of Proposition \ref{1d eigenstate} $(i)$ and $(iii)$]
For $(i)$, by Lemma \ref{1d energy upper bound}, it suffices to show a lower bound. Let us consider $\chi_\omega\phi_\omega$ as its trivial extension to the real line. Then, repeating the argument to prove Lemma \ref{1d energy upper bound} but with Proposition \ref{1d eigenstate} $(ii)$, one can show that $\|\chi_\omega\phi_\omega\|_{L^2(\R)}^2=1+O(\omega^{-\frac{1}{2}})$ and $\langle(-\partial_s^2+s^2)(\chi_\omega\phi_\omega), \chi_\omega\phi_\omega\rangle_{L^2(\mathbb{R},ds)}=\lambda_\omega+O(\omega^{-\frac{1}{2}})$. Hence, it follows that $1\leq \langle(-\partial_s^2+s^2)\frac{\chi_\omega\phi_\omega}{\|\chi_\omega\phi_\omega\|_{L^2(\R,ds)}},\frac{\chi_\omega\phi_\omega}{\|\chi_\omega\phi_\omega\|_{L^2(\R,ds)}}\rangle_{L^2(\mathbb{R},ds)}=\lambda_\omega+O(\omega^{-\frac{1}{2}})$.

For $(iii)$, assume that \eqref{1d lowest energy convergence} holds. 
We observe that for $s>\sqrt{\omega}$,
$$
|v(s)|^2=-2\int_{s}^\infty  v  v'  ds_1\le 2\|v\|_{L^2(1_{[-\sqrt{\omega},\infty)}ds)} \|v'\|_{L^2(1_{[-\sqrt{\omega},\infty)}ds)}, 
$$
Then by \eqref{1d lowest energy convergence}, we have
$$
\|\phi_\omega-\phi_\infty\|_{L^4(1_{[-\sqrt{\omega},\infty)}ds)}^4\le \Big\{\sup_{s>\sqrt{\omega}}|\phi_\omega-\phi_\infty|^2\Big\}\|\phi_\omega-\phi_\infty\|_{L^2(1_{[-\sqrt{\omega},\infty)}ds)}^2 \lesssim \omega^{-2},
$$ 
which proves \eqref{1d l4conv}.

Next, we prove \eqref{1d lowest energy convergence}.
We note that by Proposition \ref{1d eigenstate} $(ii)$, $\|(1-\chi_\omega)\phi_\omega\|_{\Sigma({\bf 1}_{[-\sqrt{\omega},\infty)}ds)}\lesssim \omega^{-\frac{1}{2}}$, where $J\subset \R$, $\|v\|_{\Sigma({\bf 1}_{J} ds)}=\langle v, v\rangle_{\Sigma({\bf 1}_{J} ds)}^{\frac{1}{2}}$ and 
$$\langle v,w\rangle_{\Sigma({\bf 1}_{J}ds)}=\int_{-\infty}^\infty (vw+v'w'+s^2vw){\bf 1}_{J}(s) ds.$$ Thus, it suffices to show $\|\chi_\omega\phi_\omega-\phi_\infty\|_{\Sigma(ds)}\lesssim \omega^{-\frac{1}{2}}$, where $\|v\|_{\Sigma(ds)}:=\|v\|_{\Sigma({\bf 1}_{\R} ds)}$. For the proof, we decompose $\chi_\omega\phi_\omega-\phi_\infty=c_\omega\phi_\infty+r_\omega$ with 
$$
c_\omega=\frac{\langle\chi_\omega\phi_\omega-\phi_\infty,\phi_\infty\rangle_{\Sigma(ds)}}{\|\phi_\infty\|_{\Sigma(ds)}^2}=\langle\chi_\omega\phi_\omega,\phi_\infty\rangle_{L^2(\mathbb{R}, ds)}-1,$$ because $(1-\partial_s^2+s^2)\phi_\infty=2\phi_\infty$. On the other hand, by Proposition \ref{1d eigenstate} $(ii)$, $c_\omega=\frac{1}{2}(\|\chi_\omega\phi_\omega\|_{L^2(\mathbb{R},ds)}^2+\|\phi_\infty\|_{L^2(\mathbb{R},ds)}^2-\|\chi_\omega\phi_\omega-\phi_\infty\|_{L^2(\mathbb{R},ds)}^2)-1=-\frac{1}{2}\|\chi_\omega\phi_\omega-\phi_\infty\|_{L^2(\mathbb{R},ds)}^2+O(\omega^{-\frac{1}{2}})$. Thus, we have 
$$\begin{aligned}
\|\chi_\omega\phi_\omega-\phi_\infty\|_{\Sigma(ds)}&\leq |c_\omega|\|\phi_\infty\|_{\Sigma(ds)}+\|r_\omega\|_{\Sigma(ds)}\\
&\leq \frac{1}{\sqrt{2}}\|\chi_\omega\phi_\omega-\phi_\infty\|_{L^2(\mathbb{R},ds)}^2+\|r_\omega\|_{\Sigma(ds)}+O\left(\omega^{-\frac{1}{2}}\right).
\end{aligned}$$
Now, we recall from the proof of Proposition \ref{1d eigenstate} $(i)$, 
$\|\chi_\omega\phi_\omega\|_{L^2(\mathbb{R},ds)}\to1$. Hence, it follows  that $\chi_\omega\phi_\omega\to \phi_\infty$ in $L^2(\mathbb{R},ds)$, and consequently $\|\chi_\omega\phi_\omega-\phi_\infty\|_{\Sigma(ds)}\lesssim \|r_\omega\|_{\Sigma(ds)}+O(\omega^{-\frac{1}{2}})$.

It remains to show $\|r_\omega\|_{\Sigma(ds)}\lesssim\omega^{-\frac{1}{2}}$. Indeed, $\langle r_\omega,\phi_\infty \rangle_{\Sigma(ds)}=2\int_{\R} r_\omega  \phi_\infty ds=0$ and $-\partial_s^2+s^2$ has a simple lowest eigenvalue $1$ (with the eigenfunction $\phi_\infty$) and the second eigenvalue is $3$. Thus, it follows that $\|r_\omega\|_{\Sigma(ds)}=\|\frac{\sqrt{1-\partial_s^2+s^2}}{-\partial_s^2+s^2}(-\partial_s^2+s^2)r_\omega\|_{L^2(\mathbb{R},ds)}\leq\frac{2}{3}\|(-\partial_s^2+s^2)r_\omega\|_{L^2(\mathbb{R},ds)}$, where we used that $\frac{\sqrt{1+\lambda}}{\lambda}\leq\frac{2}{3}$ for $\lambda\geq 3$. Moreover, we have
$$\begin{aligned}
(-\partial_s^2+s^2)r_\omega&=(-\partial_s^2+s^2)\left(\chi_\omega\phi_\omega-(c_\omega+1)\phi_\infty\right)\\
&=\big(-\phi_\omega''+s^2\phi_\omega)\chi_\omega-2\phi_\omega'\chi_\omega'-\chi_\omega''\phi_\omega-(c_\omega+1)\phi_\infty,
\end{aligned}$$
which is  equal to, by the assumptions on $U_\omega$, the equation $\mathcal{H}_\omega^{(1D)}\phi_\omega=\lambda_\omega\phi_\omega$  and Proposition \ref{1d eigenstate} $(i)$ and $(ii)$,
$$\begin{aligned}
&\left( \frac{1}{\sqrt{\omega}\sigma_\omega}\phi_\omega'+(s^2-U_\omega) \phi_\omega+\lambda_\omega\phi_\omega\right)\chi_\omega-(c_\omega+1)\phi_\infty+O_{L^2(\mathbb{R},ds)}\left(\omega^{-\frac{1}{2}}\right)\\
&=\lambda_\omega\chi_\omega\phi_\omega-(c_\omega+1)\phi_\infty+O_{L^2(\mathbb{R},ds)}\left(\omega^{-\frac{1}{2}}\right)=r_\omega+O_{L^2(\mathbb{R},ds)}\left(\omega^{-\frac{1}{2}}\right), 
\end{aligned}$$
where $v=O_{L^2(\mathbb{R},ds)}\left(\omega^{-\frac{1}{2}}\right)$ is defined by $\|v\|_{L^2(\R,ds)}=O(\omega^{-\frac12})$.
Thus, it follows that $\|r_\omega\|_{\Sigma(ds)}\leq\frac{2}{3}\|r_\omega\|_{\Sigma(ds)}+O(\omega^{-\frac{1}{2}})$, which completes the proof.
\end{proof}

\begin{proof}[Sketch of Proof of Proposition \ref{1d eigenstate} $(iv)$]
Proposition \ref{1d eigenstate} $(iv)$ follows almost identically as the ground state case, thus we omit the proof. Indeed, for the second eigenvalue for $\mathcal{H}_\omega^{(1D)}$, we may consider the variational problem 
\begin{equation}\label{1d 2nd eigenvalue problem}
\lambda_\omega'=\min \left\{ \mathcal{E}_\omega^{(1d)}(v)=\frac12\langle \mathcal{H}_\omega^{(1D)} v,v\rangle_{L^2(\sigma_\omega ds)}:\ \|v\|_{L^2(\sigma_\omega ds)}=1\textup{ and }\langle v,\phi_\omega\rangle_{L^2(\sigma_\omega ds)}=0\right\}.
\end{equation} 
Since the second eigenvalue of $-\partial_s^2+s^2$ is 3 and the corresponding eigenfunction is square exponentially decreasing, repeating the proof of Proposition \ref{1d eigenstate} $(i)$-$(iii)$, one can obtain the second lowest eigenvalue asymptotic. 
\end{proof}

\section{Truncated projection onto the 2D lowest eigenspace}\label{subsec: truncated projection}

For both time-dependent/independent cases, we are concerned with 3D states mostly concentrated on the 2D lowest energy state $\Phi_\omega(s,z)=\phi_\omega(s)\phi_\infty(z)$.

\begin{remark}\label{why truncation?}
At first glance, one might try to approximate such a 3D state by the  projected state
\begin{equation}\label{projection0}
(\mathcal{P}_{\Phi_\omega}v)(s,\theta,z):=\langle v(\cdot,\cdot,\theta),\Phi_\omega\rangle_{L^2(\sigma_\omega dsdz)}\Phi_{\omega}(s,z).
\end{equation}
However, it turns out to be impossible, because the projected state $ \mathcal{P}_{\Phi_\omega}v$ has infinite modified energy. Indeed, we have
\begin{equation}\label{too singular integral}
\big\|\partial_\theta(\mathcal{P}_{\Phi_\omega}v)\big\|_{L^2(\frac{1}{\sigma_\omega})}^2 =\big\|\partial_\theta \langle v(\cdot,\cdot,\theta),\Phi_\omega\rangle_{L^2(\sigma_\omega dsdz)}\big\|_{L^2(\mathbb{S}^1)}^2 \int_{-\sqrt{\omega}}^\infty \frac{\phi_\omega(s)^2}{1+\frac{s}{\sqrt{\omega}}}ds=\infty,
\end{equation}
because $\phi_\omega(-\sqrt{\omega})>0 $ (see Remark \ref{linear differential operator remarks} $(iii)$).
\end{remark}

In this section, we introduce a modified projection to bypass non-integrability in \eqref{too singular integral}, and collect some useful properties. Let $\chi:[0,\infty)\to[0,1]$ be a smooth cut-off  such that $\chi\equiv0$ on $[0,1]$ and $\chi\equiv1$ on $[2,\infty)$ given in Theorem \ref{dimension reduction for tNLS}. For the lowest energy state $\Phi_\omega$, let
$$\tilde{\Phi}_\omega(s,z):=\frac{\chi_\omega(s)\Phi_{\omega}(s,z)}{\|\chi_\omega \Phi_{\omega} \|_{L^2(\sigma_\omega dsdz)}}$$
be a cut-offed normalized function, where $\chi_\omega(s)=\chi(s+\sqrt{\omega})$. Note that $\chi_\omega\equiv0$ on $[-\sqrt{\omega},-\sqrt{\omega}+1]$ and $\chi_\omega\equiv1$ on $[-\sqrt{\omega}+2,\infty)$. We define the 2D $\tilde{\Phi}_\omega$-directional projection $\mathcal{P}_{\tilde{\Phi}_\omega}$ by 
\begin{equation}\label{projection}
(\mathcal{P}_{\tilde{\Phi}_\omega}v)(s,\theta,z):=v_\parallel(\theta)\tilde{\Phi}_{\omega}(s,z),
\end{equation}
where 
$$v_\parallel(\theta):=\big\langle v(\cdot,\cdot,\theta),\tilde{\Phi}_\omega\big\rangle_{L^2(\sigma_\omega dsdz)}=\int_{-\infty}^\infty\int_{-\sqrt{\omega}}^\infty v(s,z,\theta) \tilde{\Phi}_\omega(s,z)\sigma_\omega(s)dsdz$$
is the $\tilde{\Phi}_\omega(s,z)$-directional component, and define its orthogonal complement by
$$\mathcal{P}_{\tilde{\Phi}_\omega}^\perp:=1-\mathcal{P}_{\tilde{\Phi}_\omega}.$$
Obviously, by orthogonality, we have
\begin{equation}\label{oth2l}
\|v\|_{L^2(\sigma_\omega)}^2=\|\mathcal{P}_{\tilde{\Phi}_\omega}v\|_{L^2(\sigma_\omega)}^2+\|\mathcal{P}_{\tilde{\Phi}_\omega}^\perp v\|_{L^2(\sigma_\omega)}^2.
\end{equation}

By definition, it is obvious that $\|v_\parallel\|_{L^2(\mathbb{S}^1)}\leq \|v\|_{L^2(\sigma_\omega)}$. Indeed, the directional component satisfies more bounds.

\begin{lemma}[Bounds for the $\tilde{\Phi}_\omega(s,z)$-directional component]\label{mno1}
For any $1\leq p<\infty$, we have
$$\|v_\parallel\|_{L^p(\mathbb{S}^1)}\lesssim \|v\|_{L^p(\sigma_\omega)},\quad\|v_\parallel\|_{L^2(\mathbb{S}^1)}\lesssim \|v\|_{L^2(\frac{1}{\sigma_\omega})}.$$
\end{lemma}

\begin{proof}
By the Gaussian bound for $\phi_\omega$ (Proposition \ref{1d eigenstate} $(ii)$) and the H\"older inequality, we show that $v_\parallel=\langle v(\cdot,\cdot,\theta), \tilde{\Phi}_\omega\rangle_{L^2(\sigma_\omega dsdz)}$ obeys $\|v_\parallel\|_{L^p(\mathbb{S}^1)}\leq \|\tilde{\Phi}_\omega\|_{L^{p'}(\sigma_\omega dsdz)}\|v\|_{L^p(\sigma_\omega)}\sim \|v\|_{L^p(\sigma_\omega)}$ and $\|v_\parallel\|_{L^2(\mathbb{S}^1)}\leq  \|\tfrac{1}{\sigma_\omega}v\|_{L^2(\sigma_\omega)}\|\sigma_\omega\tilde{\Phi}_\omega\|_{L^{p'}(\sigma_\omega dsdz)}\lesssim  \|v\|_{L^2(\frac{1}{\sigma_\omega})}$.
\end{proof}

The following lemma asserts that as $\omega$ increases, the modified projection gets closer to the lowest eigenstate projection $\mathcal{P}_{\Phi_\omega}$ (see \eqref{projection0}) except the angular derivative semi-norm $\|\partial_\theta v\|_{L^2(\frac{1}{\sigma_\omega})}$.

\begin{lemma}[Truncation error bounds]\label{truncation error}
Let $1\le p<\infty$ and   $\alpha_0>0$ be the constant given in Proposition \ref{1d eigenstate}.  For any $c\in(0,\alpha_0)$, there exists $\omega_c\geq 1$ such that if $\omega\geq \omega_c$, then
$$\begin{aligned}
 \big\|(\mathcal{P}_{\Phi_\omega}-\mathcal{P}_{\tilde{\Phi}_\omega})v\big\|_{L^2(\sigma_\omega)\cap {\dot{\Sigma}}_{\omega; (s,z)}} =\big\|(\mathcal{P}_{\Phi_\omega}^\perp-\mathcal{P}_{\tilde{\Phi}_\omega}^\perp) v\big\|_{L^2(\sigma_\omega)\cap{\dot{\Sigma}}_{\omega; (s,z)}}&\lesssim e^{-c\omega}\|v\|_{L^2(\sigma_\omega)},\\
\big\|(\mathcal{P}_{\Phi_\omega}-\mathcal{P}_{\tilde{\Phi}_\omega})v\big\|_{L^p(\sigma_\omega)}=\big\|(\mathcal{P}_{\Phi_\omega}^\perp-\mathcal{P}_{\tilde{\Phi}_\omega}^\perp) v\big\|_{L^p(\sigma_\omega)}&\lesssim e^{-c\omega}\|v\|_{L^p(\sigma_\omega)}, 
\end{aligned}$$
where $\mathcal{P}_{\Phi_\omega}^\perp:=1-\mathcal{P}_{\Phi_\omega}.$
\end{lemma}

\begin{proof} 
Since $\mathcal{P}_{\Phi_\omega}^\perp -\mathcal{P}_{\tilde{\Phi}_\omega}^\perp=\mathcal{P}_{\tilde{\Phi}_\omega}-\mathcal{P}_{\Phi_\omega}$, it suffices to show the lemma for
$$(\mathcal{P}_{\Phi_\omega} -\mathcal{P}_{\tilde{\Phi}_\omega}) v=\langle v(\cdot,\cdot,\theta), \Phi_\omega-\tilde{\Phi}_\omega\rangle_{L^2(\sigma_\omega dsdz)} \Phi_\omega+v_\parallel(\theta)(\Phi_\omega-\tilde{\Phi}_\omega).$$
Indeed, by the fast decay of $\phi_\omega$ (Proposition \ref{1d eigenstate} $(ii)$), we have $\|\Phi_\omega-\tilde{\Phi}_\omega\|_{L^r(\sigma_\omega dsdz)}\lesssim e^{-c\omega}$ for any $r>1$. Hence, the H\"older inequality and Lemma \ref{mno1} imply that 
$$\|(\mathcal{P}_{\Phi_\omega}-\mathcal{P}_{\tilde{\Phi}_\omega})v\|_{L^p(\sigma_\omega)}\lesssim e^{-c\omega}\big(\|v\|_{L^p(\sigma_\omega)}+\|v_{\parallel}\|_{L^p(\mathbb{S}^1)}\big)\lesssim e^{-c\omega}\|v\|_{L^p(\sigma_\omega)}.$$
By the same way, one can show that $\|(\mathcal{P}_{\Phi_\omega}-\mathcal{P}_{\tilde{\Phi}_\omega})v\|_{L^2(\sigma_\omega)\cap {\dot{\Sigma}}_{\omega; (s,z)}}\lesssim e^{-c\omega}\|v\|_{L^2(\sigma_\omega)}$.\end{proof}

Moreover, the orthogonal complement projection satisfies the following bounds.

\begin{lemma}[Bounds for the orthogonal complement]\label{orthogonal complement bound}
Let $\alpha_0>0$ be the constant given in Proposition \ref{1d eigenstate}.  For any $c\in(0,\alpha_0)$, there exists $\omega_c\geq 1$ such that if $\omega\geq \omega_c$, then
$$\|\mathcal{P}_{\tilde{\Phi}_\omega}^\perp v\|_{{\dot{\Sigma}}_{\omega; (s,z)}}, \|\mathcal{P}_{\tilde{\Phi}_\omega} ^\perp v\|_{L^2(\sigma_\omega)}, \|\nabla_{(s,z)}\mathcal{P}_{\tilde{\Phi}_\omega} ^\perp v\|_{L^2(\sigma_\omega)}\lesssim \|v\|_{\dot{\Sigma}_{\omega;(s,z)}}+e^{-c\omega}\|v\|_{L^2(\sigma_\omega)}.$$
\end{lemma}
\begin{proof} By Lemma \ref{truncation error}, it suffices to prove that
$$\|\mathcal{P}_{\Phi_\omega}^\perp v\|_{{\dot{\Sigma}}_{\omega; (s,z)}}, \|\mathcal{P}_{\Phi_\omega} ^\perp v\|_{L^2(\sigma_\omega)}, \|\nabla_{(s,z)}\mathcal{P}_{\Phi_\omega} ^\perp v\|_{L^2(\sigma_\omega)}\lesssim \|v\|_{\dot{\Sigma}_{\omega;(s,z)}}.$$
Indeed, by the definition of the semi-norm $\|\cdot\|_{\dot{\Sigma}_{\omega;(s,z)}}$ (see \eqref{seminorm}) with $\|\mathcal{P}_{\Phi_\omega} v\|_{\dot{\Sigma}_{\omega;(s,z)}}=0$, we have $\|\mathcal{P}_{\Phi_\omega}^\perp v\|_{{\dot{\Sigma}}_{\omega; (s,z)}}=\|v\|_{{\dot{\Sigma}}_{\omega; (s,z)}}$ and  $\Lambda_\omega'\|\mathcal{P}_{\Phi_\omega}^\perp v\|_{L^2(\sigma_\omega)}^2\leq\langle\mathcal{H}_\omega^{(2D)}\mathcal{P}_{\Phi_\omega}^\perp v,\mathcal{P}_{\Phi_\omega}^\perp v \rangle_{L^2(\sigma_\omega)}=\|v\|_{\dot{\Sigma}_{\omega;(s,z)}}^2+\Lambda_\omega\|\mathcal{P}_{\Phi_\omega}^\perp v\|_{L^2(\sigma_\omega)}^2$, where $\Lambda_\omega$ and $\Lambda_\omega'$ are respectively the first and the second lowest eigenvalues for $\mathcal{H}_\omega^{(2D)}$. Thus, the spectral gap (see \eqref{secei}) implies that $\|\mathcal{P}_{\Phi_\omega}^\perp v\|_{L^2(\sigma_\omega)}^2\leq\frac{1}{\Lambda_\omega'-\Lambda_\omega}\|v\|_{\dot{\Sigma}_{\omega;(s,z)}}^2$ and  
\begin{align*}
\|\nabla_{(s,z)}\mathcal{P}_{\Phi_\omega} ^\perp v\|_{L^2(\sigma_\omega)}^2&\leq \langle\mathcal{H}_\omega^{(2D)}\mathcal{P}_{\Phi_\omega}^\perp v,\mathcal{P}_{\Phi_\omega}^\perp v \rangle_{L^2(\sigma_\omega)}\\
&= \|v\|_{\dot{\Sigma}_{\omega;(s,z)}}^2+\Lambda_\omega\|\mathcal{P}_{\Phi_\omega} ^\perp v\|_{L^2(\sigma_\omega)}^2\leq \frac{\Lambda_\omega'}{\Lambda_\omega'-\Lambda_\omega}\|v\|_{\dot{\Sigma}_{\omega;(s,z)}}^2.
\end{align*}
\end{proof}

A key advantage of using the modified projection is that if $\|\partial_\theta v\|_{L^2(\frac{1}{\sigma_\omega})}^2$ is bounded, then so is $\|\partial_\theta(\mathcal{P}_{\tilde{\Phi}_\omega}v)\|_{L^2(\frac{1}{\sigma_\omega})}^2$, since the possible singularity at $s=-\sqrt{\omega}$ is eliminated (see \eqref{too singular integral}). The following almost orthogonality proves that besides boundedness, the projection asymptotically reduces the semi-norm, i.e., $\|\partial_\theta(\mathcal{P}_{\tilde{\Phi}_\omega}v)\|_{L^2(\frac{1}{\sigma_\omega})}^2\leq \|\partial_\theta v\|_{L^2(\frac{1}{\sigma_\omega})}^2+o_\omega(1)$. 

\begin{lemma}[Asymptotic Pythagorean identities]\label{asymptotic orthogonality}
Let $\alpha_0>0$ be a constant given in Proposition \ref{1d eigenstate}. For any $c\in(0,\alpha_0)$, there exists $\omega_c\geq 1$ such that if $\omega\geq \omega_c$, then
\begin{align}
\Big|\|v\|_{L^2(\frac{1}{\sigma_\omega})}^2-\big\|\mathcal{P}_{\tilde{\Phi}_\omega} v\big\|_{L^2(\frac{1}{\sigma_\omega})}^2-\big\|\mathcal{P}_{\tilde{\Phi}_\omega}^\perp v\big\|_{L^2(\frac{1}{\sigma_\omega})}^2\Big|&\lesssim \omega^{-\frac{1}{2}}\|v\|_{L^2(\frac{1}{\sigma_\omega})}^2.\label{asymptotic Pythagorean identity 3}
\end{align}
\end{lemma}

\begin{proof}
It suffices to show that $|\langle\mathcal{P}_{\tilde{\Phi}_\omega} v, \mathcal{P}_{\tilde{\Phi}_\omega}^\perp v\rangle_{L^2(\frac{1}{\sigma_\omega})}|\lesssim \omega^{-\frac{1}{2}}\|v\|_{L^2(\frac{1}{\sigma_\omega})}^2$. Indeed, we have
$$\langle\mathcal{P}_{\tilde{\Phi}_\omega} v, \mathcal{P}_{\tilde{\Phi}_\omega}^\perp v\rangle_{L^2(\frac{1}{\sigma_\omega})}=\big\langle(1-\sigma_\omega^2)\mathcal{P}_{\tilde{\Phi}_\omega} v, \mathcal{P}_{\tilde{\Phi}_\omega}^\perp v\big\rangle_{L^2(\frac{1}{\sigma_\omega})}=-\omega^{-\frac{1}{2}}\big\langle s(1+\sigma_\omega)\mathcal{P}_{\tilde{\Phi}_\omega} v, \mathcal{P}_{\tilde{\Phi}_\omega}^\perp v\big\rangle_{L^2(\frac{1}{\sigma_\omega})},$$
because $\langle \mathcal{P}_{\tilde{\Phi}_\omega} v, \mathcal{P}_{\tilde{\Phi}_\omega}^\perp v\rangle_{L^2(\sigma_\omega)}=0$ by orthogonality and $\sigma_\omega-1=\frac{s}{\sqrt{\omega}}$. Hence, it follows that
$$\big|\langle\mathcal{P}_{\tilde{\Phi}_\omega} v, \mathcal{P}_{\tilde{\Phi}_\omega}^\perp v\rangle_{L^2(\frac{1}{\sigma_\omega})}\big|\leq \omega^{-\frac{1}{2}}\|s(1+\sigma_\omega)\mathcal{P}_{\tilde{\Phi}_\omega} v\|_{L^2(\frac{1}{\sigma_\omega})}\Big\{\|v\|_{L^2(\frac{1}{\sigma_\omega})}+\|\mathcal{P}_{\tilde{\Phi}_\omega} v\|_{L^2(\frac{1}{\sigma_\omega})}\Big\}.$$
Note that by the fast decay of $\phi_\omega$ (see Proposition \ref{1d eigenstate}) and the $v_\parallel$-bound (Lemma \ref{mno1}),
$$\begin{aligned}
\|s(1+\sigma_\omega)\mathcal{P}_{\tilde{\Phi}_\omega} v\|_{L^2(\frac{1}{\sigma_\omega})}&=\frac{\|s\tfrac{1+\sigma_\omega}{\sigma_\omega}\chi_\omega \Phi_\omega\|_{L^2(\sigma_\omega dsdz)}\|v_\parallel\|_{L^2(\mathbb{S}^1)}}{{\|\chi_\omega \Phi_{\omega} \|_{L^2(\sigma_\omega dsdz)}}}\lesssim \|v_\parallel\|_{L^2(\mathbb{S}^1)}\lesssim \|v\|_{L^2(\frac{1}{\sigma_\omega})},\\
\|\mathcal{P}_{\tilde{\Phi}_\omega} v\|_{L^2(\frac{1}{\sigma_\omega})}&=\frac{\|\tfrac{\chi_\omega}{\sigma_\omega} \Phi_\omega\|_{L^2(\sigma_\omega dsdz)}\|v_\parallel\|_{L^2(\mathbb{S}^1)}}{{\|\chi_\omega \Phi_{\omega} \|_{L^2(\sigma_\omega dsdz)}}}\lesssim \|v_\parallel\|_{L^2(\mathbb{S}^1)}\lesssim \|v\|_{L^2(\frac{1}{\sigma_\omega})}.
\end{aligned}$$
Therefore, \eqref{asymptotic Pythagorean identity 3} follows.
\end{proof}

\section{Refined Gagliardo-Nirenberg inequality}\label{subsec: modified GN inequality}

For the modified problem in Section \ref{sec: reformulation}, we translate basic inequalities in the new $(s,z,\theta)$-coordinates via the relation
$$v(s,z,\theta)=\omega^{\frac{1}{4}}u(s+\sqrt{\omega},z, \theta)$$
with $\|u\|_{L^p(\mathbb{R}^3)}=\omega^{-\frac{1}{4}+\frac{1}{2p}}\|v\|_{L^p(\sigma_\omega)}$ and $\|\nabla u\|_{L^2(\mathbb{R}^3)}^2=\|\nabla_{(s,z)}v\|_{L^2(\sigma_\omega)}^2+\frac{1}{\omega}\|\partial_\theta v\|_{L^2(\frac{1}{\sigma_\omega})}^2$. For instance, the Sobolev inequality $\|u\|_{L^6(\mathbb{R}^3)}\lesssim \|\nabla u\|_{L^2(\mathbb{R}^3)}$ and the Gagliardo-Nirenberg inequality $\|u\|_{L^4(\mathbb{R}^3)}^4\lesssim \|u\|_{L^2(\mathbb{R}^3)}\|\nabla u\|_{L^2(\mathbb{R}^3)}^3$ are restated respectively as
\begin{equation}\label{translated Sobolev}
\|v\|_{L^6(\sigma_\omega)}\lesssim \omega^{\frac{1}{6}}\Big\{\|\nabla_{(s,z)}v\|_{L^2(\sigma_\omega)}^2+\tfrac{1}{\omega}\|\partial_\theta v\|_{L^2(\frac{1}{\sigma_\omega})}^2\Big\}^{\frac{1}{2}}
\end{equation}
and
\begin{equation}\label{translated GN}
\|v\|_{L^4(\sigma_\omega)}^4\lesssim\sqrt{\omega}\|v\|_{L^2(\sigma_\omega)}\Big\{\|\nabla_{(s,z)}v\|_{L^2(\sigma_\omega)}^2+\tfrac{1}{\omega}\|\partial_\theta v\|_{L^2(\frac{1}{\sigma_\omega})}^2\Big\}^{\frac{3}{2}}.
\end{equation}
However, these inequalities are not good enough to capture the subcritical nature of the problem under the additional constraint $\|v\|_{{\dot{\Sigma}}_{\omega;(s,z)}}^2\leq\delta\omega$.

Our key analysis tool is the following refined version of the inequality \eqref{translated GN}.
\begin{proposition}[Refined Gagliardo-Nirenberg inequality]\label{GN}
\begin{equation}\label{eq: GN}
\|v\|_{L^4(\sigma_\omega)}^4\lesssim\sqrt{\omega}\|v\|_{L^2(\sigma_\omega)}\|\nabla_{(s,z)}v\|_{L^2(\sigma_\omega)}^2\Big\{\|\nabla_{(s,z)}v\|_{L^2(\sigma_\omega)}^2+\tfrac{1}{\omega}\|\partial_\theta v\|_{L^2(\frac{1}{\sigma_\omega})}^2\Big\}^{\frac{1}{2}}.
\end{equation}
\end{proposition}

For the proof, we recall the following standard inequalities.
\begin{lemma}
\begin{enumerate}[$(i)$]
\item (Gagliardo-Nirenberg inequality on $\mathbb{R}^d$) If $\frac{1}{q}=\frac{1}{2}-\frac{\beta}{d}$, $q>2$ and $0<\beta<1$, then
\begin{equation}\label{standard GN ineq}
\|u\|_{L^q(\mathbb{R}^d)}\lesssim\|u\|_{L^2(\mathbb{R}^d)}^{1-\beta}\|\nabla u\|_{L^2(\mathbb{R}^d)}^{\beta}.
\end{equation}
\item (Gagliardo-Nirenberg inequality on the unit circle $\mathbb{S}^1$ \cite{BO, GLT}) If $q>2$, then
\begin{equation}\label{circle GN ineq}
\|u\|_{L^q(\mathbb{S}^1)}\lesssim \|u\|_{L^2(\mathbb{S}^1)}^{\frac{q+2}{2q}}\|u\|_{H^1(\mathbb{S}^1)}^\frac{q-2}{2q}.
\end{equation}
\item (Radial Gagliardo-Nirenberg inequality\footnote{It simply follows from the well-known inequality $\underset{y\in\mathbb{R}^2\setminus\{0\}}\sup|y| |u(y)|^2\lesssim \|u\|_{L^2(\mathbb{R}^2)}\|\nabla_y u\|_{L^2(\mathbb{R}^2)}$ for radial functions, via $u(|y|)=\omega^{-\frac{1}{4}}v(|y|-\sqrt{\omega})$ (see \cite{CO} for instance).})
\begin{equation}\label{radial GN inequ}
\sup_{s>-\sqrt{\omega}}\sigma_\omega(s) |v(s)|^2\lesssim\|v\|_{L^2(\sigma_\omega ds)}\|\partial_sv\|_{L^2(\sigma_\omega ds)}.
\end{equation}
\end{enumerate}
\end{lemma}

\begin{proof}[Proof of Proposition \ref{GN}]
First, we apply the inequality \eqref{circle GN ineq} on $\mathbb{S}^1$,
$$\begin{aligned}
\|v\|_{L^4(\sigma_\omega)}^4&=\big\|\|v\|_{L^4(\mathbb{S}^1)}^4\big\|_{L^1(\sigma_\omega dsdz)}\\
&\lesssim \big\|\|v\|_{L^2(\mathbb{S}^1)}^3\|\partial_\theta v\|_{L^2(\mathbb{S}^1)}\big\|_{L^1(\sigma_\omega dsdz)}+\big\|\|v\|_{L^2(\mathbb{S}^1)}^4\big\|_{L^1(\sigma_\omega dsdz)}\\
&=:I+II.
\end{aligned}$$
For $I$, by the H\"older and the Minkowski inequalities, we have
$$\begin{aligned}
I&\leq \big\| \|\sigma_\omega^{\frac{1}{3}}v\|_{L^2(\mathbb{S}^1)}\big\|_{L^{6}(\sigma_\omega dsdz)}^3\big\|\|\tfrac{1}{\sigma_\omega}\partial_\theta v\|_{L^2(\mathbb{S}^1)}\big\|_{L^2(\sigma_\omega dsdz)}\\
&\leq\big\| \|\sigma_\omega^{\frac{1}{3}}v\|_{L^{6}(\sigma_\omega dsdz)}\big\|_{L^2(\mathbb{S}^1)}^3\|\partial_\theta v\|_{L^2(\frac{1}{\sigma_\omega})}.
\end{aligned}$$
Then, it follows from the radial inequality \eqref{radial GN inequ} and the inequality \eqref{standard GN ineq} with $d=1$ for the $z$-variable that 
$$\begin{aligned}
\|\sigma_\omega^{\frac{1}{3}}v\|_{L^{6}(\sigma_\omega dsdz)}^6&=\int_{-\infty}^\infty\int_{-\sqrt{\omega}}^\infty (\sigma_\omega|v |^2)^2  |v |^2 \sigma_\omega dsdz\lesssim\int_{-\infty}^\infty\|v\|_{L^2(\sigma_\omega ds)}^4\|\partial_sv\|_{L^2(\sigma_\omega ds)}^2dz\\
&\leq \left\|\|v\|_{L^\infty(\mathbb{R},dz)}\right\|_{L^2(\sigma_\omega ds)}^4\int_{-\infty}^\infty\|\partial_sv\|_{L^2(\sigma_\omega ds)}^2dz\\
&\lesssim\Big\|\|v\|_{L_z^2(\mathbb{R}, dz)}^{\frac{1}{2}}\|\partial_zv\|_{L_z^2(\mathbb{R}, dz)}^{\frac{1}{2}}\Big\|_{L^2(\sigma_\omega ds)}^4\|\partial_sv\|_{L^2(\sigma_\omega dsdz)}^2\\
&\leq\|v\|_{L^2(\sigma_\omega dsdz)}^2\|\nabla_{(s,z)}v\|_{L^2(\sigma_\omega dsdz)}^4.
\end{aligned}$$
Consequently, by the H\"older inequality in $\|\cdot\|_{L^2(\mathbb{S}^1)}$, we obtain
$$\big\|\|v\|_{L^{6}(\sigma_\omega^3dsdz)}\big\|_{L^2(\mathbb{S}^1)}\lesssim \Big\| \|v\|_{L^2(\sigma_\omega dsdz)}^\frac{1}{3}\|\nabla_{(s,z)}v\|_{L^2(\sigma_\omega dsdz)}^\frac{2}{3}\Big\|_{L^2(\mathbb{S}^1)}\le  \|v\|_{L^2(\sigma_\omega)}^\frac{1}{3}\|\nabla_{(s,z)}v\|_{L^2(\sigma_\omega)}^\frac{2}{3}.$$ 
Hence, we conclude that $I\lesssim\|v\|_{L^2(\sigma_\omega)}\|\nabla_{(s,z)}v\|_{L^2(\sigma_\omega)}^2\|\partial_\theta v\|_{L^2(\frac{1}{\sigma_\omega})}$. On the other hand, for $II$, applying a Gagliardo-Nirenberg type inequality\footnote{For $u=u(|y|, z)$, the inequality $\|u\|_{L^4(\mathbb{R}^3)}^4\lesssim \|u\|_{L^2(\mathbb{R}^3)}\|\nabla u\|_{L^2(\mathbb{R}^3)}^3$ can be written as 
$$\int_{-\infty}^\infty\int_0^\infty |u(r,z)|^4 rdrdz\lesssim \left\{\int_{-\infty}^\infty\int_0^\infty |u(r,z)|^2 rdrdz\right\}^{\frac12}\left\{\int_{-\infty}^\infty\int_0^\infty |\nabla_{(r,z)}u(r,z)|^2 rdrdz\right\}^{\frac32}.$$
Then, inserting $u(r,z)=\omega^{-\frac{1}{4}}v(r-\sqrt{\omega},z)$, we obtain the desired inequality.} 
$$\|v\|_{L^4(\sigma_\omega dsdz)}^4\lesssim \sqrt{\omega}\|v\|_{L^2(\sigma_\omega dsdz)}\|\nabla_{(s,z)}v\|_{L^2(\sigma_\omega dsdz)}^3,$$
we prove that 
$$\begin{aligned}
II&=\left\|\|v\|_{L^2(\mathbb{S}^1)}\right\|_{L^4(\sigma_\omega dsdz)}^4\leq \left\|\|v\|_{L^4(\sigma_\omega dsdz)}\right\|_{L^2(\mathbb{S}^1)}^4\\
&\lesssim\sqrt{\omega}\Big\|\|v\|_{L^2(\sigma_\omega dsdz)}^\frac14\|\nabla_{(s,z)}v\|_{L^2(\sigma_\omega dsdz)}^\frac{3}{4}\Big\|_{L^2(\mathbb{S}^1)}^4\\
&\le\sqrt{\omega}\|v\|_{L^2(\sigma_\omega)} \|\nabla_{(s,z)}v\|_{L^2(\sigma_\omega)}^3.
\end{aligned}$$
Therefore, combining the bounds for $I$ and $II$, we complete the proof.
\end{proof}

In our applications, we modify the inequality expressing in terms of the quantities in the modified mass and energy (see \eqref{modified mass'} and \eqref{modified energy'}) but we also emphasize uniformity (in $\omega$) for the implicit constant in the statement.

\begin{corollary}\label{GN1}
There exists $C_{GN}>0$, independent of large $\omega\geq 1$, such that 
$$\|v\|_{L^4(\sigma_\omega)}^4\leq C_{GN}\|v\|_{L^2(\sigma_\omega)}\Big\{\sqrt{\omega}\|v\|_{\dot{\Sigma}_{\omega;(s,z)}}^2\|v\|_{\dot{\Sigma}_{\omega}}+\|v \|_{L^2(\sigma_\omega)}^2\|\partial_\theta v\|_{L^2(\frac{1}{\sigma_\omega})}+\|v \|_{L^2(\sigma_\omega)}^3\Big\},$$
where $\|v\|_{\dot{\Sigma}_{\omega}}^2=\|v\|_{\dot{\Sigma}_{\omega;(s,z)}}^2+\frac{1}{\omega}\|\partial_\theta v\|_{L^2(\frac{1}{\sigma_\omega})}^2$.
\end{corollary}

\begin{proof}
We decompose $v=\mathcal{P}_{\tilde{\Phi}_\omega} v+\mathcal{P}_{\tilde{\Phi}_\omega}^\perp v$, where $\mathcal{P}_{\tilde{\Phi}_\omega}$ and $\mathcal{P}_{\tilde{\Phi}_\omega}^\perp$ are the modified projections in Section \ref{subsec: truncated projection}. For the orthogonal complement $\tilde{v}_\omega^\perp=\mathcal{P}_{\tilde{\Phi}_\omega}^\perp v$, we apply the refined Gagliardo-Nirenberg inequality (Proposition \ref{GN}), the  orthogonality \eqref{oth2l} and the asymptotic Pythagorean identities \eqref{asymptotic Pythagorean identity 3} to obtain
$$\begin{aligned}
\|\tilde{v}_\omega^\perp\|_{L^4(\sigma_\omega)}^4&\lesssim \sqrt{\omega}\|\tilde{v}_\omega^\perp\|_{L^2(\sigma_\omega)}\|\nabla_{(s,z)}\tilde{v}_\omega^\perp\|_{L^2(\sigma_\omega)}^2\Big\{\|\nabla_{(s,z)}\tilde{v}_\omega^\perp\|_{L^2(\sigma_\omega)}^2+\tfrac{1}{\omega}\|\partial_\theta \tilde{v}_\omega^\perp\|_{L^2(\frac{1}{\sigma_\omega})}^2\Big\}^{\frac{1}{2}}\\
&\lesssim \sqrt{\omega}\|v\|_{L^2(\sigma_\omega)}\|\nabla_{(s,z)}\tilde{v}_\omega^\perp\|_{L^2(\sigma_\omega)}^2\Big\{\|\nabla_{(s,z)}\tilde{v}_\omega^\perp\|_{L^2(\sigma_\omega)}^2+\tfrac{1}{\omega}\|\partial_\theta v\|_{L^2(\frac{1}{\sigma_\omega})}^2\Big\}^{\frac{1}{2}}.
\end{aligned}$$
Then, it follows from the bound for $\mathcal{P}_{\tilde{\Phi}_\omega}^\perp$ (Lemma \ref{orthogonal complement bound}) that 
$$\|\tilde{v}_\omega^\perp\|_{L^4(\sigma_\omega)}^4\lesssim\|v\|_{L^2(\sigma_\omega)}\Big\{\sqrt{\omega}\|v\|_{\dot{\Sigma}_{\omega;(s,z)}}^2\|v\|_{\dot{\Sigma}_{\omega}}+\|v \|_{L^2(\sigma_\omega)}^2\|\partial_\theta v\|_{L^2(\frac{1}{\sigma_\omega})}+\|v \|_{L^2(\sigma_\omega)}^3\Big\}.$$
For the $\tilde{\Phi}_\omega$-directional component $\mathcal{P}_{\tilde{\Phi}_\omega} v$, we apply   the Gagliardo-Nirenberg inequality on $\mathbb{S}^1$ \eqref{circle GN ineq}. Then, it follows that 
$$\begin{aligned}
\|\mathcal{P}_{\tilde{\Phi}_\omega} v\|_{L^4(\sigma_\omega)}^4&\sim \|v_\parallel\|_{L^4(\mathbb{S}^1)}^4  \lesssim\|v_\parallel\|_{L^2(\mathbb{S}^1)}^3\|\partial_\theta v_\parallel\|_{L^2(\mathbb{S}^1)}+\|v_\parallel\|_{L^2(\mathbb{S}^1)}^4 \\
&\leq \|v \|_{L^2(\sigma_\omega)}^3\|\partial_\theta v\|_{L^2(\frac{1}{\sigma_\omega})}+\|v \|_{L^2(\sigma_\omega)}^4 ,
\end{aligned}$$
where the $v_\parallel$-bounds (Lemma \ref{mno1}) is used in the last inequality. Therefore, combining the bounds for $\mathcal{P}_{\tilde{\Phi}_\omega} v$ and $\tilde{v}_\omega^\perp$, we prove the desired inequality. 
\end{proof}

As a direct consequence, we show that in the focusing case, a wave function having uniformly bounded energy is forbidden to exist in the region $\frac{K}{\omega}<\|v\|_{\dot{\Sigma}_\omega}^2\leq \delta\omega$. 

\begin{corollary}[Forbidden region]\label{forbidden region}
Let $\kappa=-1$ and $\omega\geq 1$ be sufficiently large. Then, for fixed $m,E>0$, there exist small $\delta=\delta(m,E)>0$ and large $C=C(m)$, $K=K(m,E)>0$, independent of large $\omega\geq1$, such that if $v\in\Sigma_\omega$, $\mathcal{M}_\omega[v]=m$,
$$\mathcal{E}_\omega[v]\leq E\quad\textup{and}\quad\|v\|_{\dot{\Sigma}_{\omega;(s,z)}}^2\leq \delta\omega,$$
then
$$\mathcal{E}_\omega[v]\geq- C\quad\textup{and}\quad \|v\|_{\dot{\Sigma}_\omega}^2=\|v\|_{\dot{\Sigma}_{\omega;(s,z)}}^2+\frac{1}{\omega}\|\partial_\theta v_\omega\|_{L^2(\frac{1}{\sigma_\omega})}^2\leq \frac{K}{\omega}.$$
\end{corollary}

\begin{proof}
By the assumptions on $v$, the refined Gagliardo-Nirenberg inequality (Corollary \ref{GN1}) yields a lower bound on the energy, 
$$\mathcal{E}_\omega[v]\geq\frac{\omega}{2}\|v\|_{\dot{\Sigma}_{\omega}}^2-\frac{C_{GN}}{4}\sqrt{m}\Big\{\sqrt{\delta}\omega\|v\|_{\dot{\Sigma}_{\omega;(s,z)}}\|v\|_{\dot{\Sigma}_{\omega}}+m\|\partial_\theta v\|_{L^2(\frac{1}{\sigma_\omega})}+m^{\frac{3}{2}}\Big\}.$$
Consequently, it follows from the Cauchy-Schwarz inequality and smallness of $\delta>0$ that
$$E\geq\mathcal{E}_\omega[v]\geq \frac{\omega}{4}\|v\|_{\dot{\Sigma}_{\omega; (s,z)}}^2+\frac{1}{4}\|\partial_\theta v\|_{L^2(\frac{1}{\sigma_\omega})}^2-C(m)=\frac{\omega}{4}\|v\|_{\dot{\Sigma}_\omega}^2-C(m)$$
for some $C(m)$. Thus, we obtain the desired bound on $\|v\|_{\dot{\Sigma}_{\omega}}^2$.
\end{proof}

\begin{remark}\label{remark: refined constraint}
By the relation \eqref{u-v relation}, Corollary \ref{forbidden region} is translated into that if $u\in \Sigma$, $M[u]=m$, $E_\omega[u]\leq E$ and $\langle (H_\omega^\perp-\Lambda_\omega) u, u\rangle_{L^2(\mathbb{R}^3)}\leq \delta\omega$, then $\langle ( H_\omega-\Lambda_\omega) u, u\rangle_{L^2(\mathbb{R}^3)}\leq  \frac{K}{\omega}$.
\end{remark}

\section{Global-in-time dimension reduction for the time-dependent NLS: Proof of Theorem \ref{dimension reduction for tNLS}}\label{global}

This section is devoted to the first part of the paper where the full 3D \textit{time-dependent} NLS is discussed. First of all, using the refined Gagliardo-Nirenberg inequality, we provide a simple criteria for global existence of solutions (Proposition \ref{prop: 3D tNLS} and \ref{global existence; focusing case}). An important remark is that the family of global solutions we consider includes solutions which evolve from a neighborhood of a constrained minimizer (Theorem \ref{thm: existence}). It is necessary to obtain its orbital stability  (Theorem \ref{thm: unique} $(ii)$), because orbital stability a priori-ly requires solutions nearby to exist for all time. Secondly, we establish the dimension reduction for those global solutions in the strong confinement limit, and derive the time-dependent periodic 1D NLS (Theorem \ref{dimension reduction for tNLS}).

\subsection{Well-posedness of the time-dependent 3D NLS, and a criteria for global existence}

We consider the initial-value problem for the 3D NLS \eqref{NLS}, which is in a strong form represented as 
\begin{equation}\label{3D NLS; integral form}
u_\omega(t)=e^{-it\omega(H_\omega-\Lambda_\omega)}u_{\omega,0}-i\kappa\int_0^t e^{-i(t-t_1)(H_\omega-\Lambda_\omega)}(|u_\omega|^2u_\omega)(t_1)dt_1,
\end{equation}
where $H_\omega=-\Delta+U_\omega(|y|-\sqrt{\omega})+z^2$. This equation is locally well-posed in the weighted Sobolev space $\Sigma$ with the norm
$$\|u\|_{\Sigma}:=\left\{\|u\|_{L^2(\mathbb{R}^3)}^2+\|\nabla u\|_{L^2(\mathbb{R}^3)}^2+\|xu\|_{L^2(\mathbb{R}^3)}^2\right\}^{\frac{1}{2}}.$$

\begin{proposition}[Local well-posedness for the time-dependent 3D NLS \eqref{NLS}]\label{prop: 3D tNLS}\  Let $\kappa=\pm1.$
\begin{enumerate}[$(i)$]
\item (Local well-posedness in $\Sigma$) For initial data $u_{\omega,0}\in \Sigma$, there exist $T>0$ and a unique solution $u_\omega(t)\in C([0,T]; \Sigma)\cap L_t^2([0,T]; W^{1,6})$ to the integral equation \eqref{3D NLS; integral form}.
\item (Conservation laws)  $u_\omega(t)$ preserves the mass $M[u_\omega]=\|u\|_{L^2(\mathbb{R}^3)}^2$ and the energy
$$E_\omega[u_\omega]= \frac{\omega}{2}\langle ( H_\omega-\Lambda_\omega) u_\omega, u_\omega\rangle_{L^2(\mathbb{R}^3)}+\frac{\kappa\sqrt{\omega}}{4}\int_{\R^3}|u_\omega|^{4}dx.$$
\item (Blow-up criteria) Let $T_\textup{max}\in (0,\infty]$ be the maximal existence time, i.e., the supremum of $T>0$ in $(i)$. If $T_\textup{max}<\infty$, then
$$\lim_{t\to T_\textup{max}^-}\langle (H_\omega-\Lambda_\omega) u_\omega(t), u_\omega(t)\rangle_{L^2(\mathbb{R}^3)}=\infty.$$
\item (Global existence; defocusing case) If $\kappa=1$, then $T_\textup{max}=\infty$ and 
$$\sup_{t\geq 0}\langle ( H_\omega-\Lambda_\omega) u_\omega (t), u_\omega(t)\rangle_{L^2(\mathbb{R}^3)}\lesssim\frac{1}{\omega}E_\omega[u_{\omega,0}].$$
\end{enumerate}

\end{proposition}

\begin{proof}
By a standard contraction mapping argument, $(i)$ and $(ii)$ follow. Moreover, if $T_\textup{max}<\infty$, then $\|u_\omega(t)\|_{\Sigma}\to \infty$ as $t\to T_\textup{max}^-$ (see \cite{C} for instance). Hence, together with the mass conservation and the assumption on the potential $U_\omega$, we prove $(iii)$. For $(iv)$, the energy conservation yields that if $\kappa=1$, then $\langle ( H_\omega-\Lambda_\omega) u_\omega (t), u_\omega(t)\rangle_{L^2(\mathbb{R}^3)}\leq\frac{2}{\omega}E_\omega[u_\omega(t)]=\frac{2}{\omega}E_\omega[u_{\omega,0}]$, and thus $(iii)$ implies global existence.
\end{proof}

In the focusing case $(\kappa=-1)$, the 3D NLS \eqref{NLS} is not globally well-posed; it admits finite blow-up solutions. Nevertheless, the following proposition asserts that solutions exist for all time under some seemingly weak additional constraint \eqref{additional constraint; initial data}.

\begin{proposition}[A criteria for global existence; focusing case]\label{global existence; focusing case} Let $\kappa=-1$ and $\omega \geq 1$ be sufficiently large. Then, for fixed $m, E>0$, there exist small $\delta=\delta(m, E)>0$ and large $K=K(m, E)>0$, independent of large $\omega\geq1$, such that if $u_{\omega,0}\in \Sigma$, $M[u_{\omega,0}]=m$,
\begin{equation}\label{additional constraint; initial data}
E_\omega[u_{\omega,0}]\leq E\quad\textup{and}\quad \langle (H_\omega^\perp-\Lambda_\omega) u_{\omega,0}, u_{\omega,0}\rangle_{L^2(\mathbb{R}^3)}\leq \delta\omega,
\end{equation}
where $H_\omega^\perp:=-\partial_r^2-\frac{1}{r}\partial_r-\partial_z^2+U(|y|-\sqrt{\omega})+z^2$ with $r=|y|$, then the solution $u_\omega(t)$ to the 3D NLS \eqref{NLS} with $u_\omega(0)=u_{\omega,0}$ exists globally in time, and
$$\sup_{t\geq 0}\langle ( H_\omega-\Lambda_\omega) u_\omega(t), u_\omega(t)\rangle_{L^2(\mathbb{R}^3)}\leq K\omega^{-1}.$$
\end{proposition}

\begin{proof}
For the proof, a key observation is that there is an area in the function space $\Sigma_\omega$ prohibited by the refined Gagliardo-Nirenberg inequality (Corollary \ref{GN1}). Indeed, by the assumption and Corollary \ref{forbidden region} (see Remark \ref{remark: refined constraint}), the initial data $u_{\omega,0}$ satisfies the stronger bound $\langle ( H_\omega-\Lambda_\omega) u_{\omega,0}, u_{\omega,0}\rangle_{L^2(\mathbb{R}^3)}\leq  K\omega^{-1}$. Hence, by continuity (in $\Sigma$) of the nonlinear solution $u_\omega(t)$, we have $\langle (H_\omega^\perp-\Lambda_\omega) u_{\omega}(t), u_{\omega}(t)\rangle_{L^2(\mathbb{R}^3)}\leq  \delta\omega$ at least on a short time interval. Then, by the conservation laws and Corollary \ref{forbidden region}, this bound is immediately improved to $\langle ( H_\omega-\Lambda_\omega) u_{\omega}(t), u_{\omega}(t)\rangle_{L^2(\mathbb{R}^3)}\leq  K\omega^{-1}$. Therefore, by iteration, we conclude that the same bound holds for all $t\geq0$, and $u_\omega(t)$ never blows up.
\end{proof}

\subsection{Global well-posedness for the time-dependent periodic 1D NLS}

Next, following \cite[Section 3]{C}, we summarize a well-posedness result for the periodic 1D NLS,
\begin{equation}\label{1d cubic tNLS; cauchy problem}
\left\{\begin{aligned}
i\partial_tw+\partial_\theta^2w-\frac{\kappa}{2\pi}|w|^2w&=0,\\
w(0)&=w_0\in H^1(\mathbb{S}^1),
\end{aligned}\right.
\end{equation}
where $w = w(t,\theta) : I(\subset \mathbb{R}) \times \mathbb{S}^1 \rightarrow \C$. Note that due to some technical difficulty (see Remark \ref{remark: why weak convergence?} below), the dimension reduction in Theorem \ref{dimension reduction for tNLS} is restricted to a weak-$^*$ sense. For this reason, we need uniqueness of weak solutions, and thus we provide a well-posedness result including weak solutions.

Given $T>0$, we say that 
$$w=w(t,\theta)\in L^\infty([0,T];H^1(\mathbb{S}^1))\cap W^{1,\infty}([0,T]; H^{-1}(\mathbb{S}^1))$$
is a \textit{weak (resp., strong) $H^1(\mathbb{S}^1)$-solution} on the interval $[0,T]$ if $i\partial_tw+\partial_\theta^2w-\frac{\kappa}{2\pi}|w|^2w=0$ holds in $H^{-1}(\mathbb{S}^1)$ for almost every $t\in [0,T]$ (resp., for all $t\in[0,T]$) and $w(0)=w_0$, equivalently
$$w(t)=e^{it\partial_\theta^2}w_0-\frac{i\kappa}{2\pi}\int_0^t e^{i(t-t_1)\partial_\theta^2}(|w|^2w)(t_1)dt_1\quad\textup{in }H^1(\mathbb{S}^1)$$
for almost every $t\in[0,T]$ (resp., for all $t\in[0,T]$) (see \cite[Proposition 3.1.3]{C}).

\begin{proposition}[Global well-posedness for the time-dependent 1D periodic NLS]\label{1D tNLS} \ 
\begin{enumerate}[$(i)$]
\item (Global existence for strong solutions \cite[Theorem 3.5.1 and Corollary 3.5.2]{C}) For $w_0\in H^1(\mathbb{S}^1)$, the initial-value problem \eqref{1d cubic tNLS; cauchy problem} has a strong $H^1(\mathbb{S}^1)$-solution $w(t)$ on the interval $[0,\infty)$, Moreover, the solution $w(t)$ preserves the mass $M[w]=\|w\|_{L^2(\mathbb{S}^1)}^2$ and the energy $E[w]=\frac{1}{2}\|\partial_\theta w\|_{L^2(\mathbb{S}^1)}^2+\frac{\kappa}{8\pi}\|w\|_{L^4(\mathbb{S}^1)}^4$.
\item (Uniqueness of weak solutions \cite[Proposition 4.2.1]{C}) If $w_1(t)$ and $w_2(t)$ are weak $H^1(\mathbb{S}^1)$-solutions to \eqref{1d cubic tNLS; cauchy problem} with initial data $w_0\in H^1(\mathbb{S}^1)$ on $[0,T]$, then $w_1(t)=w_2(t)$ for almost every $t\in[0,T]$.
\end{enumerate}
\end{proposition}

\subsection{Derivation of the time-dependent 1D periodic NLS; Proof of Theorem \ref{dimension reduction for tNLS}}

We prove our first main result. Let $u_\omega(t)$ be a solution to the time-dependent 3D NLS \eqref{NLS} constructed in Proposition \ref{prop: 3D tNLS}. Then, by the assumptions in the theorem and Proposition \ref{prop: 3D tNLS} in the defocusing case/Proposition \ref{global existence; focusing case} in the focusing case, $u_\omega(t)$ exists globally in time. In both cases, we have 
\begin{equation}\label{t higher state estimate1}
\sup_{t\geq 0}\langle ( H_\omega-\Lambda_\omega) u_\omega(t), u_\omega(t)\rangle_{L^2(\mathbb{R}^3)}\lesssim\omega^{-1}.
\end{equation}

\subsubsection{Decomposition}
For dimension reduction, it is convenient to us the transformation
\begin{equation}\label{u(t)-v(t) relation}
v_\omega(t,s,z,\theta)=\omega^{\frac{1}{4}}u_\omega(t,s+\sqrt{\omega}, z,\theta),
\end{equation}
where $u_\omega(t,r,z,\theta)$ denotes, with abuse of notation, $u_\omega(t,x)$ in cylindrical coordinates. By construction, $v_\omega(t)$ is a global solution to the equation 
\begin{equation}\label{tNLS; v}
i\partial_t v_\omega=\omega(\mathcal{H}_{\omega}^{(2D)}-\Lambda_\omega)v_\omega-\frac{1}{\sigma_\omega^2(s)}\partial_\theta^2v_\omega +\kappa|v_\omega|^2v_\omega
\end{equation}
satisfying 
\begin{equation}\label{t higher state estimate; translated}
\sup_{t\geq 0}\|v_\omega(t)\|_{\dot{\Sigma}_\omega}\lesssim\omega^{-\frac{1}{2}}.
\end{equation}
Let 
\begin{equation}\label{w_omega(t) definition}
v_{\omega,\parallel}(t,\theta)=\langle v_\omega(t,s,z,\theta),\tilde{\Phi}_\omega(s,z)\rangle_{L^2(\sigma_\omega dsdz)}
\end{equation}
be the $\tilde{\Phi}_\omega$-directional component of $v_\omega$, where $\tilde{\Phi}_\omega=\frac{\chi_\omega(s)\Phi_\omega(s,z)}{\|\chi_\omega(s)\Phi_\omega(s,z)\|_{L^2(\sigma_\omega dsdz)}}$ is the truncated 2D lowest eigenfunction and $\chi_\omega$ is the smooth cut-off such that $\chi_\omega\equiv1$ on $[-\sqrt{\omega}+2,\infty)$ and $\textup{supp}\chi_\omega\subset[-\sqrt{\omega}+1,\infty)$ (see Section \ref{subsec: truncated projection}). We decompose
\begin{equation}\label{v(t) decomposition}
v_\omega(t,s,z,\theta)=v_{\omega,\parallel}(t,\theta)\chi_\omega(s)\Phi_0(s,z)+r_\omega(t,s,z,\theta),
\end{equation}
where the remainder is given by
\begin{equation}\label{v(t) remainder}
r_\omega(t,s,z,\theta)=v_{\omega,\parallel}(t,\theta)\big(\tilde{\Phi}_\omega(s,z)-\chi_\omega(s)\Phi_0(s,z)\big)+(\mathcal{P}_{\tilde{\Phi}_\omega}^\perp v_\omega)(t,s,z,\theta).
\end{equation}
For the core part, we observe from the 3D equation \eqref{tNLS; v} that $v_{\omega,\parallel}$ solves 
\begin{equation}\label{w_omega(t) equation}
\begin{aligned}
i\partial_tv_{\omega,\parallel}&=\omega\langle (\mathcal{H}_\omega^{(2D)}-\Lambda_\omega)v_\omega, \tilde{\Phi}_\omega \rangle_{L^2(\sigma_\omega dsdz)}-\langle \partial_\theta^2v_\omega, \tilde{\Phi}_\omega \rangle_{L^2(\frac{1}{\sigma_\omega} dsdz)}\\
&\quad+\kappa \langle |v_\omega|^2v_\omega, \tilde{\Phi}_\omega \rangle_{L^2(\sigma_\omega dsdz)}.
\end{aligned}
\end{equation}
Our goal is then to show that $v_{\omega,\parallel}$ converges to a solution to the 1D periodic NLS \eqref{1d cubic tNLS; cauchy problem}, and that the remainder vanishes as $\omega\to\infty$.

\begin{remark}\label{remark: why weak convergence?}
For the derivation of the 1D periodic NLS, we need to extract the linear term $\partial_\theta^2v_{\omega,\parallel}$ from $\langle \partial_\theta^2v_\omega, \tilde{\Phi}_\omega \rangle_{L^2(\frac{1}{\sigma_\omega} dsdz)}$ by showing that as $\omega\to\infty$, 
$$\langle \partial_\theta^2 v_\omega, \tilde{\Phi}_\omega \rangle_{L^2(\frac{1}{\sigma_\omega} dsdz)}-\partial_\theta^2v_{\omega,\parallel}= \langle \partial_\theta^2 v_\omega, (1-\sigma_\omega^2)\tilde{\Phi}_\omega \rangle_{L^2(\frac{1}{\sigma_\omega} dsdz)}\to 0.$$
Indeed, we have $(1-\sigma_\omega^2)\tilde{\Phi}_\omega=\frac{s}{\sqrt{\omega}}(2+\frac{s}{\sqrt{\omega}})\tilde{\Phi}_\omega\to 0$ in a suitable norm. On the other hand, if one wishes to obtain a strong convergence, one needs to control $\partial_\theta^2 v_\omega$. However, the 3D problem \eqref{NLS} has no conservation law controlling higher-order derivative norms, and thus it does not seem possible to obtain a global-in-time strong convergence. To avoid this difficulty, we consider weak convergence.
\end{remark}

\subsubsection{Uniform bounds on the flow $v_{\omega,\parallel}(t)$}

We aim to show that $v_{\omega,\parallel}(t)$ converges weak-$^*$ly to a solution to the periodic NLS. However, before proving this, we need to confirm that $\|v_{\omega,\parallel}(t)\|_{C_t([0,T];H^1(\mathbb{S}^1))}$ and $\|\partial_t v_{\omega,\parallel}(t)\|_{C_t([0,T];H^{-1}(\mathbb{S}^1))}$ are uniformly bounded in $\omega\geq\omega_*\gg1$, where $\|w\|_{H^{-1}(\mathbb{S}^1)}:=\underset{g\in H^1(\mathbb{S}^1)}\sup|\langle g, w\rangle_{L^2(\mathbb{S}^1)}|$; then the Banach-Alaoglu theorem can be applied.
 
For $\|v_{\omega,\parallel}(t)\|_{H^1(\mathbb{S}^1)}$, we observe that by the change of variables \eqref{u(t)-v(t) relation} and the mass conservation law, $\|v_\omega(t)\|_{L^2(\sigma_\omega)}=\|u_\omega(t)\|_{L^2(\mathbb{R}^3)}=\|u_0\|_{L^2(\mathbb{R}^3)}=\sqrt{m}$ for all $t\geq 0$. On the other hand, recalling the definition of the semi-norm $\|\cdot\|_{\dot{\Sigma}_\omega}$ (see \eqref{seminorm2}), we see that the global bound \eqref{t higher state estimate; translated} implies a uniform bound
\begin{equation}\label{v_omega(t) derivative bound}
\sup_{\omega\geq\omega_*}\|\partial_\theta v_\omega(t)\|_{L^2(\frac{1}{\sigma_\omega})}\lesssim 1
\end{equation} for all $t\geq 0$. Hence, $v_{\omega,\parallel}(t,\theta)=\langle v_\omega(t,\cdot,\cdot,\theta),\tilde{\Phi}_\omega\rangle_{L^2(\sigma_\omega dsdz)}$ satisfies
 $$\|v_{\omega,\parallel}(t)\|_{L^2(\mathbb{S}^1)}\leq \|v_\omega(t)\|_{L^2(\sigma_\omega)}=\sqrt{m}$$
and 
 $$\|\partial_\theta v_{\omega,\parallel}(t)\|_{L^2(\mathbb{S}^1)}\leq \left\|\|\partial_\theta v_\omega(t)\|_{L^2(\frac{1}{\sigma_\omega}dsdz)}\|\sigma_\omega\tilde{\Phi}_\omega\|_{L^2(\sigma_\omega)}\right\|_{L^2(\mathbb{S}^1)}\lesssim\|\partial_\theta v_\omega(t)\|_{L^2(\frac{1}{\sigma_\omega})} \lesssim1,$$
because $\Phi_\omega$ is a rapidly decreasing function. Thus, we conclude that for all $t\geq 0$,
\begin{equation}\label{w(t) H1 uniform bound}
\underset{\omega\geq\omega_*}\sup\|v_{\omega,\parallel}(t)\|_{H^1(\mathbb{S}^1)}\lesssim 1.
\end{equation}

For a uniform bound on $\|\partial_t v_{\omega,\parallel}(t)\|_{H^{-1}(\mathbb{S}^1)}$, it suffices to show that all terms on the right hand side of the equation \eqref{w_omega(t) equation} are bounded in $H^{-1}(\mathbb{S}^1)$. First, we claim that for all $t\geq0$, 
\begin{equation}\label{w_omega(t) equation; first term}
\lim_{\omega\to \infty}\big\|\omega\langle (\mathcal{H}_\omega^{(2D)}-\Lambda_\omega)v_\omega(t), \tilde{\Phi}_\omega \rangle_{L^2(\sigma_\omega dsdz)}\big\|_{L^2(\mathbb{S}^1)}=0.
\end{equation}
Indeed, it follows from $(\mathcal{H}_\omega^{(2D)}-\Lambda_\omega)\Phi_\omega=0$, with $\tilde{\Phi}_\omega(s,z)=\chi_\omega(s)\Phi_\omega(s,z)$, that 
$$\begin{aligned}
&\big\|\langle \omega(\mathcal{H}_\omega^{(2D)}-\Lambda_\omega)v_\omega, \tilde{\Phi}_\omega \rangle_{L^2(\sigma_\omega dsdz)}\big\|_{L^2(\mathbb{S}^1)}\\
&=\omega\big\|\big\langle v_\omega, (-\chi_\omega''\Phi_\omega-2\chi_\omega'\partial_s\Phi_\omega-\tfrac{1}{\sqrt{\omega}\sigma_\omega}\chi_\omega'\Phi_\omega)\big\rangle_{L^2(\sigma_\omega dsdz)}\big\|_{L^2(\mathbb{S}^1)}\\
&\leq \|v_\omega\|_{L^2(\sigma_\omega)} \omega\big\|\chi_\omega''\Phi_\omega+2\chi_\omega'\partial_s\Phi_\omega+\tfrac{1}{\sqrt{\omega}\sigma_\omega}\chi_\omega'\Phi_\omega \big\|_{L^2(\sigma_\omega dsdz)}.
\end{aligned}$$
However, since $\Phi_\omega$ is rapidly decreasing and $\chi'_\omega(s)$ and $\chi''_\omega(s)$ are supported only on $[-\sqrt{\omega}+1,-\sqrt{\omega}]$, we have
$$\big\|\langle \omega(\mathcal{H}_\omega^{(2D)}-\Lambda_\omega)v_\omega(t), \tilde{\Phi}_\omega \rangle_{L^2(\sigma_\omega dsdz)}\big\|_{L^2(\mathbb{S}^1)}\leq \sqrt{m}o_\omega(1)\to 0.$$
Moreover, we have
$$\begin{aligned}
\big\|\langle \partial_\theta^2v_\omega(t), \tilde{\Phi}_\omega \rangle_{L^2(\frac{1}{\sigma_\omega} dsdz)}\big\|_{H^{-1}(\mathbb{S}^1)}&\lesssim \big\|\langle \partial_\theta v_\omega(t), \tilde{\Phi}_\omega \rangle_{L^2(\frac{1}{\sigma_\omega} dsdz)}\big\|_{L^2(\mathbb{S}^1)}\\
&\leq \|\partial_\theta v_\omega(t)\|_{L^2(\frac{1}{\sigma_\omega})}\|\tilde{\Phi}_\omega \|_{L^2(\frac{1}{\sigma_\omega} dsdz)}\lesssim1\quad\textup{(by \eqref{v_omega(t) derivative bound})}.
\end{aligned}$$
On the other hand, the energy conservation law $\mathcal{E}_\omega[v_\omega(t)]=\mathcal{E}_\omega[v_\omega(0)]$ (equivalent to $E_\omega[u_\omega(t)]=E_\omega[u_{\omega,0}]$) and \eqref{t higher state estimate; translated} yield 
$$\sup_{\omega\geq\omega_*}\|v_\omega(t)\|_{L^4(\sigma_\omega)}\lesssim1$$
for all $t\geq 0$. Consequently, we obtain
$$\begin{aligned}
\big\|\langle |v_\omega|^2v_\omega(t), \tilde{\Phi}_\omega \rangle_{L^2(\sigma_\omega dsdz)}\big\|_{H^{-1}(\mathbb{S}^1)}&\leq\Big\|\||v_\omega(t)|^3\|_{L^{\frac{4}{3}}(\sigma_\omega dsdz)}\|\tilde{\Phi}_\omega\|_{L^4(\sigma_\omega dsdz)}\Big\|_{L^{\frac{4}{3}}(\mathbb{S}^1)}\\
&\lesssim \|v_\omega(t)\|_{L^4(\sigma_\omega)}^3\lesssim1.
\end{aligned}$$
Therefore, collecting all, by the equation \eqref{w_omega(t) equation}, we conclude that for all $t\geq 0$, 
$$\underset{\omega\geq\omega_*}\sup\|\partial_t v_{\omega,\parallel}(t)\|_{H^{-1}(\mathbb{S}^1)}\lesssim 1.$$

\subsubsection{Dimension reduction: Proof of Theorem \ref{dimension reduction for tNLS} $(i)$}

For the dimension reduction, it suffices to show that for all $t\geq0$,
\begin{equation}\label{t dimension reduction claim}
\big\|v_{\omega,\parallel}(t,\theta)\big(\tilde{\Phi}_\omega(s,z)-\chi_\omega(s)\Phi_0(s,z)\big)\big\|_{L^2(\sigma_\omega)}+\|(\mathcal{P}_{\tilde{\Phi}_\omega}^\perp v_\omega)(t)\|_{L^2(\sigma_\omega)}\lesssim\omega^{-\frac{1}{2}},
\end{equation}
because  
$$
\left\| (1-\chi(|y|))v_{\omega,\parallel}(t,\theta) \omega^{-1/4}\Phi_0(|y|-\sqrt{\omega},z) \right\|_{L^2(\mathbb{R}^3)}\lesssim \omega^{-\frac12}
$$ 
and by the decomposition \eqref{v(t) remainder} and the change of variables by \eqref{u(t)-v(t) relation}, \eqref{t dimension reduction claim} implies 
$$\left\|u_\omega(t,x)-v_{\omega,\parallel}(t,\theta)\left(\chi(|y|)\omega^{-1/4}\Phi_0(|y|-\sqrt{\omega},z)\right)\right\|_{L^2(\mathbb{R}^3)}=\|r_\omega(t)\|_{L^2(\sigma_\omega)} .$$
Indeed, since both $\phi_0$ and $\phi_\omega$ are rapidly decreasing and $\|\phi_\omega-\phi_0\|_{L^2(\sigma_\omega ds) }\lesssim\omega^{-\frac{1}{2}}$  (see Proposition \ref{1d eigenstate} and Corollary \ref{2d eigenstate}), by the uniform bound on $\|u_\omega(t)\|_{H^1(\mathbb{S}^1)}$, we prove that 
$$\begin{aligned}
&\big\|v_{\omega,\parallel}(t,\theta)\big(\tilde{\Phi}_\omega(s,z)-\chi_\omega(s)\Phi_0(s,z)\big)\big\|_{L^2(\sigma_\omega)}\\
&=\|v_{\omega,\parallel}(t)\|_{L^2(\mathbb{S}^1)}\|\tilde{\Phi}_\omega(s,z)-\chi_\omega(s)\Phi_0(s,z)\|_{L^2(\sigma_\omega dsdz)} \lesssim \omega^{-\frac{1}{2}}.
\end{aligned}$$
On the other hand, Lemma \ref{truncation error} implies that $\|(\mathcal{P}_{\tilde{\Phi}_\omega}^\perp v_\omega)(t)-(\mathcal{P}_{\Phi_\omega}^\perp v_\omega)(t)\|_{L^2(\sigma_\omega)}\lesssim\omega^{-\frac{1}{2}}$. Moreover, by the spectral gap (Corollary \ref{2d eigenstate}) and \eqref{t higher state estimate; translated}, we have 
$$\begin{aligned}
\|(\mathcal{P}_{\Phi_\omega}^\perp v_\omega)(t)\|_{L^2(\sigma_\omega)}&\lesssim\|(\mathcal{P}_{\Phi_\omega}^\perp v_\omega)(t)\|_{{\dot{\Sigma}}_{\omega;(s,z)}}\le\|v_\omega(t)\|_{{\dot{\Sigma}}_{\omega}}\lesssim\omega^{-\frac{1}{2}}.
\end{aligned}$$

\subsubsection{Initial data}\label{sec: initial data}
Now, we present how the initial data for the 1D NLS is prepared in Theorem \ref{dimension reduction for tNLS} $(ii)$. We consider the decomposition \eqref{v(t) decomposition} at $t=0$. Indeed, by \eqref{t higher state estimate1} and the dimension reduction (Theorem \ref{dimension reduction for tNLS} $(i)$), we have
$\|u_{\omega,0}\|_{H^1(\R^3)}\lesssim 1$ and  
$$\lim_{\omega\to\infty}\left\|u_{\omega,0}(x)-v_{\omega,\parallel}(0,\theta)\left( \omega^{-\frac{1}{4}}\Phi_0(|y|-\sqrt{\omega},z)\right)\right\|_{L^2(\mathbb{R}^3)}=0.$$ 
On the other hand, by the $H^1(\mathbb{S}^1)$-uniform bound \eqref{w(t) H1 uniform bound}, it follows from the Banach-Alaoglu theorem that there exists $\{\omega_j\}_{j=1}^\infty$, with $\omega_j\to\infty$, such that $v_{\omega_j,\parallel}(0,\theta)\rightharpoonup w_{\infty,0}$ in $H^1(\mathbb{S}^1)$. Therefore, we conclude that  as $j\to \infty$, 
$$u_{\omega_j,0}(x)-w_{\infty,0}(\theta)\left(  \omega_j^{-\frac14}\Phi_0(|y|-\sqrt{\omega_j},z)\right)\rightharpoonup 0\quad\textup{in }H^1(\mathbb{R}^3).$$

\subsubsection{Derivation of the 1D NLS}
Fix arbitrary $T>0$, and let $\{v_{\omega,\parallel}(t)\}_{\omega\gg1}$ be a one-parameter family of solutions to \eqref{w_omega(t) equation} with $v_{\omega,\parallel}(0)=\langle v_\omega(0),\tilde{\Phi}_\omega\rangle_{L^2(\sigma_\omega dsdz)}$. We have $v_{\omega_j,\parallel}(0)\rightharpoonup w_{\infty,0}$ in $H^1(\mathbb{S}^1)$ for some $\{\omega_j\}_{j=1}^\infty$ with $\omega_j\to\infty$. On the other hand, by the uniform bounds obtained previously and the Banach-Alaoglu theorem, we have sub-sequential convergences $v_{\omega_{j_k},\parallel} \rightharpoonup^* w_\infty$ in $L^\infty([0,T]; H^1(\mathbb{S}^1))$ and $\partial_t v_{\omega_{j_k},\parallel}\rightharpoonup^* \partial_tw_\infty$ in $L^\infty([0,T]; H^{-1}(\mathbb{S}^1))$ as $j_k\to\infty$. Therefore, by uniqueness of weak solutions (Proposition \ref{1D tNLS}), it suffices to show that the limit $w_\infty(t)$ is a weak $H^1(\mathbb{S}^1)$-solution to the 1D NLS \eqref{1d cubic tNLS; cauchy problem}\footnote{$w_\infty(0)=w_{\infty,0}$, because   $v_{\omega_{j_k},\parallel}$ is a subsequence of $v_{\omega_{j},\parallel}$ and $v_{\omega_j,\parallel}(0)\rightharpoonup w_{\infty,0}$ in $H^1(\mathbb{S}^1)$.}.

To prove it, we fix an arbitrary $g=g(t,\theta)\in L^1([0,T];H^1(\mathbb{S}^1))$, and test $g$ against the equation \eqref{w_omega(t) equation}. Then, by \eqref{w_omega(t) equation} and Fubini's theorem, one can write 
\begin{equation}\label{w_omega(t) equation-tested}
\begin{aligned}
\big\langle g, i\partial_t v_{\omega,\parallel}\big\rangle&=\big\langle g, \omega\langle (\mathcal{H}_\omega^{(2D)}-\Lambda_\omega)v_\omega, \tilde{\Phi}_\omega \rangle_{L^2(\sigma_\omega dsdz)}\big\rangle\\
&\quad+\langle \partial_\theta g, \partial_\theta v_{\omega,\parallel}\rangle+\big\langle  (\partial_\theta g)(1-\sigma_\omega^2)\tilde{\Phi}_\omega, \partial_\theta v_\omega\big\rangle_{L^2([0,T];L^2(\frac{1}{\sigma_\omega}))}\\
&\quad+\kappa\big\langle g \tilde{\Phi}_\omega,  |v_{\omega,\parallel}\chi_\omega\Phi_0|^2(v_{\omega,\parallel}\chi_\omega\Phi_0) \big\rangle_{L^2([0,T];L^2(\sigma_\omega))}\\
&\quad+\kappa\big\langle g \tilde{\Phi}_\omega,  |v_\omega|^2v_\omega -|v_{\omega,\parallel}\chi_\omega\Phi_0|^2(v_{\omega,\parallel}\chi_\omega\Phi_0)\big\rangle_{L^2([0,T];L^2(\sigma_\omega))}\\
&=:I_\omega+II_\omega+III_\omega+\kappa IV_\omega+\kappa V_\omega,
\end{aligned}
\end{equation}
where $\langle\cdot,\cdot\rangle=\langle\cdot,\cdot\rangle_{L^2([0,T];L^2(\mathbb{S}^1))}$. By construction, it is obvious that $\langle g, \partial_tv_{\omega_{j_k},\parallel}\rangle\to\langle g, \partial_tw_\infty\rangle$ and $II_{\omega_{j_k}}\to \langle g, -\partial_\theta^2w_\infty\rangle$. By \eqref{w_omega(t) equation; first term}, $I_\omega\to 0$. We will show that $III_\omega, V_\omega\to 0$ and 
$$IV_{\omega_{j_k}}\to\big\langle g, \tfrac{1}{2\pi}|w_\infty|^2w_\infty\big\rangle.$$
Then, we conclude that $w_\infty(t)$ is a weak solution, so the proof is completed.

For $III_\omega$, by the H\"older inequality, we obtain that 
$$\begin{aligned}
|III_\omega|&=\big|\big\langle  (\partial_\theta g)(1-\sigma_\omega^2)\tilde{\Phi}_\omega, \partial_\theta v_\omega\big\rangle_{L^2([0,T];L^2(\frac{1}{\sigma_\omega}))}\big|\\
&\leq \|\partial_\theta g\|_{L^1([0,T];L^2(\mathbb{S}^1))}\|(1-\sigma_\omega^2)\tilde{\Phi}_\omega\|_{L^2(\frac{1}{\sigma_\omega} dsdz)}\|\partial_\theta v_\omega\|_{L^\infty([0,T];L^2(\frac{1}{\sigma_\omega}))}.
\end{aligned}$$
Because $1-\sigma_\omega(s)=-\frac{s}{\sqrt{\omega}}$ and $\Phi_\omega$ is rapidly decreasing, $\|(1-\sigma_\omega^2)\tilde{\Phi}_\omega\|_{L^2(\frac{1}{\sigma_\omega} dsdz)}\to0$. By \eqref{v_omega(t) derivative bound}, $\|\partial_\theta v_\omega(t)\|_{L^2(\frac{1}{\sigma_\omega})}\lesssim 1$. Thus, it follows that $III_\omega\to0$. For $IV_\omega$, by the convergence $\Phi_\omega\to \Phi_0$ in $L^2(\sigma_\omega dsdz)$ with their Gaussian decay and $\|\Phi_0\|_{L^4(\mathbb{R}^2)}^4=\frac{1}{2\pi}$, we write
$$IV_\omega=\big\langle \tilde{\Phi}_\omega,  |\chi_\omega\Phi_0|^2(\chi_\omega\Phi_0) \big\rangle_{L^2(\sigma_\omega dsdz)} \big\langle g, |v_{\omega,\parallel}|^2v_{\omega,\parallel}\big\rangle=  \big\langle g, \tfrac{1}{2\pi}|v_{\omega,\parallel}|^2v_{\omega,\parallel}\big\rangle+o_\omega(1).$$
Then, the Rellich–Kondrachov theorem $H^1(\mathbb{S}^1)\hookrightarrow L^4(\mathbb{S}^1)$ implies $IV_{\omega_{j_k}}\to\big\langle g, \tfrac{1}{2\pi}|w_\infty|^2w_\infty\big\rangle$. For the convergence $V_\omega\to 0$, by the H\"older inequality and the decomposition \eqref{v(t) decomposition}, it suffices to show that $\|r_\omega\|_{L^4(\sigma_\omega)}\to 0$, where $r_\omega=v_{\omega,\parallel}(\tilde{\Phi}_\omega-\chi_\omega\Phi_0)+\mathcal{P}_{\tilde{\Phi}_\omega}^\perp v_\omega$ (see \eqref{v(t) remainder}). Indeed, $\|v_{\omega,\parallel}(\tilde{\Phi}_\omega-\chi_\omega\Phi_0)\|_{L^4(\sigma_\omega)}\leq\|v_{\omega,\parallel}\|_{L^4(\mathbb{S}^1)}\|\tilde{\Phi}_\omega-\chi_\omega\Phi_0\|_{L^4(\sigma_\omega dsdz)} \lesssim\|v_{\omega,\parallel}\|_{H^1(\mathbb{S}^1)}\|\tilde{\Phi}_\omega-\chi_\omega\Phi_0\|_{L^4(\sigma_\omega dsdz)} \to0$. On the other hand, by the refined Gagliardo-Nirenberg inequality (Corollary \ref{GN1}), we can show that $\|\mathcal{P}_{\tilde{\Phi}_\omega}^\perp v_\omega\|_{L^4(\sigma_\omega)}\to 0$, because 
$$
\|\partial_\theta\mathcal{P}_{\tilde{\Phi}_\omega}^\perp v_\omega\|_{L^2(\frac{1}{\sigma_\omega})}\lesssim \|\partial_\theta  v_\omega\|_{L^2(\frac{1}{\sigma_\omega})}\le \sqrt{\omega}\|v_\omega\|_{\dot{\Sigma}_\omega}\lesssim 1 \ \ (\textup{by } \eqref{asymptotic Pythagorean identity 3} \textup{ and } \eqref{t higher state estimate; translated}),
$$
\begin{align*}
\|\mathcal{P}_{\tilde{\Phi}_\omega}^\perp v_\omega\|_{L^2(\sigma_\omega)}, \|\mathcal{P}_{\tilde{\Phi}_\omega}^\perp v_\omega\|_{\dot{\Sigma}_{\omega;(s,z)}}&\lesssim e^{-c\omega}\| v_\omega\|_{L^2(\sigma_\omega)}+\|\mathcal{P}_{\Phi_\omega}^\perp v_\omega\|_{\dot{\Sigma}_{\omega;(s,z)}}\\
&\lesssim  e^{-c\omega}+\|v_\omega(t)\|_{\dot{\Sigma}_\omega}\lesssim\omega^{-\frac{1}{2}}\ \ (\textup{by } \textup{Lemma  } \ref{orthogonal complement bound} \textup{ and } \eqref{t higher state estimate; translated}).
\end{align*}
Therefore, it follows that $\|r_\omega\|_{L^4(\sigma_\omega)}\to 0$.

\section{Existence of a minimizer and its dimension reduction limit: Proof of Theorem \ref{thm: existence} and \ref{drth}}\label{sec: existence}
We consider the modified nonlinear minimization problem 
\begin{equation}\label{modified variational problem'}
\mathcal{J}_\omega^{(3D)}(m):=\min\left\{\mathcal{E}_\omega(v): \  v\in\Sigma_\omega,\ \mathcal{M}_\omega(v)=m\textup{ and } \|v\|_{{\dot{\Sigma}}_{\omega;(s,z)}}\leq\delta\sqrt{\omega} \right\}.
\end{equation}
In this section, we prove 1) existence of a minimizer (Proposition \ref{existence of a minimizer}); 2) its uniform decay/bound  properties (Proposition \ref{Gaussian L^2 bounds} and Proposition \ref{uniform angle derivative bounds}); 3) the dimension reduction to the 1D periodic ground state (Proposition \ref{adprop}). 

\subsection{Construction of a minimizer}

By a concentration-compactness argument, we establish existence of  a minimizer for the variational problem $\mathcal{J}_\omega^{(3D)}(m)$,   but we also present some basic properties directly coming from its construction. 

\begin{proposition}[Existence of a minimizer and its preliminary properties]\label{existence of a minimizer}
For sufficiently large $\omega\geq 1$, the following hold.
\begin{enumerate}[(i)]
\item The minimization problem $\mathcal{J}_\omega^{(3D)}(m)$ occupies a positive minimizer, denoted by $\mathcal{Q}_\omega$.
\item A minimizer $\mathcal{Q}_\omega$ solves the Euler-Lagrange equation 
\begin{equation}\label{elp}
\omega( \mathcal{H}_\omega^{(2D)}-\Lambda_\omega)\mathcal{Q}_\omega-\frac{1}{\sigma_\omega^2}\partial_\theta^2 \mathcal{Q}_\omega- |\mathcal{Q}_\omega|^2\mathcal{Q}_\omega=- \mu_\omega \mathcal{Q}_\omega,
\end{equation}
where $\mu_\omega\in \R$ is a Lagrange multiplier.
Thus, it satisfies 
\begin{equation}\label{po1}
\omega \|\mathcal{Q}_\omega\|_{{\dot{\Sigma}}_{\omega;(s,z)}}^2+\|\partial_\theta\mathcal{Q}_\omega\|_{L^2(\frac{1}{\sigma_\omega})}^2- \|\mathcal{Q}_\omega\|_{L^4(\sigma_\omega)}^4+\mu_\omega  \|\mathcal{Q}_\omega\|_{L^2(\sigma_\omega)}^2=0.
\end{equation}
\item $\sqrt{\omega}\|\mathcal{Q}_\omega\|_{{\dot{\Sigma}}_{\omega;(s,z)}}$, $\|\partial_\theta\mathcal{Q}_\omega\|_{L^2(\frac{1}{\sigma_\omega})}$, $\|\mathcal{Q}_\omega\|_{L^6(\sigma_\omega)}$ and $|\mu_\omega|$ are  bounded uniformly in $\omega$.
\end{enumerate}
\end{proposition}

\begin{remark}
As mentioned in the introduction, the energy minimization only under a mass constraint is super-critical in that the energy is not bounded from below, but it is also considered partially sub-critical in a certain regime $(\omega\gg1)$. To capture the sub-critical nature of the problem, the additional constraint $\|v\|_{{\dot{\Sigma}}_{\omega;(s,z)}}\leq\delta\sqrt{\omega}$ is imposed, and we analyze the energy minimization  problem using the refined Gagliardo-Nirenberg inequality (Corollary \ref{GN1}). We refer to \cite{BBJV, LY, HJ1, HJ2} for similar settings.
\end{remark}

For the proof, we first show a uniform negative upper bound on the lowest energy. 
 
\begin{lemma}[Negative minimum energy]\label{upper}
For sufficiently large $\omega\geq1$, we have
$$\mathcal{J}_\omega^{(3D)}(m)\le J_\infty^{(1D)}(m)+O(\omega^{-\frac12})<0.$$
\end{lemma}

\begin{proof}
We employ $\chi_\omega\Phi_\omega Q_\infty=\chi_\omega(s)\Phi_\omega(s,z) Q_\infty(\theta)$, where $ Q_\infty$ is the ground state for 1D circle problem $J_\infty^{(1D)}(m)$, $\Phi(s,z)=\phi_\omega(s)\phi_\infty(z)$, $\chi_\omega=\chi(\cdot+\sqrt{\omega})$ and $\chi:[0,\infty)\to[0,1]$ is given by $\chi\equiv0$ on $[0,1]$ and $\chi\equiv1$ on $[2,\infty)$. Then, it follows from the fast decay of $\phi_\omega$ (Proposition \ref{1d eigenstate} $(ii)$) and $(\mathcal{H}_\omega^{(2D)}-\Lambda_\omega)\Phi_\omega=0$ that $\|\chi_\omega\Phi_\omega Q_\infty\|_{L^2(\sigma_\omega)}^2=m+O(\omega^{-\frac{1}{2}})$, $\|\chi_\omega\Phi_\omega Q_\infty\|_{{\dot{\Sigma}}_{\omega; (s,z)}}=O(\omega^{-\frac{1}{2}})$ and $\mathcal{E}_\omega(\chi_\omega\Phi_\omega Q_\infty)=\mathcal{E}_\omega(\Phi_\omega Q_\infty)+O(\omega^{-\frac12})=\mathcal{E}_\infty( Q_\infty)+O(\omega^{-\frac12})$. Therefore, we conclude that $ \mathcal{J}_\omega^{(3D)}(m)\leq \mathcal{E}_\omega(\frac{\chi_\omega\Phi_\omega Q_\infty}{\|\chi_\omega\Phi_\omega Q_\infty\|_{L^2(\sigma_\omega)}})=E_\infty( Q_\infty)+O(\omega^{-\frac12})=J_\infty^{(1D)}(m)+O(\omega^{-\frac12})$. In addition, using the constant function $\sqrt{\frac{m}{2\pi}}$, we get $J_\infty^{(1D)}(m)<0$ (refer to \cite[Proposition 3.2]{GLT}).
 \end{proof}

\begin{remark}\label{remark: refined constraint1} 
 Lemma \ref{upper} assures that Corollary \ref{forbidden region} can be applied to a minimizing sequence for $\mathcal{J}_\omega^{(3D)}(m)$. Consequently, the minimum energy level is finite, and the  constraint $\|v\|_{\dot{\Sigma}_{\omega;(s,z)}}\leq\delta\sqrt{\omega}$ is immediately strengthened to $\|v\|_{\dot{\Sigma}_\omega}\leq\sqrt{K(m)}\omega^{-\frac{1}{2}}$. Therefore, a minimizer (if it exists) does not meet the boundary $\|v\|_{\dot{\Sigma}_{\omega;(s,z)}}=\delta\sqrt{\omega}$ of the constraint so that the Euler-Lagrange equation can be expressed with one multiplier rather than two.
\end{remark}
 
\begin{proof}[Proof of Proposition \ref{existence of a minimizer}]
Suppose that $\omega\geq1$ is sufficiently large. Let $\{v_n\}_{n=1}^\infty$ be  a minimizing sequence for  $\mathcal{J}_\omega^{(3D)}(m)$, that is, $\mathcal{E}_\omega(v_n)\to\mathcal{J}_\omega^{(3D)}(m)$ as $n\to \infty$. Then, Lemmas \ref{upper} and \ref{forbidden region} imply that $\|v_n\|_{ \Sigma_\omega}$ is bounded  uniformly in $n$ (see \eqref{norm} for the definition of the norm $\|\cdot\|_{ \Sigma_\omega}$), and thus we may assume that $v_n\rightharpoonup v$ in $\Sigma_\omega$ as $n\rightarrow \infty$.

We claim that 
\begin{equation}\label{mnoo2}
\lim_{n\to\infty}\|v_n-v\|_{L^2(\sigma_\omega)}=0.
\end{equation}
Indeed, by the assumption on $U_\omega$ \textbf{\textup{(H2)}}, for any $\epsilon>0$, there exists $R_\epsilon>0$, independent of $n$, such that
\begin{align*}
&\int_{\mathbb{S}^1}\iint_{([-\sqrt{\omega},\infty)\times \R )\setminus [-R_\epsilon,R_\epsilon]^2}   |v_n-v|^2 \sigma_\omega(s)dsdzd\theta\\
&\leq \epsilon\int_{\mathbb{S}^1}\int_{-\infty}^\infty\int_{-\sqrt{\omega}}^\infty  (U_\omega(s)+z^2)|v_n-v|^2 \sigma_\omega(s)dsdzd\theta\\
&\leq \epsilon \|v_n-v\|_{\Sigma_{\omega}}\lesssim\epsilon \sup_{n\in \mathbb{N}}\|v_n\|_{\Sigma_{\omega}},
\end{align*}
where the implicit constant is independent of $n$. On the other hand, it follows from the Rellich-Kondrachov compactness theorem\footnote{By $v_n(s,z,\theta)=\omega^{\frac{1}{4}}u_n(s+\sqrt{\omega},z, \theta)$ (see \eqref{u-v relation}), $\|u_n\|_{H^1(\R^3)}\le \|v_n\|_{ \Sigma_\omega}.$ Then we have $u_n\to u$ in $L^2_{loc}(\R^3)$, and this implies that $\|\mathbf{1}_{(s,z)\in [-R_\epsilon,R_\epsilon]^2\cap ([-\sqrt{\omega},\infty)\times \R)}(v_n-v)\|_{L^2(\sigma_\omega)}\to0$.}  that 
$$\|\mathbf{1}_{(s,z)\in [-R_\epsilon,R_\epsilon]^2\cap ([-\sqrt{\omega},\infty)\times \R)}(v_n-v)\|_{L^2(\sigma_\omega)}\to0.$$
 Thus, the claim \eqref{mnoo2} follows.

By the claim \eqref{mnoo2}, we have $\|v\|_{L^2(\sigma_\omega)}^2=m$. Moreover, the Gagliardo-Nirenberg inequality (Corollary \ref{GN1}) implies $v_n\to v$ in $ L^4(\sigma_\omega)$. Hence, it follows from the weak convergence $v_n\rightharpoonup v$ in $\Sigma_\omega$ that 
\begin{align*}
 \mathcal{J}_\omega^{(3D)}(m)+o_n(1)&= \mathcal{E}_\omega(v_n)= \frac{\omega}{2}\|v_n\|_{{\dot{\Sigma}}_{\omega}}^2-\frac{1}{4}\|v_n\|_{L^4(\sigma_\omega)}^4\\
&= \mathcal{E}_\omega(v)+ \mathcal{E}_\omega(v_n-v)  +o_n(1)\\
&\ge \mathcal{J}_\omega^{(3D)}(m)+ \frac{\omega}{2}\|v_n-v\|_{{\dot{\Sigma}}_{\omega}}^2+o_n(1).
\end{align*}
Therefore, we conclude that $\|v_n-v\|_{ \Sigma_\omega}\to 0$ and the limit $v$ is a minimizer.

We denote a minimizer $v$ by $\mathcal{Q}_\omega$. Then, by Remark \ref{remark: refined constraint1} and integration by parts, one can derive the Euler-Lagrange equation \eqref{elp} and the identity \eqref{po1}. 
   On the other hand,  we recall that by the relation \eqref{u-v relation}, $R_\omega(x)=R_\omega(r,z,\theta)=\omega^{-\frac14}\mathcal{Q}_\omega(r-\sqrt{\omega},z,\theta)$ solves the equivalent nonlinear elliptic equation 
\begin{equation}\label{requ}
 (-\Delta_x+U(|y|-\sqrt{\omega})+z^2-\Lambda_\omega) R_\omega-\omega^{-\frac12}|R_\omega|^2R_\omega=-\frac{\mu_\omega}{\omega} R_\omega
\textup{ in } \R^3.
\end{equation}
Then we may assume that $\mathcal{Q}_\omega$ is non-negative, because $ \mathcal{E}_\omega(\mathcal{Q}_\omega)\geq\mathcal{E}_\omega(|\mathcal{Q}_\omega|)$.  Thus, the strong maximum principle can be applied to $R_\omega$ to prove that $\mathcal{Q}_\omega$ is positive.

For $(iii)$, we observe from Corollary \ref{forbidden region}, Lemma \ref{upper} and the refined Gagliardo-Nirenberg inequality (Corollary \ref{GN1}) that $\sqrt{\omega}\|\mathcal{Q}_\omega\|_{{\dot{\Sigma}}_\omega}$ and $\|\mathcal{Q}_\omega\|_{L^4(\sigma_\omega)}$ are uniformly bounded, and so is $|\mu_\omega|$ by the identity \eqref{po1}. It remains to estimate $\|\mathcal{Q}_\omega\|_{L^6(\sigma_\omega)}$.  We will estimate $ \mathcal{P}_{\tilde{\Phi}_\omega} \mathcal{Q}_\omega$ and $\mathcal{P}_{\tilde{\Phi}_\omega}^\perp\mathcal{Q}_\omega$ separately. Indeed, it is obvious that 
$$\|\mathcal{P}_{\tilde{\Phi}_\omega} \mathcal{Q}_\omega\|_{L^6(\sigma_\omega)}= \|\mathcal{Q}_{\omega,\parallel}\|_{L^6(\mathbb{S}^1)}\|\Phi_{\omega} \chi_\omega \|_{L^2(\sigma_\omega dsdz)}^{-1}\|\Phi_{\omega} \chi_\omega \|_{L^6(\sigma_\omega dsdz)}\lesssim \|\mathcal{Q}_{\omega,\parallel}\|_{L^6(\mathbb{S}^1)}.$$
Hence, it follows from the Gagliardo-Nirenberg inequality on $\mathbb{S}^1$  \eqref{circle GN ineq}, the $v_\parallel$-bounds (Lemma \ref{mno1}) and  uniform boundedness of $\|\partial_\theta \mathcal{Q}_{\omega}\|_{L^2(\frac{1}{\sigma_\omega})}$ that
$$\begin{aligned}
\|\mathcal{P}_{\tilde{\Phi}_\omega} \mathcal{Q}_\omega\|_{L^6(\sigma_\omega)}^6&\lesssim \|\mathcal{Q}_{\omega,\parallel}\|_{L^2(\mathbb{S}^1)}^4\|\partial_\theta \mathcal{Q}_{\omega,\parallel}\|_{L^2(\mathbb{S}^1)}^2+\|\mathcal{Q}_{\omega,\parallel}\|_{L^2(\mathbb{S}^1)}^6\\
&\lesssim \|\mathcal{Q}_\omega\|_{L^2(\sigma_\omega)}^4\|\partial_\theta \mathcal{Q}_{\omega}\|_{L^2(\frac{1}{\sigma_\omega})}^2+\|\mathcal{Q}_\omega\|_{L^2(\sigma_\omega)}^6\lesssim 1.
\end{aligned}$$
For $\mathcal{P}_{\tilde{\Phi}_\omega}^\perp\mathcal{Q}_\omega$, by the Sobolev inequality \eqref{translated Sobolev}, we obtain
$$\|\mathcal{P}_{\tilde{\Phi}_\omega}^\perp\mathcal{Q}_\omega\|_{L^6(\sigma_\omega)}\lesssim \omega^{-\frac{1}{3}}\left\{\omega\|\nabla_{(s,z)}\mathcal{P}_{\tilde{\Phi}_\omega}^\perp\mathcal{Q}_\omega\|_{L^2(\sigma_\omega)}^2+\|\partial_\theta \mathcal{P}_{\tilde{\Phi}_\omega}^\perp\mathcal{Q}_\omega\|_{L^2(\frac{1}{\sigma_\omega})}^2\right\}^{\frac{1}{2}}.$$
Then, applying   Lemma \ref{orthogonal complement bound} to the first term and the asymptotic Pythagorean identity \eqref{asymptotic Pythagorean identity 3} to the second term on the right hand side of the above inequality, we prove that
$$\begin{aligned}
\|\mathcal{P}_{\tilde{\Phi}_\omega}^\perp\mathcal{Q}_\omega\|_{L^6(\sigma_\omega)}&\lesssim \omega^{-\frac{1}{3}}\bigg\{\omega\|\mathcal{Q}_\omega\|_{\dot{\Sigma}_\omega}^2 +\omega^{-\frac{1}{2}}\bigg\}^{\frac{1}{2}}.
\end{aligned}$$
Therefore, $\|\mathcal{P}_{\tilde{\Phi}_\omega}^\perp\mathcal{Q}_\omega\|_{L^6(\sigma_\omega)}$ is also uniformly bounded.
\end{proof}

\subsection{Uniform Gaussian decay of an energy minimizer}\label{uniformgaus}

The purpose of this subsection is to provide more detailed information about an energy minimizer constructed in Proposition \ref{existence of a minimizer}.

\begin{proposition}[Uniform Gaussian weighted $L^2$ bounds]\label{Gaussian L^2 bounds}
  For any $c\in (0,\alpha_0)$, if $\omega\geq 1$ is large enough, then we have
\begin{align}
\|e^{c(s^2+z^2)}\mathcal{Q}_\omega\|_{L^2(\sigma_\omega)}^2+\|e^{c(s^2+z^2)}\mathcal{Q}_\omega\|_{\dot{\Sigma}_{\omega;(s,z)}}^2+\frac{1}{\omega}\|e^{c(s^2+z^2)}\partial_\theta\mathcal{Q}_\omega\|_{L^2(\frac{1}{\sigma_\omega})}^2&\lesssim1,\label{Gaussian L^2 bound 1}\\
\|\sigma_\omega e^{c(s^2+z^2)} ( \mathcal{H}_\omega^{(2D)}-\Lambda_\omega)\mathcal{Q}_\omega\|_{L^2(\sigma_\omega)}^2+\frac{1}{\omega^2}\|e^{c(s^2+z^2)}\partial_\theta^2\mathcal{Q}_\omega\|_{L^2(\frac{1}{\sigma_\omega})}^2&\lesssim1,\label{Gaussian L^2 bound 2}
\end{align}
where  $\alpha_0>0$ is a constant given in Proposition \ref{1d eigenstate}.
\end{proposition}

\begin{remark}\label{Gaussian L^2 bound remark}
\begin{enumerate}[$(i)$]
\item Proposition \ref{Gaussian L^2 bounds} proves very strong localization near the ring $s=0$ and $z=0$ (or $|y|=\sqrt{\omega}$ and $z=0$).
\item By the definition of the semi-norm (see \eqref{seminorm}), $\|e^{c(s^2+z^2)}\nabla_{(s,z)}\mathcal{Q}_\omega\|_{L^2(\sigma_\omega)}\lesssim1$.
\item For the angle derivatives, we only have inequalities $\|e^{c(s^2+z^2)}\partial_\theta\mathcal{Q}_\omega\|_{L^2(\frac{1}{\sigma_\omega})}\lesssim\sqrt{\omega}$ and $\|e^{c(s^2+z^2)}\partial_\theta^2\mathcal{Q}_\omega\|_{L^2(\frac{1}{\sigma_\omega})}\lesssim\omega$ with increasing upper bounds. However, we still have good localization; there exists $K>0$ such that 
$$\|\mathbf{1}_{|(s,z)|\geq K\sqrt{\ln \omega}}\partial_\theta\mathcal{Q}_\omega\|_{L^2(\frac{1}{\sigma_\omega})}\lesssim \omega^{-cK^2+\frac12}, \|\mathbf{1}_{|(s,z)|\geq K\sqrt{\ln \omega}}\partial_\theta^2\mathcal{Q}_\omega\|_{L^2(\frac{1}{\sigma_\omega})}\lesssim \omega^{-cK^2+1}.
$$
\end{enumerate}
\end{remark}

As mentioned in Remark \ref{Gaussian L^2 bound remark}, the angle derivative bounds in Proposition \ref{Gaussian L^2 bounds} are rather weak near the ring $s=0$ and $z=0$. For the first-order derivative, we, on the other hand, have $\|\partial_\theta\mathcal{Q}_\omega\|_{L^2(\frac{1}{\sigma_\omega})}\lesssim1$ (Proposition \ref{existence of a minimizer} $(iii)$). The following result refines some crudeness for the second-order derivative.

\begin{proposition}[Uniform $L^2(\sigma_\omega)$ bounds for angle derivatives]\label{uniform angle derivative bounds}
$$\|\partial_\theta^2\mathcal{Q}_\omega\|_{L^2(\frac{1}{\sigma_\omega})}\lesssim1.$$
\end{proposition}
\begin{remark}
By Proposition \ref{existence of a minimizer} $(iii)$ and Proposition \ref{uniform angle derivative bounds}, $\|\partial_\theta\mathcal{Q}_\omega\|_{L^2(\frac{1}{\sigma_\omega})}$ and $\|\partial_\theta^2\mathcal{Q}_\omega\|_{L^2(\frac{1}{\sigma_\omega})}$ are uniformly bounded, and so is $\|\mathcal{Q}_{\omega,\parallel}\|_{H^2(\mathbb{S}^1)}$ due to Lemma \ref{mno1}. Hence, by the Sobolev embedding $H^2(\mathbb{S}^1)\hookrightarrow C^1(\mathbb{S}^1)$, $\mathcal{Q}_{\omega,\parallel}$ and $\mathcal{Q}_{\omega,\parallel}'$ are uniformly bounded point-wisely.
\end{remark}

As a first step, we prove a primitive $L^\infty$-bound.

\begin{lemma}[$L^\infty$-bound  ]\label{L^infty bound}
$$
 \esssup_{(s,z,\theta)\in[-\sqrt{\omega},\infty)\times\mathbb{R}\times \mathbb{S}^1}|\mathcal{Q}_\omega(s,z,\theta)|= \|\mathcal{Q}_\omega\|_{L^\infty}\lesssim\omega^{\frac{1}{4}}.
$$
\end{lemma}

For the proof, we employ a Sobolev-type inequality from the semi-group theory, since it is a convenient tool to deal with non-negative external potentials. 

\begin{lemma}[Semigroup Sobolev inequality]\label{semigroup Sobolev} It holds that
$$
 \esssup_{(s,z,\theta)\in[-\sqrt{\omega},\infty)\times\mathbb{R}\times \mathbb{S}^1}|v(s,z,\theta)|=\|v\|_{L^\infty}\lesssim \omega^{\frac{1}{4}}\left\|\left(\mathcal{H}_\omega^{(2D)}-\frac{1}{\omega\sigma_\omega^2}\partial_\theta^2+1\right)v\right\|_{L^2(\sigma_\omega)}.$$
\end{lemma}

\begin{proof}
By $v(s,z,\theta)=\omega^{\frac{1}{4}}u(s+\sqrt{\omega},z,\theta)$ with $|y|=s+\sqrt{\omega}$, the left and the right hand sides of the inequality respectively can be written as $\|v\|_{L^\infty}=\omega^{\frac{1}{4}}\|u\|_{L^\infty(\mathbb{R}^3)}$ and
$$ \left\|\left(\mathcal{H}_\omega^{(2D)}-\frac{1}{\omega\sigma_\omega^2}\partial_\theta^2+1\right)v\right\|_{L^2(\sigma_\omega)}= \left\|(-\Delta+U(|y|-\sqrt{\omega})+z^2+1)u\right\|_{L^2(\mathbb{R}^3)}.$$
Thus, the inequality is equivalent to $\|u\|_{L^\infty(\mathbb{R}^3)}\lesssim \left\|(-\Delta_x+U_\omega(|y|-\sqrt{\omega})+z^2+1)u\right\|_{L^2(\mathbb{R}^3)}$, which indeed holds true by the theory of Schr\"odinger semigroups \cite[Theorem B.2.1]{Simon}.
\end{proof}

\begin{proof}[Proof of Lemma \ref{L^infty bound}]
By the Euler-Lagrange equation \eqref{elp}, it follows from Proposition \ref{existence of a minimizer} $(iii)$ that 
$\|( \mathcal{H}_\omega^{(2D)}-\frac{1}{\omega\sigma_\omega^2}\partial_\theta^2+1)\mathcal{Q}_\omega\|_{L^2(\sigma_\omega)}\le\|(\Lambda_\omega+1-\frac{\mu_\omega}{\omega})\mathcal{Q}_\omega\|_{L^2(\sigma_\omega)}+\frac{1}{\omega}\|\mathcal{Q}_\omega\|_{L^6(\sigma_\omega)}^3$
is uniformly bounded. Hence, Lemma \ref{semigroup Sobolev} proves the desired bound.
\end{proof}

Next, we show Gaussian weighted $L^2(\sigma_\omega)$-bounds following the argument of Agmon \cite{Agmon} as we did in the proof of Proposition \ref{1d eigenstate} $(ii)$.

\begin{proof}[Proof of Proposition \ref{Gaussian L^2 bounds}]
In the proof, if there is no confusion, we omit the integral domain 
$$\iiint dsdzd\theta=\int_{\mathbb{S}^1}\int_{-\infty}^\infty\int_{-\sqrt{\omega}}^\infty dsdzd\theta.$$
Let $g(s,z): [-\sqrt{\omega},\infty)\times\mathbb{R}\to\mathbb{R}$ be a bounded smooth function  with $g(s,z)=1$ for large $|(s,z)|>1$ to be chosen later. Then, it is obvious from the Euler-Lagrange equation \eqref{elp} that 
\begin{equation}\label{elp'}
\begin{aligned}
\left( \mathcal{H}_\omega^{(2D)}-\Lambda_\omega-\frac{1}{\omega\sigma_\omega^2}\partial_\theta^2\right)(\mathcal{Q}_\omega g)&=\frac{\mathcal{Q}_\omega^2-\mu_\omega}{\omega}\mathcal{Q}_\omega g-\frac{1}{\sqrt{\omega}\sigma_\omega}\mathcal{Q}_\omega(\partial_s g)\\
&\quad-2(\nabla_{(s,z)}\mathcal{Q}_\omega)\cdot(\nabla g)-\mathcal{Q}_\omega(\Delta g).
\end{aligned}
\end{equation}
Recalling the definition of the semi-norm $\|\cdot\|_{\dot{\Sigma}_\omega}$ (see \eqref{seminorm2}), we write 
$$\|\mathcal{Q}_\omega  g\|_{\dot{\Sigma}_\omega}^2=\bigg\langle \frac{\mathcal{Q}_\omega^2-\mu_\omega}{\omega}\mathcal{Q}_\omega g-\frac{1}{\sqrt{\omega}\sigma_\omega}\mathcal{Q}_\omega(\partial_sg)-\underbrace{2(\nabla_{(s,z)}\mathcal{Q}_\omega)\cdot(\nabla g)}_{(*)}-\mathcal{Q}_\omega(\Delta g), \mathcal{Q}_\omega  g\bigg\rangle_{L^2(\sigma_\omega)}.$$
Note that for the inner product including the term $(*)$, by integration by parts, we have
\begin{equation}\label{intbyp}
\begin{aligned}
-\left\langle2(\nabla_{(s,z)}\mathcal{Q}_\omega)\cdot\nabla g, \mathcal{Q}_\omega  g\right\rangle_{L^2(\sigma_\omega)}&=-\iiint \left((\nabla_{(s,z)}\mathcal{Q}_\omega^2)\cdot \nabla g\right)g\sigma_\omega dsdzd\theta\\
&=\iiint \mathcal{Q}_\omega^2 \left\{\left(\Delta g\right)g \sigma_\omega+|\nabla g|^2\sigma_\omega+\frac{1}{\sqrt{\omega}}g(\partial_sg)\right\} dsdzd\theta.
\end{aligned}
\end{equation}
Thus, it follows that 
\begin{equation}\label{3d elliptic eq lemma claim proof}
\|\mathcal{Q}_\omega  g\|_{\dot{\Sigma}_\omega}^2=\iiint\left\{\frac{\mathcal{Q}_\omega^2-\mu_\omega}{\omega}g^2+|\nabla g|^2\right\}\mathcal{Q}_\omega^2\sigma_\omega dsdzd\theta,
\end{equation}
and consequently by definition of the semi-norm $\|\cdot\|_{\dot{\Sigma}_\omega}$ again, 
\begin{equation}\label{intbyp1}
\begin{aligned}
0&\leq\|\nabla_{(s,z)}(\mathcal{Q}_\omega g)\|_{L^2(\sigma_\omega)}^2+\frac{1}{\omega}\|\partial_\theta (\mathcal{Q}_\omega g)\|_{L^2(\frac{1}{\sigma_\omega})}^2\\
&\leq-\iiint \left\{U_\omega(s)+z^2-\Lambda_\omega-\frac{\mathcal{Q}_\omega^2-\mu_\omega}{\omega}-\frac{|\nabla g|^2}{g^2}\right\}(\mathcal{Q}_\omega g)^2\sigma_\omega dsdzd\theta.
\end{aligned}
\end{equation}

Now, we insert $g(s,z)=e^{f_L(s,z)}$, where $L\geq \sqrt{\omega}$ is a large number and $f_L(s,z)=c(s^2+z^2)\eta(\frac{|(s,z)|}{L})$ for some $c\in (0,\alpha_0)$ and a smooth cut-off $\eta$ such that $\eta\equiv1$ on $[-\frac14,\frac14]$ and $\eta$ is supported on $[-\frac12,\frac12]$. Then, it follows that
$$0\geq\iiint \left\{U_\omega(s)+z^2-\Lambda_\omega-\frac{\mathcal{Q}_\omega^2-\mu_\omega}{\omega}-|\nabla f_L|^2\right\}\big(e^{f_L}\mathcal{Q}_\omega\big)^2\sigma_\omega dsdzd\theta.$$
We observe that by the assumptions on $U_\omega$ \textbf{\textup{(H2)}} and the $L^\infty$-bound (Lemma \ref{L^infty bound}), if $\omega\geq1$ is large, there exists $R\in(1,\sqrt{\omega})$, independent of $\omega\geq 1$ and $L\ge \sqrt{\omega}$, such that $U_\omega(s)+z^2-\Lambda_\omega-\frac{\mathcal{Q}_\omega^2-\mu_\omega}{\omega}-|\nabla f_L|^2\geq 1$ for all $|(s,z)|\geq R$, while there exists $C_{R}>0$, independent of $\omega\geq 1$ and $L\ge \sqrt{\omega}$, such that $|U_\omega(s)+z^2-\Lambda_\omega-\frac{\mathcal{Q}_\omega^2-\mu_\omega}{\omega}-|\nabla f_L|^2|e^{2f_L}\leq C_R^2<\infty$ for all $|(s,z)|\leq R$. Therefore, for a smooth cut-off $\eta$ such that $\eta\equiv1$ on $[-\frac{1}{4},\frac{1}{4}]$ and $\eta(s)=0$ for $|s|\geq \frac{1}{2}$, it follows that 
$$\begin{aligned}
&\big\|\mathbf{1}_{|(s,z)|\geq R}(e^{f_L}\mathcal{Q}_\omega)\big\|_{L^2(\sigma_\omega)}^2\\
&\leq \iiint_{|(s,z)|\geq R} \left\{U_\omega(s)+z^2-\Lambda_\omega-\frac{\mathcal{Q}_\omega^2-\mu_\omega}{\omega}-|\nabla f_L|^2\right\}\big(e^{f_L}\mathcal{Q}_\omega\big)^2\sigma_\omega dsdzd\theta\\
&\leq -\iiint_{|(s,z)|\leq R}\left\{U_\omega(s)+z^2-\Lambda_\omega-\frac{\mathcal{Q}_\omega^2-\mu_\omega}{\omega}-|\nabla f_L|^2\right\}\big(e^{f_L}\mathcal{Q}_\omega\big)^2\sigma_\omega dsdzd\theta\leq C_R^2,
\end{aligned}$$
and taking $L\to\infty$, we obtain $\|\mathbf{1}_{|(s,z)|\geq R}(e^{c(s^2+z^2)}\mathcal{Q}_\omega)\|_{L^2(\sigma_\omega)}\leq C_R$. On the other hand, we have 
$\|\mathbf{1}_{|(s,z)|\leq R}(e^{c(s^2+z^2)}\mathcal{Q}_\omega)\|_{L^2(\sigma_\omega)}\leq e^{cR^2}\|\mathcal{Q}_\omega\|_{L^2(\sigma_\omega)}^2=e^{cR^2}m$. Since $R$ does not depend on $\omega\geq 1$, we prove 
\begin{equation}\label{l2ex}
\|e^{c(s^2+z^2)}\mathcal{Q}_\omega\|_{L^2(\sigma_\omega)}^2\lesssim1.
\end{equation} To estimate the second and the third terms in \eqref{Gaussian L^2 bound 1}, we take $g=e^{c(s^2+z^2)\eta(\frac{|(s,z)|}{L})}$ in \eqref{3d elliptic eq lemma claim proof}. Then, \eqref{l2ex} with  $c_0\in (c,\alpha_0)$ implies that the right hand side of \eqref{3d elliptic eq lemma claim proof} is bounded uniformly in large $L\geq 1$. Thus, sending $L\to\infty$, we prove $\|e^{c(s^2+z^2)}\mathcal{Q}_\omega\|_{\dot{\Sigma}_\omega }^2\lesssim1$.

For \eqref{Gaussian L^2 bound 2}, we insert $g=e^{c(s^2+z^2)}$ in \eqref{elp'} and multiply both sides by $\sigma_\omega$. Then, estimating the right side by \eqref{Gaussian L^2 bound 1} with  $c_0\in (c,\alpha_0)$, together with the $L^\infty$-bound (Lemma \ref{L^infty bound}), yields 
$$\left\|\sigma_\omega ( \mathcal{H}_\omega^{(2D)}-\Lambda_\omega)(\mathcal{Q}_\omega e^{c(s^2+z^2)})-\frac{1}{\omega\sigma_\omega}\partial_\theta^2(\mathcal{Q}_\omega e^{c(s^2+z^2)})\right\|_{L^2(\sigma_\omega)}^2\lesssim 1.$$
Finally, expanding the left hand side, we can deduce \eqref{Gaussian L^2 bound 2}, because the cross inner product $\langle ( \mathcal{H}_\omega^{(2D)}-\Lambda_\omega)v, -\partial_\theta^2v\rangle_{L^2(\sigma_\omega)}=\|\partial_\theta v\|_{\dot{\Sigma}_{\omega;(s,z)}}^2$ is non-negative.
\end{proof}

Next, we prove uniform $L^2(\sigma_\omega)$ bounds for the second-order angle derivatives.
\begin{proof}[Proof of Proposition \ref{uniform angle derivative bounds}]
By Proposition \ref{Gaussian L^2 bounds} (see Remark \ref{Gaussian L^2 bound remark}), it is enough to show that $\|\eta(\tfrac{s}{\sqrt{\omega}})\partial_\theta^2\mathcal{Q}_\omega\|_{L^2(\frac{1}{\sigma_\omega})}\lesssim1$, where $\eta$ is a smooth cut-off such that $\eta\equiv1$ on $[-\frac{1}{4},\frac{1}{4}]$ and $\eta$ is supported on $[-\frac{1}{2},\frac{1}{2}]$. For the proof, we take $g=\eta(\frac{s}{\sqrt{\omega}})$ in \eqref{elp'}, we write
\begin{equation}\label{elp''}
\begin{aligned}
\left( \mathcal{H}_\omega^{(2D)}-\Lambda_\omega-\frac{1}{\omega\sigma_\omega^2}\partial_\theta^2\right)(\mathcal{Q}_\omega \eta(\tfrac{s}{\sqrt{\omega}}))&=\frac{\mathcal{Q}_\omega^2-\mu_\omega}{\omega}\mathcal{Q}_\omega \eta(\tfrac{s}{\sqrt{\omega}})-\frac{1}{\sqrt{\omega}\sigma_\omega}\mathcal{Q}_\omega(\eta(\tfrac{s}{\sqrt{\omega}}))'\\
&\quad-2(\partial_s\mathcal{Q}_\omega)(\eta(\tfrac{s}{\sqrt{\omega}}))'-\mathcal{Q}_\omega(\eta(\tfrac{s}{\sqrt{\omega}}))''.
\end{aligned}
\end{equation}
We multiply both sides of \eqref{elp''} by $\sigma_\omega$. Then, we can show that the $L^2(\sigma_\omega)$-norm of the right hand side is bounded by $\omega^{-1}$, because  $\| \mathcal{Q}_\omega\|_{L^6( \sigma_\omega)}\lesssim1$ (see Proposition \ref{existence of a minimizer} (iii)), the derivatives of $\eta(\tfrac{s}{\sqrt{\omega}})$ is supported on $\frac{\sqrt{\omega}}{4}\leq |s|\leq\frac{\sqrt{\omega}}{2}$, and $\mathcal{Q}_\omega$ and $\partial_s\mathcal{Q}_\omega$ can take Gaussian weights (Proposition \ref{Gaussian L^2 bounds}). Therefore, it follows that
$$\begin{aligned}
\frac{1}{\omega^2}&\gtrsim \left\|\sigma_\omega( \mathcal{H}_\omega^{(2D)}-\Lambda_\omega)(\mathcal{Q}_\omega \eta(\tfrac{s}{\sqrt{\omega}}))-\frac{1}{\omega\sigma_\omega}\partial_\theta^2(\mathcal{Q}_\omega \eta(\tfrac{s}{\sqrt{\omega}}))\right\|_{L^2(\sigma_\omega)}^2\\
&\geq \big\|\sigma_\omega( \mathcal{H}_\omega^{(2D)}-\Lambda_\omega)(\mathcal{Q}_\omega \eta(\tfrac{s}{\sqrt{\omega}}))\big\|_{L^2(\sigma_\omega)}^2+\frac{1}{\omega^2}\|\eta(\tfrac{s}{\sqrt{\omega}})\partial_\theta^2\mathcal{Q}_\omega\|_{L^2(\frac{1}{\sigma_\omega})}^2,
\end{aligned}$$
because the cross-inner product is non-negative.
\end{proof}

\subsection{Dimension reduction} \label{dimension reduction section}

We close this section by proving the dimension reduction in the strong confinement limit $\omega\to\infty$.

\begin{proposition}[Dimension reduction]\label{adprop}
For sufficiently large $\omega\geq 1$, the following hold.
\begin{enumerate}[(i)]
\item (Improved vanishing higher eigenstates) 
$$
 \omega \| \mathcal{Q}_\omega\|_{{\dot{\Sigma}}_{\omega; (s,z)}}^2+ \big\|\partial_\theta(\mathcal{P}_{\tilde{\Phi}_\omega}^\perp \mathcal{Q}_\omega)\big\|_{L^2(\frac{1}{\sigma_\omega})}^2=O(\omega^{-\frac12}).
$$
\item (Minimum energy convergence)
$$J_\infty^{(1D)}(m)= \mathcal{J}_\omega^{(3D)}(m)+O(\omega^{-\frac12}).$$
\item (Strong convergences to the circle ground state) Translating along the $\theta$-axis if necessary, $\mathcal{Q}_\omega$ satisfies 
$$\lim_{\omega\to\infty}\|\mathcal{Q}_{\omega,\parallel}-Q_\infty\|_{H^1(\mathbb{S}^1)}=0\textup{ and }
\lim_{\omega\to\infty}\|\partial_\theta\mathcal{Q}_\omega-\partial_\theta\mathcal{Q}_{\infty}(\theta)\Phi_\omega(s,z)\|_{L^2(\sigma_\omega)}=0.$$
\item (Lagrange multiplier convergence)
$$\mu_\omega=\mu_\infty+o_\omega(1).$$
\end{enumerate}
\end{proposition}

\begin{proof}
We prove $(i)$ and $(ii)$ by extracting the mass and the energy of $\mathcal{Q}_{\omega,\parallel}$. Indeed, by orthogonality, $m=\|\mathcal{Q}_\omega\|_{L^2(\sigma_\omega)}^2=\|\mathcal{Q}_{\omega,\parallel}\|_{L^2 (\mathbb{S}^1)}^2+\|\mathcal{P}_{\tilde{\Phi}_\omega} ^\perp\mathcal{Q}_\omega\|_{L^2(\sigma_\omega)}^2$. However, by the bound for $\mathcal{P}_{\tilde{\Phi}_\omega}^\perp$ (Lemma \ref{orthogonal complement bound}) and a uniform bound in Proposition \ref{existence of a minimizer} $(iii)$, we have $\|\mathcal{P}_{\tilde{\Phi}_\omega}^\perp  \mathcal{Q}_\omega\|_{L^2(\sigma_\omega)}^2\lesssim\|\mathcal{Q}_\omega\|_{{\dot{\Sigma}}_{\omega; (s,z)}}^2 +e^{-c\omega}\|\mathcal{Q}_\omega\|_{L^2(\sigma_\omega)}^2\lesssim \omega^{-1}$. Thus, we obtain
$$M_\infty(\mathcal{Q}_{\omega,\parallel})=\|\mathcal{Q}_{\omega,\parallel}\|_{L^2 (\mathbb{S}^1)}^2=m+O(\omega^{-1}).$$
On the other hand, we observe that the potential energy is well approximated by the projection, because by the H\"older inequality and Proposition  \ref{existence of a minimizer} $(iii)$, 
\begin{align*}
\big|\|\mathcal{Q}_\omega\|_{L^4(\sigma_\omega)}^4 -\|\mathcal{P}_{\tilde{\Phi}_\omega} \mathcal{Q}_\omega\|_{L^4(\sigma_\omega)}^4\big|\lesssim \left\{\|\mathcal{Q}_\omega\|_{L^6(\sigma_\omega)}^3+\|\mathcal{P}_{\tilde{\Phi}_\omega}  \mathcal{Q}_\omega\|_{L^6(\sigma_\omega)}^3\right\}\|\mathcal{P}_{\tilde{\Phi}_\omega}^\perp \mathcal{Q}_\omega\|_{L^2(\sigma_\omega)}\lesssim \omega^{-\frac{1}{2}}
\end{align*}
and by Proposition \ref{1d eigenstate} (iii),
$$
\|\mathcal{P}_{\tilde{\Phi}_\omega} \mathcal{Q}_\omega\|_{L^4(\sigma_\omega)}^4=\frac{\|\phi_\infty\|_{L^4(\R^)}^4\|\chi_\omega \phi_\omega \|_{L^4(\sigma_\omega)}^4\|\mathcal{Q}_{\omega,\parallel}\|_{L^4 (\mathbb{S}^1)}^4}{\|\chi_\omega \phi_\omega \|_{L^2(\sigma_\omega)}^4}=\Big(\frac{1}{2\pi}+O(\omega^{-\frac12})\Big)\|\mathcal{Q}_{\omega,\parallel}\|_{L^4 (\mathbb{S}^1)}^4.
$$ 
Hence, by \eqref{tue3} and the asymptotic Pythagorean identity \eqref{asymptotic Pythagorean identity 3}, the energy $\mathcal{E}_\omega(\mathcal{Q}_\omega)$ can be written as 
$$\mathcal{E}_\omega(\mathcal{Q}_\omega)=\mathcal{E}_\infty(\mathcal{Q}_{\omega,\parallel})+\frac{\omega}{2}\|\mathcal{Q}_\omega\|_{{\dot{\Sigma}}_{\omega; (s,z)}}^2+\frac{1}{2}\big\|\partial_\theta(\mathcal{P}_{\tilde{\Phi}_\omega}^\perp \mathcal{Q}_\omega)\big\|_{L^2(\frac{1}{\sigma_\omega})}^2+O(\omega^{-\frac{1}{2}}).$$
Thus, it follows that
$$\begin{aligned}
J_\infty^{(1D)}(m)&\le E_\infty\left(\frac{\sqrt{m}\mathcal{Q}_{\omega,\parallel}}{\|\mathcal{Q}_{\omega,\parallel}\|_{L^2 (\mathbb{S}^1)}}\right)=E_\infty( \mathcal{Q}_{\omega,\parallel})+O(\omega^{-\frac{1}{2}})\\
&=\mathcal{E}_\omega(\mathcal{Q}_\omega)-\frac{\omega}{2}\| \mathcal{Q}_\omega\|_{{\dot{\Sigma}}_{\omega; (s,z)}}^2-\frac12\big\|\partial_\theta(\mathcal{P}_{\tilde{\Phi}_\omega}^\perp \mathcal{Q}_\omega)\big\|_{L^2(\frac{1}{\sigma_\omega})}^2+O(\omega^{-\frac12})\\
&=\mathcal{J}_\omega^{(3D)}(m)-\frac{\omega}{2}\| \mathcal{Q}_\omega\|_{{\dot{\Sigma}}_{\omega; (s,z)}}^2-\frac12\big\|\partial_\theta(\mathcal{P}_{\tilde{\Phi}_\omega}^\perp \mathcal{Q}_\omega)\big\|_{L^2(\frac{1}{\sigma_\omega})}^2+O(\omega^{-\frac12}).
\end{aligned}$$
Combining the upper bound on $\mathcal{J}_\omega^{(3D)}(m)$ (Lemma \ref{upper}), we conclude that $J_\infty^{(1D)}(m)= \mathcal{J}_\omega^{(3D)}(m)+O(\omega^{-\frac12})$ and $\omega \| \mathcal{Q}_\omega\|_{{\dot{\Sigma}}_{\omega; (s,z)}}^2+ \big\|\partial_\theta(\mathcal{P}_{\tilde{\Phi}_\omega}^\perp \mathcal{Q}_\omega)\big\|_{L^2(\frac{1}{\sigma_\omega})}^2=O(\omega^{-\frac12})$.

For $(iii)$, we note from the proof of Proposition \ref{adprop} $(i)$ and $(ii)$ that $\{\frac{\sqrt{m}\mathcal{Q}_{\omega,\parallel}}{\|\mathcal{Q}_{\omega,\parallel}\|_{L^2 (\mathbb{S}^1)}}\}$ is shown to be a minimizing sequence for $J_\infty^{(1D)}(m)$. Hence, by the concentration-compactness property for $J_\infty^{(1D)}(m)$ and the uniqueness of the minimizer $Q_\infty$ (see \cite{L, GLT}), we conclude that up to a translation, $\mathcal{Q}_{\omega,\parallel}\rightarrow Q_\infty$ in $H^1(\mathbb{S}^1)$ as $\omega\rightarrow \infty$. We assume that translating along the $\theta$-axis if necessary, $\mathcal{Q}_{\omega,\parallel}\rightarrow Q_\infty$ in $H^1(\mathbb{S}^1)$ as $\omega\rightarrow \infty$. On the other hand, this convergence immediately implies that $\|\partial_\theta(\mathcal{P}_{\tilde{\Phi}_\omega}\mathcal{Q}_\omega)-\partial_\theta\mathcal{Q}_{\infty}(\theta)\Phi_\omega(s,z)\|_{L^2(\sigma_\omega)}=\|\partial_\theta(\mathcal{Q}_{\omega,\parallel}-Q_\infty)\|_{L^2(\mathbb{S}^1)}+o_\omega(1)\to 0$ as $\omega\rightarrow \infty$. Thus, it remains to show that $\|\partial_\theta(\mathcal{P}_{\tilde{\Phi}_\omega}^\perp\mathcal{Q}_\omega)\|_{L^2(\sigma_\omega)}\to 0$. Indeed, since $\mathcal{P}_{\tilde{\Phi}_\omega}^\perp$ and $\partial_\theta$ commute, using   Proposition \ref{adprop} $(i)$, we obtain 
$$\begin{aligned}
\|\partial_\theta(\mathcal{P}_{\tilde{\Phi}_\omega}^\perp\mathcal{Q}_\omega)\|_{L^2(\sigma_\omega)}^2 &=\int_{\mathbb{S}^1}\int_{-\infty}^{\infty}\int_{-\sqrt{\omega}}^\infty \frac{\sigma_\omega^2-1}{\sigma_\omega} |\partial_\theta(\mathcal{P}_{\tilde{\Phi}_\omega}^\perp\mathcal{Q}_\omega)|^2d s dz d\theta+o_\omega(1)\\
&\lesssim\big\|\sqrt{|\sigma_\omega^2-1|}\partial_\theta \mathcal{Q}_\omega\big\|_{L^2(\frac{1}{\sigma_\omega})}^2+\big\|\sqrt{|\sigma_\omega^2-1|}\partial_\theta(\mathcal{P}_{\tilde{\Phi}_\omega}\mathcal{Q}_\omega)\big\|_{L^2(\frac{1}{\sigma_\omega})}^2+o_\omega(1).
\end{aligned}$$
For the first term,  we have
$$\begin{aligned}
&\big\|\sqrt{|\sigma_\omega^2-1|}\partial_\theta \mathcal{Q}_\omega\big\|_{L^2(\frac{1}{\sigma_\omega})}\\
&\leq\omega^{-\frac{1}{2}}\|\mathbf{1}_{|s|\leq \omega^\frac18}\sqrt{|s|(s+2\sqrt{\omega})}\partial_\theta \mathcal{Q}_\omega\|_{L^2(\frac{1}{\sigma_\omega})}+\omega^{-\frac{1}{2}}\|\mathbf{1}_{|s|\geq \omega^\frac18}\sqrt{|s|(s+2\sqrt{\omega})}\partial_\theta \mathcal{Q}_\omega\|_{L^2(\frac{1}{\sigma_\omega})}\\
&=o_\omega(1)\quad\textup{(by Proposition \ref{existence of a minimizer} $(iii)$ and Proposition \ref{Gaussian L^2 bounds})}.
\end{aligned}$$
For the second term on the bound, we have
$$\begin{aligned}
\|\sqrt{|\sigma_\omega^2-1|}\partial_\theta(\mathcal{P}_{\tilde{\Phi}_\omega}\mathcal{Q}_\omega)\|_{L^2(\frac{1}{\sigma_\omega})}&=\|\chi_\omega\Phi_\omega\|_{L^2(\sigma_\omega)}^{-1}\|\sqrt{|\sigma_\omega^2-1|}\chi_\omega\Phi_\omega\|_{L^2(\frac{1}{\sigma_\omega})}\|\partial_\theta\mathcal{Q}_{\omega,\parallel}\|_{L^2(\mathbb{S}^1)}\\
&\lesssim o_\omega(1)\|\partial_\theta\mathcal{Q}_{\omega}\|_{L^2(\frac{1}{\sigma_\omega})}\quad\textup{(by Lemma \ref{mno1})}.
\end{aligned}$$
Therefore, we conclude that $\|\partial_\theta(\mathcal{P}_{\tilde{\Phi}_\omega}^\perp\mathcal{Q}_\omega)\|_{L^2(\sigma_\omega)}=o_\omega(1)$.

For $(iv)$, we note that by \eqref{tue3}, \eqref{asymptotic Pythagorean identity 3}, \eqref{po1}   and Proposition \ref{adprop} $(i)$,
\begin{align*}
\mathcal{E}_\omega(\mathcal{Q}_\omega) &=\frac{\omega}{4}\|\mathcal{Q}_\omega\|_{{\dot{\Sigma}}_{\omega; (s,z)}}^2+\frac{1}{4}\|\partial_\theta \mathcal{Q}_\omega\|_{L^2(\frac{1}{\sigma_\omega})}^2-\frac{\mu_\omega}{4}\|\mathcal{Q}_\omega\|_{L^2(\sigma_\omega)}^2\\
&=\frac14\|\partial_\theta \mathcal{Q}_{\omega,\parallel}\|_{L^2(\mathbb{S}^1)}^2-\frac{\mu_\omega m}{4}+o_\omega(1),
\end{align*}
while
\begin{align*}
 E_\infty(Q_\infty)=\frac12\|\partial_\theta Q_\infty\|_{L^2(\mathbb{S}^1)}^2-\frac{1}{8\pi} \|Q_\infty\|_{L^4(\mathbb{S}^1)}^4=\frac14\|\partial_\theta Q_\infty\|_{L^2(\mathbb{S}^1)}^2-\frac{\mu_\infty m}{4}.
\end{align*}
Thus, combining with the two convergences $ \mathcal{J}_\omega^{(3D)}(m)\to J_\infty^{(1D)}(m)$ and $\mathcal{Q}_{\omega,\parallel}\rightarrow Q_\infty$ in $H^1(\mathbb{S}^1)$, we conclude $\mu_\omega\rightarrow \mu_\infty$.
 \end{proof}

\begin{proof}[Proof of Theorem \ref{drth}] We write
\begin{align*}
\mathcal{Q}_\omega(s,z,\theta)-\mathcal{Q}_\infty(\theta) \chi_\omega(s) \Phi_0(s,z) &=\Big(\mathcal{Q}_{\omega,\parallel}(\theta)-Q_\infty(\theta)\Big)\frac{\chi_\omega(s) \Phi_\omega(s,z)}{\|\chi_\omega  \Phi_\omega\|_{L^2(\sigma_\omega dsdz)}}+\mathcal{P}_{\tilde{\Phi}_\omega}^\perp \mathcal{Q}_\omega\\
&\quad  +\Big(\|\chi_\omega  \Phi_\omega\|_{L^2(\sigma_\omega dsdz)}^{-1}\Phi_\omega(s,z)-\Phi_0(s,z)\Big)Q_\infty(\theta) \chi_\omega(s)  .
\end{align*}
Then we see that by the bound for $\mathcal{P}_{\tilde{\Phi}_\omega}^\perp$(Lemma \ref{orthogonal complement bound}), the fast decay and convergence of $\phi_\omega$(Proposition \ref{1d eigenstate}),  strong convergences of $\mathcal{Q}_{\omega,\parallel}$ and vanishing higher eigenstates (Proposition \ref{adprop}),
$$
\|\mathcal{Q}_\omega -\mathcal{Q}_\infty  \chi_\omega \Phi_0\|_{L^2(\sigma_\omega)\cap {\dot{\Sigma}}_{\omega;(s,z)}}  \lesssim \|\mathcal{Q}_{\omega,\parallel} -Q_\infty \|_{L^2(\mathbb{S}^1)} +\|\mathcal{P}_{\tilde{\Phi}_\omega}^\perp \mathcal{Q}_\omega\|_{L^2(\sigma_\omega)\cap {\dot{\Sigma}}_{\omega;(s,z)}}+o_\omega(1)  =o_\omega(1) 
$$ 
and
$$
 \|\partial_\theta (\mathcal{Q}_\omega -\mathcal{Q}_\infty  \chi_\omega \Phi_0)\|_{L^2(\frac{1}{\sigma_\omega})}\lesssim \|\partial_\theta(\mathcal{Q}_{\omega,\parallel} -Q_\infty) \|_{L^2(\mathbb{S}^1)}+\big\|\partial_\theta(\mathcal{P}_{\tilde{\Phi}_\omega}^\perp \mathcal{Q}_\omega)\big\|_{L^2(\frac{1}{\sigma_\omega})}+o_\omega(1)=o_\omega(1),
$$
and therefore we get
$$
\lim_{\omega\to \infty}\|\mathcal{Q}_\omega(s,z,\theta)-\mathcal{Q}_\infty(\theta) \chi_\omega(s) \Phi_0(s,z) \|_{\Sigma_\omega}=0.
$$
\end{proof}

\section{Uniqueness of a ground state: Proof of Theorem \ref{thm: unique}}\label{sec: uniqueness}

In this section, we derive a coercivity estimate of the linearized operator at the constrained minimizer $\mathcal{Q}_\omega$ invoking that for the 1D periodic NLS (see \cite{GLT} and Appendix \ref{sec: Energy minimization for the cubic NLS on the unit circle}) and the dimension reduction limit. Then, finally, we complete the proof uniqueness of a minimizer $\mathcal{Q}_\omega$.

\subsection{Linearized operator at a 3D ground state}
For our 3D variational problem $\mathcal{J}_\omega^{(3D)}(m)$, the operator \begin{equation}\label{eq: linearized operator}
\mathcal{L}_\omega:=\omega(\mathcal{H}_\omega^{(2D)}-\Lambda_\omega)-\frac{1}{\sigma_\omega^2}\partial_\theta^2+\mu_\omega-3 \mathcal{Q}_\omega^2,
\end{equation}
acting on  $L^2(\sigma_\omega)$, arises when linearizing about a ground state $\mathcal{Q}_\omega$. By rotation invariance with respect to the $z$-axis in the formulation in $\mathbb{R}^3$ (see \eqref{requ}), its kernel contains $\partial_\theta\mathcal{Q}_\omega$. Moreover, it has at least one negative eigenvalue, because $\langle\mathcal{L}_\omega \mathcal{Q}_\omega, \mathcal{Q}_\omega\rangle_{L^2(\sigma_\omega)}<0$.

The following proposition asserts that in the strong confinement regime, the linearized operator $\mathcal{L}_\omega$ is strictly positive on the subspace $\textup{span}\{\mathcal{Q}_\omega, \partial_\theta\mathcal{Q}_\omega\}^\perp\subset L^2(\sigma_\omega)$.

\begin{proposition}[Coercivity of the linearized operator $\mathcal{L}_\omega$]\label{coc1}
Suppose that $m\neq 2\pi^2$, $\omega\ge1$ is large enough, and $\mathcal{Q}_\omega$ satisfies 
\begin{equation}\label{direas}
\lim_{\omega\to\infty}\|\mathcal{Q}_{\omega,\parallel}-Q_\infty\|_{H^1(\mathbb{S}^1)}=
\lim_{\omega\to\infty}\|\partial_\theta\mathcal{Q}_\omega-\partial_\theta\mathcal{Q}_{\infty}(\theta)\Phi_\omega(s,z)\|_{L^2(\sigma_\omega)}=0.
\end{equation}
  Then, we have
$$
\langle \mathcal{L}_\omega \varphi,\varphi\rangle_{L^2(\sigma_\omega)}\ge \frac{L_\infty}{8}\| \varphi_{\omega,\parallel}\|_{H^1(\mathbb{S}^1)}^2+\frac{\omega}{4}\|\varphi\|_{\dot{\Sigma}_{\omega;(s,z)}}^2
$$
for all $\varphi\in \Sigma_\omega$ such that $\langle \varphi, \mathcal{Q}_\omega\rangle_{L^2(\sigma_\omega)}=\langle\varphi,  \partial_\theta \mathcal{Q}_\omega\rangle_{L^2(\sigma_\omega)}=0$, where $L_\infty>0$ is the constant given in Proposition \ref{prop: limit linearized operator coercivity}. 
\end{proposition}

\begin{remark}
Proposition \ref{coc1} implies the non-degeneracy, i.e., $\textup{Ker}(\mathcal{L}_\omega)=\textup{span}\{\partial_\theta\mathcal{Q}_\omega\}$, and it also shows that the linearized operator has only one negative eigenvalue.
\end{remark}

\begin{proof}[Proof of Proposition \ref{coc1}]
Suppose that $\varphi$ is orthogonal to $\mathcal{Q}_\omega$ and $\partial_\theta \mathcal{Q}_\omega$ in $L^2(\sigma_\omega)$. For the proof, we introduce  an auxiliary 3D linear operator
$$\tilde{\mathcal{L}}_\omega:=-\frac{1}{\sigma_\omega^2}\partial_\theta^2+\mu_\infty-3\big( Q_\infty(\theta)\tilde{\Phi}_\omega(s,z)\big)^2,$$
where $\tilde{\Phi}(s,z)=\frac{\chi_\omega(s)\Phi_{\omega}(s,z)}{\|\chi_\omega \Phi_{\omega} \|_{L^2(\sigma_\omega dsdz)}}$ (see \eqref{projection}).
Then, by the definitions, $\langle\mathcal{L}_\omega\varphi,\varphi\rangle_{L^2(\sigma_\omega)}$ can be written as
$$\begin{aligned}
\langle\mathcal{L}_\omega\varphi,\varphi\rangle_{L^2(\sigma_\omega)}&=\omega\|\varphi\|_{\dot{\Sigma}_{\omega;(s,z)}}^2+\langle\tilde{\mathcal{L}}_\omega\varphi,\varphi\rangle_{L^2(\sigma_\omega)}+(\mu_\omega-\mu_\infty)\|\varphi\|_{L^2(\sigma_\omega)}^2\\
&\quad-3\big\langle(\mathcal{Q}_{\omega,\parallel}^2- Q_\infty^2)\tilde{\Phi}_\omega^2\varphi,\varphi\big\rangle_{L^2(\sigma_\omega)}-3\big\langle(\mathcal{Q}_{\omega,\parallel}\tilde{\Phi}_\omega+\mathcal{Q}_\omega)(\mathcal{P}_{\tilde{\Phi}_\omega}^\perp\mathcal{Q}_\omega)\varphi,\varphi\big\rangle_{L^2(\sigma_\omega)},
\end{aligned}$$
where $-3\mathcal{Q}_\omega^2$ in $\mathcal{L}_\omega$ is expanded by the decomposition $\mathcal{Q}_\omega=\mathcal{Q}_{\omega,\parallel} \tilde{\Phi}_\omega+\mathcal{P}_{\tilde{\Phi}_\omega}^\perp\mathcal{Q}_\omega$. We recall from Proposition \ref{adprop} $(iv)$ that $\mu_\omega-\mu_\infty=o_\omega(1)$. On the other hand, by the H\"older inequality and the convergence $\mathcal{Q}_{\omega,\parallel}\to  Q_\infty$ in $H^1(\mathbb{S}^1)$, we have
$$\begin{aligned}
\big|\big\langle(\mathcal{Q}_{\omega,\parallel}^2-Q_\infty^2)\tilde{\Phi}_\omega^2\varphi, \varphi\big\rangle_{L^2(\sigma_\omega)}\big|&\lesssim \|\mathcal{Q}_{\omega,\parallel}-Q_\infty\|_{L^4(\mathbb{S}^1)}\big(\|\mathcal{Q}_{\omega,\parallel}\|_{L^4(\mathbb{S}^1)}+\|Q_\infty\|_{L^4(\mathbb{S}^1)}\big)\|\varphi\|_{L^4(\sigma_\omega)}^2\\
&=o_\omega(1)\|\varphi\|_{L^4(\sigma_\omega)}^2,
\end{aligned}$$
and by the the fact that $\mathcal{P}_{\tilde{\Phi}_\omega}^\perp\mathcal{Q}_\omega\to 0$ in $L^4(\sigma_\omega)$ (by Lemma \ref{orthogonal complement bound}, Proposition \ref{GN}, Proposition \ref{adprop} $(i)$),
$$\begin{aligned}
&\big|\big\langle(\mathcal{Q}_{\omega,\parallel}\tilde{\Phi}_\omega+\mathcal{Q}_\omega)\mathcal{P}_{\tilde{\Phi}_\omega}^\perp\mathcal{Q}_\omega\varphi, \varphi\big\rangle_{L^2(\sigma_\omega)}\big|\\
&\leq\big(\|\mathcal{Q}_{\omega,\parallel}\tilde{\Phi}_\omega\|_{L^4(\sigma_\omega)}+\|\mathcal{Q}_\omega\|_{L^4(\sigma_\omega)}\big)\|\mathcal{P}_{\tilde{\Phi}_\omega}^\perp\mathcal{Q}_\omega\|_{L^4(\sigma_\omega)}\|\varphi\|_{L^4(\sigma_\omega)}^2=o_\omega(1)\|\varphi\|_{L^4(\sigma_\omega)}^2.
\end{aligned}$$
Thus, it follows that
$$\langle \mathcal{L}_\omega \varphi,\varphi\rangle_{L^2(\sigma_\omega)}\ge \omega \|\varphi\|_{\dot{\Sigma}_{\omega;(s,z)}}^2+\langle \tilde{\mathcal{L}}_\omega \varphi,\varphi\rangle_{L^2(\sigma_\omega)}-o_\omega(1)\|\varphi\|_{L^2(\sigma_\omega)}^2 -o_\omega(1)\|\varphi\|_{L^4(\sigma_\omega)}^2.$$
We note that by Lemma \ref{orthogonal complement bound}, 
\begin{equation}\label{simple L^2 bound}
\|\varphi\|_{L^2(\sigma_\omega)}^2=\|\mathcal{P}_{\tilde{\Phi}_\omega}\varphi\|_{L^2(\sigma_\omega)}^2+\|\mathcal{P}_{\tilde{\Phi}_\omega}^\perp\varphi\|_{L^2(\sigma_\omega)}^2\lesssim\|\varphi_{\omega,\parallel}\|_{L^2(\mathbb{S}^1)}^2+\|\varphi\|_{\dot{\Sigma}_{\omega;(s,z)}}^2,
\end{equation}
and consequently by the refined Gagliardo-Nirenberg inequality (Corollary \ref{GN1}) and Young's inequality, one can easily show that 
\begin{equation}\label{GN1 application}
\begin{aligned}
\|\varphi\|_{L^4(\sigma_\omega)}^2&\lesssim\|\varphi\|_{L^2(\sigma_\omega)}^2+\|\partial_\theta \varphi\|_{L^2(\frac{1}{\sigma_\omega})}^2+\omega\|\varphi\|_{\dot{\Sigma}_{\omega;(s,z)}}^2\\
&\lesssim \|\varphi_{\omega,\parallel}\|_{L^2(\mathbb{S}^1)}^2+\|\partial_\theta \varphi\|_{L^2(\frac{1}{\sigma_\omega})}^2+(1+\omega)\|\varphi\|_{\dot{\Sigma}_{\omega;(s,z)}}^2.
\end{aligned}
\end{equation}
Hence, applying  \eqref{GN1 application}, we obtain
$$\begin{aligned}
\langle \mathcal{L}_\omega \varphi,\varphi\rangle_{L^2(\sigma_\omega)}&\ge(1-o_\omega(1))\omega\|\varphi\|_{\dot{\Sigma}_{\omega;(s,z)}}^2+\langle \tilde{\mathcal{L}}_\omega \varphi,\varphi\rangle_{L^2(\sigma_\omega)}-o_\omega(1)\|\partial_\theta\varphi\|_{L^2(\frac{1}{\sigma_\omega})}^2\\
&\quad -o_\omega(1)\|\varphi_{\omega,\parallel}\|_{L^2(\mathbb{S}^1)}^2.
\end{aligned}$$
Note that in the above inequality, the lower bound includes an unfavorable negative term $-o_\omega(1)\|\partial_\theta\varphi\|_{L^2(\frac{1}{\sigma_\omega})}^2$. In order to absorb it, we take a small number $c_\omega>0$ such that $c_\omega\to 0$ to be specified later, and employ a simple lower bound
$$\begin{aligned}
\langle \tilde{\mathcal{L}}_\omega \varphi,\varphi\rangle_{L^2(\sigma_\omega)}&\geq \|\partial_\theta \varphi\|_{L^2(\frac{1}{\sigma_\omega})}^2-\big(3\| Q_\infty \tilde{\Phi}_\omega\|_{L^\infty}^2-\mu_\infty\big)\|\varphi\|_{L^2(\sigma_\omega)}^2\\
&\geq  \|\partial_\theta \varphi\|_{L^2(\frac{1}{\sigma_\omega})}^2-K\left(\|\varphi_{\omega,\parallel}\|_{L^2(\mathbb{S}^1)}^2+ \|\varphi\|_{\dot{\Sigma}_{\omega;(s,z)}}^2\right)\quad\textup{(by \eqref{simple L^2 bound})}
\end{aligned}$$
to obtain a modified lower bound 
$$\begin{aligned}
\langle \mathcal{L}_\omega \varphi,\varphi\rangle_{L^2(\sigma_\omega)}&\ge (1-o_\omega(1))\omega\|\varphi\|_{\dot{\Sigma}_{\omega;(s,z)}}^2-o_\omega(1)\|\partial_\theta\varphi\|_{L^2(\frac{1}{\sigma_\omega})}^2-o_\omega(1)\|\varphi_{\omega,\parallel}\|_{L^2(\mathbb{S}^1)}^2\\
&\quad +(1-c_\omega)\langle \tilde{\mathcal{L}}_\omega \varphi,\varphi\rangle_{L^2(\sigma_\omega)}+c_\omega\langle   \tilde{\mathcal{L}}_\omega \varphi,\varphi\rangle_{L^2(\sigma_\omega)}\\
&\ge (1-o_\omega(1))\omega\|\varphi\|_{\dot{\Sigma}_{\omega;(s,z)}}^2+(1-c_\omega)\langle \tilde{\mathcal{L}}_\omega \varphi,\varphi\rangle_{L^2(\sigma_\omega)}\\
&\quad+\big(c_\omega-o_\omega(1)\big)\|\partial_\theta\varphi\|_{L^2(\frac{1}{\sigma_\omega})}^2 -\big(Kc_\omega+o_\omega(1)\big)\|\varphi_{\omega,\parallel}\|_{L^2(\mathbb{S}^1)}^2.
\end{aligned}$$
We assume that $0<c_\omega\leq\frac{1}{4}$, $(c_\omega-o_\omega(1))\geq\frac{c_\omega}{2}$ and $(Kc_\omega+o_\omega(1))\leq\frac{L_\infty}{8}$. Then, it is refined as 
$$\langle \mathcal{L}_\omega \varphi,\varphi\rangle_{L^2(\sigma_\omega)}\ge\frac{\omega}{2}\|\varphi\|_{\dot{\Sigma}_{\omega;(s,z)}}^2+\frac{1}{2}\langle \tilde{\mathcal{L}}_\omega \varphi,\varphi\rangle_{L^2(\sigma_\omega)}+\frac{c_\omega}{2}\|\partial_\theta\varphi\|_{L^2(\frac{1}{\sigma_\omega})}^2 -\frac{L_\infty}{8}\|\varphi_{\omega,\parallel}\|_{L^2(\mathbb{S}^1)}^2.$$
As a consequence, the proof of the proposition is reduced to that of a lower bound on the auxiliary operator 
\begin{equation}\label{midpf}
\langle \tilde{\mathcal{L}}_\omega \varphi, \varphi\rangle_{L^2(\sigma_\omega)}\ge \frac{L_\infty}{2}\| \varphi_{\omega,\parallel}\|_{H^1(\mathbb{S}^1)}^2-\frac{\omega}{2}\|\varphi\|_{\dot{\Sigma}_{\omega;(s,z)}}^2-\frac{c_\omega}{2}\|\partial_\theta\varphi\|_{L^2(\frac{1}{\sigma_\omega})}^2.
\end{equation}

To show \eqref{midpf}, we insert the decomposition\footnote{Our goal is to extract the 1D core part from the 3D operator $\tilde{\mathcal{L}}_\omega$. } $\varphi=\mathcal{P}_{\tilde{\Phi}_\omega}\varphi+\mathcal{P}_{\tilde{\Phi}_\omega}^\perp\varphi$ in its left hand side, where $\mathcal{P}_{\tilde{\Phi}_\omega}$ is the modified projection defined in Section \ref{subsec: truncated projection}, and expand as
$$\begin{aligned}
\langle \tilde{\mathcal{L}}_\omega \varphi, \varphi\rangle_{L^2(\sigma_\omega)}&=\langle \tilde{\mathcal{L}}_\omega \mathcal{P}_{\tilde{\Phi}_\omega}\varphi, \mathcal{P}_{\tilde{\Phi}_\omega}\varphi\rangle_{L^2(\sigma_\omega)}+2\langle \tilde{\mathcal{L}}_\omega \mathcal{P}_{\tilde{\Phi}_\omega}^\perp\varphi, \mathcal{P}_{\tilde{\Phi}_\omega}\varphi\rangle_{L^2(\sigma_\omega)}\\
&\quad+\langle \tilde{\mathcal{L}}_\omega \mathcal{P}_{\tilde{\Phi}_\omega}^\perp\varphi, \mathcal{P}_{\tilde{\Phi}_\omega}^\perp\varphi\rangle_{L^2(\sigma_\omega)}.
\end{aligned}$$
Then, it follows from  Proposition \ref{1d eigenstate}   that 
$$\begin{aligned}
\langle \tilde{\mathcal{L}}_\omega \mathcal{P}_{\tilde{\Phi}_\omega}\varphi, \mathcal{P}_{\tilde{\Phi}_\omega}\varphi\rangle_{L^2(\sigma_\omega)}&=\langle \tilde{\mathcal{L}}_\omega (\varphi_{\omega,\parallel}\tilde{\Phi}_\omega), \varphi_{\omega,\parallel}\tilde{\Phi}_\omega \rangle_{L^2(\sigma_\omega)}\\
&=\|\tilde{\Phi}_\omega \|_{L^2(\frac{1}{\sigma_\omega})}^2\langle-\partial_\theta^2 \varphi_{\omega,\parallel},\varphi_{\omega,\parallel}\rangle_{L^2(\mathbb{S}^1)} +\mu_\infty\|\tilde{\Phi}_\omega \|_{L^2(\sigma_\omega)}^2 \|\varphi_{\omega,\parallel}\|_{L^2(\mathbb{S}^1)}^2\\
&\quad -3\|\tilde{\Phi}_\omega \|_{L^4(\sigma_\omega)}^4\langle Q_\infty^2\varphi_{\omega,\parallel},\varphi_{\omega,\parallel}\rangle_{L^2(\mathbb{S}^1)}\\
&=\langle\mathcal{L}_\infty\varphi_{\omega,\parallel},\varphi_{\omega,\parallel}\rangle_{L^2(\mathbb{S}^1)}-o_\omega(1)\|\varphi_{\omega,\parallel}\|_{H^1(\mathbb{S}^1)}.
\end{aligned}$$
On the other hand, by Lemma \ref{orthogonal complement bound} and the asymptotic orthogonality \eqref{asymptotic Pythagorean identity 3}\footnote{$\langle   \mathcal{P}_{\tilde{\Phi}_\omega}^\perp v,  \mathcal{P}_{\tilde{\Phi}_\omega} v\rangle_{L^2(\frac{1}{\sigma_\omega})}=O(\omega^{-\frac{1}{2}})\|v\|_{L^2(\frac{1}{\sigma_\omega})}^2$.}, we have
$$\begin{aligned}
\langle \tilde{\mathcal{L}}_\omega \mathcal{P}_{\tilde{\Phi}_\omega}^\perp\varphi, \mathcal{P}_{\tilde{\Phi}_\omega}\varphi\rangle_{L^2(\sigma_\omega)}&\geq\langle \mathcal{P}_{\tilde{\Phi}_\omega}^\perp(\partial_\theta\varphi), \mathcal{P}_{\tilde{\Phi}_\omega}(\partial_\theta\varphi)\rangle_{L^2(\frac{1}{\sigma_\omega})}+\mu_\infty\langle\mathcal{P}_{\tilde{\Phi}_\omega}^\perp\varphi, \mathcal{P}_{\tilde{\Phi}_\omega}\varphi\rangle_{L^2(\sigma_\omega)}\\
&\quad-3\| Q_\infty\tilde{\Phi}_\omega\|_{L^\infty}^2 \|\mathcal{P}_{\tilde{\Phi}_\omega}^\perp\varphi\|_{L^2(\sigma_\omega)} \|\mathcal{P}_{\tilde{\Phi}_\omega}\varphi\|_{L^2(\sigma_\omega)}\\
&\gtrsim -o_\omega(1) \|\partial_\theta\varphi\|_{L^2(\frac{1}{\sigma_\omega})}^2   -(\|\varphi\|_{\dot{\Sigma}_{\omega;(s,z)}} +e^{-c\omega}\|\varphi_{\omega,\parallel}\|_{L^2(\mathbb{S}^1)})\|\varphi_{\omega,\parallel}\|_{L^2(\mathbb{S}^1)} \\
&\gtrsim -o_\omega(1)\big(\|\partial_\theta\varphi\|_{L^2(\frac{1}{\sigma_\omega})}^2 +\omega\|\varphi\|_{\dot{\Sigma}_{\omega;(s,z)}}^2+\|\varphi_{\omega,\parallel}\|_{L^2(\mathbb{S}^1)}^2\big),
\end{aligned}$$
while Lemma \ref{orthogonal complement bound} implies that 
$$\begin{aligned}
\langle \tilde{\mathcal{L}}_\omega \mathcal{P}_{\tilde{\Phi}_\omega}^\perp\varphi, \mathcal{P}_{\tilde{\Phi}_\omega}^\perp\varphi\rangle_{L^2(\sigma_\omega)}&\geq-3\| Q_\infty\Phi_\omega\|_{L^\infty}^2 \|\mathcal{P}_{\tilde{\Phi}_\omega}^\perp\varphi\|_{L^2(\sigma_\omega)}^2-\mu_\infty\|\mathcal{P}_{\tilde{\Phi}_\omega}^\perp\varphi\|_{L^2(\sigma_\omega)}^2\\
&\gtrsim -\|\varphi\|_{\dot{\Sigma}_{\omega;(s,z)}}^2-o_\omega(1)\|\varphi_{\omega,\parallel}\|_{L^2(\mathbb{S}^1)}^2.
\end{aligned}$$
Collecting all, we conclude that 
$$\begin{aligned}
\langle \tilde{\mathcal{L}}_\omega \varphi, \varphi\rangle_{L^2(\sigma_\omega)}&\geq \langle  \mathcal{L}_\infty  \varphi_{\omega,\parallel},\varphi_{\omega,\parallel}\rangle_{L^2(\mathbb{S}^1)}-\frac{\omega}{4}\|\varphi\|_{\dot{\Sigma}_{\omega;(s,z)}}^2\\
&\quad-o_\omega(1)\|\partial_\theta\varphi\|_{L^2(\frac{1}{\sigma_\omega})}^2-o_\omega(1)\|\varphi_{\omega,\parallel}\|_{H^1(\mathbb{S}^1)}^2
\end{aligned}$$
Hence, assuming $o_\omega(1)\leq\frac{c_\omega}{2}$, the proof of \eqref{midpf} is further reduced to that of the 1D lower bound,  
\begin{equation}\label{rem3}
\langle  \mathcal{L}_\infty  \varphi_{\omega,\parallel},\varphi_{\omega,\parallel}\rangle_{L^2(\mathbb{S}^1)}\ge \frac{2L_\infty}{3}\| \varphi_{\omega,\parallel}\|_{H^1(\mathbb{S}^1)}^2-\frac{\omega}{4}\|\varphi\|_{\dot{\Sigma}_{\omega;(s,z)}}^2.
\end{equation}

For the lower bound \eqref{rem3},  we employ the coercivity of the linearized operator $\mathcal{L}_\infty$ on $\mathbb{S}^1$ (Proposition \ref{prop: limit linearized operator coercivity}). To do so, we extract the core part $\tilde{\varphi}_{\omega,\parallel}$ from $\varphi_{\omega,\parallel}$ in the form  
$$\varphi_{\omega,\parallel}=\tilde{\varphi}_{\omega,\parallel}+\frac{\langle \varphi_{\parallel}, Q_\infty\rangle_{L^2(\mathbb{S}^1)}}{\| Q_\infty\|_{L^2(\mathbb{S}^1)}^2} Q_\infty+\frac{\langle \varphi_{\parallel},\partial_\theta  Q_\infty\rangle_{L^2(\mathbb{S}^1)}}{\|\partial_\theta  Q_\infty\|_{L^2(\mathbb{S}^1)}^2}\partial_\theta  Q_\infty,$$
such that $\tilde{\varphi}_{\omega,\parallel}$ is orthogonal to $ Q_\infty$ and $\partial_\theta  Q_\infty$ in $L^2(\mathbb{S}^1)$. Then, using that $\mathcal{L}_\infty Q_\infty=-\frac{1}{\pi} Q_\infty^3$ and $\mathcal{L}_\infty(\partial_\theta Q_\infty)=0$, we write 
$$\begin{aligned}
\langle  \mathcal{L}_\infty  \varphi_{\omega,\parallel},\varphi_{\omega,\parallel}\rangle_{L^2(\mathbb{S}^1)}&=\langle  \mathcal{L}_\infty \tilde{\varphi}_{\omega,\parallel},\tilde{\varphi}_{\omega,\parallel}\rangle_{L^2(\mathbb{S}^1)}-\frac{2\langle \varphi_{\omega,\parallel}, Q_\infty\rangle_{L^2(\mathbb{S}^1)}}{\pi\| Q_\infty\|_{L^2(\mathbb{S}^1)}^2}\langle Q_\infty^3,\varphi_{\omega,\parallel}\rangle_{L^2(\mathbb{S}^1)}\\
&\quad-\frac{ \langle \varphi_{\omega,\parallel}, Q_\infty\rangle_{L^2(\mathbb{S}^1)}^2}{\pi\| Q_\infty\|_{L^2(\mathbb{S}^1)}^4}\| Q_\infty\|_{L^4(\mathbb{S}^1)}^4.
\end{aligned}$$
Note that by the assumptions $\varphi\perp \mathcal{Q}_\omega, \partial_\theta \mathcal{Q}_\omega$ in $L^2(\sigma_\omega)$ and the dimension reduction \eqref{direas}, one may expect that $\varphi_{\omega,\parallel}$ is almost orthogonal to $ Q_\infty$ and $\partial_\theta  Q_\infty$. Indeed, we have
$$\begin{aligned}
\langle \varphi_{\omega,\parallel}, Q_\infty\rangle_{L^2(\mathbb{S}^1)}&=\langle \varphi , Q_\infty\tilde{\Phi}_\omega\rangle_{L^2(\sigma_\omega)}=\langle \varphi,  \mathcal{Q}_\omega\rangle_{L^2(\sigma_\omega)}+\langle \varphi, Q_\infty \tilde{\Phi}_\omega- \mathcal{Q}_\omega\rangle_{L^2(\sigma_\omega)}\\
&=\langle \varphi, (Q_\infty - \mathcal{Q}_{\omega,\parallel})\tilde{\Phi}_\omega+\mathcal{P}_{\tilde{\Phi}_\omega}^\perp \mathcal{Q}_\omega\rangle_{L^2(\sigma_\omega)}=o_\omega(1)\|\varphi\|_{L^2(\sigma_\omega)}
\end{aligned}$$
and 
$$\begin{aligned}
\langle \varphi_{\omega,\parallel},\partial_\theta  Q_\infty\rangle_{L^2(\mathbb{S}^1)}&=\langle \varphi ,(\partial_\theta  Q_\infty)\tilde{\Phi}_\omega\rangle_{L^2(\sigma_\omega)}\\
&=\langle \varphi,  \partial_\theta \mathcal{Q}_\omega\rangle_{L^2(\sigma_\omega)}+\langle \varphi, \partial_\theta(  Q_\infty\tilde{\Phi}_\omega-\mathcal{Q}_\omega)\rangle_{L^2(\sigma_\omega)}\\ 
&=\langle \varphi, \partial_\theta ( Q_\infty\Phi_\omega- \mathcal{Q}_\omega)+\partial_\theta Q_\infty(\tilde{\Phi}_\omega-\Phi_\omega)\rangle_{L^2(\sigma_\omega)}=o_\omega(1)\|\varphi\|_{L^2(\sigma_\omega)}.
\end{aligned}$$
Therefore, it follows that  
$$\begin{aligned}
\langle  \mathcal{L}_\infty  \varphi_{\omega,\parallel},\varphi_{\omega,\parallel}\rangle_{L^2(\mathbb{S}^1)}&\geq\langle  \mathcal{L}_\infty \tilde{\varphi}_{\omega,\parallel},\tilde{\varphi}_{\omega,\parallel}\rangle_{L^2(\mathbb{S}^1)}-o_\omega(1)\|\varphi\|_{L^2(\sigma_\omega)}\|\varphi_{\omega,\parallel}\|_{L^2(\mathbb{S}^1)}\\
&\quad-o_\omega(1)\|\varphi\|_{L^2(\sigma_\omega)}^2
\end{aligned}$$
and
$$
\|\tilde{\varphi}_{\omega,\parallel}\|_{H^1(\mathbb{S}^1)}^2\ge \| \varphi_{\omega,\parallel}\|_{H^1(\mathbb{S}^1)}^2-o_\omega(1)\|\varphi\|_{L^2(\sigma_\omega)}^2.
$$
Consequently, the coercivity estimate for $\mathcal{L}_\infty$ (see \eqref{coercive for}) and the Cauchy-Schwarz inequality yield 
$$\begin{aligned}
\langle  \mathcal{L}_\infty  \varphi_{\omega,\parallel},\varphi_{\omega,\parallel}\rangle_{L^2(\mathbb{S}^1)}&\geq \big(L_\infty-o_\omega(1)\big)\|\tilde{\varphi}_{\omega,\parallel}\|_{H^1(\mathbb{S}^1)}^2-o_\omega(1)\|\varphi\|_{L^2(\sigma_\omega)}^2\\
&\geq \big(L_\infty-o_\omega(1)\big)\| \varphi_{\omega,\parallel}\|_{H^1(\mathbb{S}^1)}^2-o_\omega(1)\|\varphi\|_{L^2(\sigma_\omega)}^2.
\end{aligned}$$
Finally, applying \eqref{simple L^2 bound}, we complete the proof of \eqref{rem3}.
\end{proof}

\subsection{Uniqueness of a ground state}

To prove uniqueness, we introduce the functional
$$\mathcal{I}_\omega(v)=\mathcal{E}_\omega(v)+\frac{\mu_\omega}{2}\mathcal{M}_\omega(v).$$
For contradiction, we assume that two positive minimizers $\mathcal{Q}_\omega$ and $\tilde{\mathcal{Q}}_\omega$ exist for $\mathcal{J}_\omega^{(3D)}(m)$ such that $\tilde{\mathcal{Q}}_\omega(\cdot,\cdot,\cdot-\theta_*)\not\equiv \mathcal{Q}_\omega$ for all $0\leq \theta_*\leq 2\pi$. By the dimension reduction limit (Proposition \ref{adprop} (iii)), translating in $\theta$ if necessary, we may assume that the $\tilde{\Phi}_\omega$-directional components $ \mathcal{Q}_{\omega,\parallel}$ and $\tilde{\mathcal{Q}}_{\omega,\parallel}$ both converge to $Q_\infty$ in $H^1(\mathbb{S}^1)$ as $\omega\to \infty$, and $\lim_{\omega\to\infty}\|\partial_\theta\mathcal{Q}_\omega-\partial_\theta\mathcal{Q}_{\infty}(\theta)\Phi_\omega(s,z)\|_{L^2(\sigma_\omega)}=\lim_{\omega\to\infty}\|\partial_\theta\tilde{\mathcal{Q}}_\omega-\partial_\theta\mathcal{Q}_{\infty}(\theta)\Phi_\omega(s,z)\|_{L^2(\sigma_\omega)}=0$. On the other hand, we take the angle translation parameter $\theta_\omega\in [0,2\pi]$ such that 
$$\|\tilde{\mathcal{Q}}_\omega- \mathcal{Q}_\omega(\cdot,\cdot,\cdot+\theta_\omega)\|_{L^2(\sigma_\omega)}.=\min_{\theta_*\in[0,2\pi]}\|\tilde{\mathcal{Q}}_\omega- \mathcal{Q}_\omega(\cdot,\cdot,\cdot+\theta_*)\|_{L^2(\sigma_\omega)}.$$
Then, we have $\langle \tilde{\mathcal{Q}}_\omega, (\partial_\theta \mathcal{Q}_\omega)(\cdot,\cdot,\cdot+\theta_\omega)\rangle_{L^2(\sigma_\omega)}=0$. Indeed, we have $\theta_\omega\to0$, because $\mathcal{Q}_{\omega,\parallel}, \tilde{\mathcal{Q}}_{\omega,\parallel}\to Q_\infty$. Hence, $\mathcal{Q}_\omega(\cdot,\cdot,\cdot+\theta_\omega)$ is also a minimizer whose $\tilde{\Phi}_\omega$-directional component $ \mathcal{Q}_{\omega,\parallel}(\cdot+\theta_\omega)$ converges to $Q_\infty$ in $H^1(\mathbb{S}^1)$ as $\omega\to\infty$, and $\lim_{\omega\to\infty}\|\partial_\theta\mathcal{Q}_\omega(s,z,\theta+\theta_\omega)-\partial_\theta\mathcal{Q}_{\infty}(\theta)\Phi_\omega(s,z)\|_{L^2(\sigma_\omega)}=0$. Thus, translating the angle variable for $\mathcal{Q}_\omega$ by $-\theta_\omega$, we may assume that 
$$\langle \tilde{\mathcal{Q}}_\omega, \partial_\theta \mathcal{Q}_\omega\rangle_{L^2(\sigma_\omega)}=0   
$$
and
$$
\lim_{\omega\to\infty}\|\mathcal{Q}_{\omega,\parallel}-Q_\infty\|_{H^1(\mathbb{S}^1)}= \lim_{\omega\to\infty}\|\partial_\theta\mathcal{Q}_\omega-\partial_\theta\mathcal{Q}_{\infty}(\theta)\Phi_\omega(s,z)\|_{L^2(\sigma_\omega)}=0.
$$
Next, we decompose
$$
\tilde{\mathcal{Q}}_\omega=\sqrt{1-\delta_\omega^2}\mathcal{Q}_\omega+\mathcal{R}_\omega,
$$
where $\delta_\omega>0$ is chosen so that the remainder $\mathcal{R}_\omega$ is orthogonal to $\mathcal{Q}_\omega$, i.e.,
$$\langle \mathcal{R}_\omega, \mathcal{Q}_\omega\rangle_{L^2(\sigma_\omega)}=0.$$
Indeed, by construction, $\mathcal{R}_\omega$ is perpendicular to $\partial_\theta\mathcal{Q}_\omega$ too, because 
$$\langle \mathcal{R}_\omega,\partial_\theta\mathcal{Q}_\omega \rangle_{L^2(\sigma_\omega)}=\langle \tilde{\mathcal{Q}}_\omega,\partial_\theta\mathcal{Q}_\omega \rangle_{L^2(\sigma_\omega)}-\sqrt{1-\delta_\omega^2}\langle \mathcal{Q}_\omega,\partial_\theta\mathcal{Q}_\omega \rangle_{L^2(\sigma_\omega)}=0.$$
Note that the parameter $\delta_\omega$ and  the remainder $\mathcal{R}_\omega$ are small in the sense that 
\begin{equation}\label{deler}
\delta_\omega=\frac{1}{\sqrt{m}}\|\mathcal{R}_\omega\|_{L^2(\sigma_\omega)}=o_\omega(1),
\end{equation}
because $m=\|\tilde{\mathcal{Q}}_\omega\|_{L^2(\sigma_\omega)}^2=(1-\delta_\omega^2)m+\|\mathcal{R}_\omega\|_{L^2(\sigma_\omega)}^2$ and $0\leftarrow\|\mathcal{Q}_\omega-\tilde{\mathcal{Q}}_\omega\|_{L^2(\sigma_\omega)}^2=(1-\sqrt{1-\delta_\omega^2})^2m+ \|\mathcal{R}_\omega\|_{L^2(\sigma_\omega)}^2$.

Now, we are ready to deduce a contradiction extracting $\mathcal{I}_\omega(\mathcal{Q}_\omega)$ from
\begin{align*}
\mathcal{I}_\omega(\tilde{\mathcal{Q}}_\omega)&=\frac{\omega}{2}\|\sqrt{1-\delta_\omega^2}\mathcal{Q}_\omega+\mathcal{R}_\omega\|_{{\dot{\Sigma}}_{\omega; (s,z)}}^2+\frac{1}{2}\big\|\partial_\theta (\sqrt{1-\delta_\omega^2}\mathcal{Q}_\omega+\mathcal{R}_\omega)\big\|_{L^2(\frac{1}{\sigma_\omega})}^2\\
&\quad -\frac{1}{4}\|\sqrt{1-\delta_\omega^2}\mathcal{Q}_\omega+\mathcal{R}_\omega\|_{L^4(\sigma_\omega)}^4+\frac{\mu_\omega}{2}\|\sqrt{1-\delta_\omega^2}\mathcal{Q}_\omega+\mathcal{R}_\omega\|_{L^2(\sigma_\omega)}^2.
\end{align*}
Indeed, expanding and then reorganizing terms with respect to degree of $\mathcal{R}_\omega$, it can be expressed as 
\begin{align*}
\mathcal{I}_\omega(\tilde{\mathcal{Q}}_\omega)&=\frac{1-\delta_\omega^2}{2}\left\{\omega\|\mathcal{Q}_\omega\|_{{\dot{\Sigma}}_{\omega; (s,z)}}^2+\|\partial_\theta \mathcal{Q}_\omega\|_{L^2(\frac{1}{\sigma_\omega})}^2 -\frac{1-\delta_\omega^2}{2}\|\mathcal{Q}_\omega \|_{L^4(\sigma_\omega)}^4+\mu_\omega\|\mathcal{Q}_\omega\|_{L^2(\sigma_\omega)}^2\right\}\\
&  +\sqrt{1-\delta_\omega^2}\left\langle \left(\omega(\mathcal{H}_\omega^{(2D)}-\Lambda_\omega)-\frac{1}{\sigma_\omega^2}\partial_\theta^2+\mu_\omega\right)\mathcal{Q}_\omega-(1-\delta_\omega^2)\mathcal{Q}_\omega^3,\mathcal{R}_\omega \right\rangle_{L^2(\sigma_\omega)}\\
&+\frac12 \left\langle \left(\omega(\mathcal{H}_\omega^{(2D)}-\Lambda_\omega)-\frac{1}{\sigma_\omega^2}\partial_\theta^2+\mu_\omega\right)\mathcal{R}_\omega-3(1-\delta_\omega^2)\mathcal{Q}_\omega^2\mathcal{R}_\omega,\mathcal{R}_\omega \right\rangle_{L^2(\sigma_\omega)}\\
&-\sqrt{1-\delta_\omega^2}\langle \mathcal{Q}_\omega,\mathcal{R}_\omega^3\rangle_{L^2(\sigma_\omega)}-\frac14\|\mathcal{R}_\omega\|_{L^4(\sigma_\omega)}^4.
\end{align*}
For the zeroth order terms in the above expansion, extracting $\mathcal{I}_\omega(\mathcal{Q}_\omega)$, we write
\begin{align*}
&\frac{1-\delta_\omega^2}{2}\left\{\omega\|\mathcal{Q}_\omega\|_{{\dot{\Sigma}}_{\omega; (s,z)}}^2+\|\partial_\theta \mathcal{Q}_\omega\|_{L^2(\frac{1}{\sigma_\omega})}^2+\mu_\omega\|\mathcal{Q}_\omega\|_{L^2(\sigma_\omega)}^2 -\frac{1-\delta_\omega^2}{2}\|\mathcal{Q}_\omega \|_{L^4(\sigma_\omega)}^4\right\}\\
&=\mathcal{I}_\omega(\mathcal{Q}_\omega)-\frac{\delta_\omega^2}{2}\left\{\omega\|\mathcal{Q}_\omega\|_{{\dot{\Sigma}}_{\omega; (s,z)}}^2+\|\partial_\theta \mathcal{Q}_\omega\|_{L^2(\frac{1}{\sigma_\omega})}^2+\mu_\omega\|\mathcal{Q}_\omega\|_{L^2(\sigma_\omega)}^2 -\frac{2-\delta_\omega^2}{2}\|\mathcal{Q}_\omega \|_{L^4(\sigma_\omega)}^4\right\}\\
&=\mathcal{I}_\omega(\mathcal{Q}_\omega)-\frac{\delta_\omega^4}{4}\|\mathcal{Q}_\omega \|_{L^4(\sigma_\omega)}^4=\mathcal{I}_\omega(\mathcal{Q}_\omega)+o_\omega(1)\|\mathcal{R}_\omega\|_{L^2(\sigma_\omega)}^2, 
\end{align*}
where the elliptic equation (or \eqref{po1}) is used in the second identity and \eqref{deler} is employed in the last step. For the first order term, cancelling by the equation \eqref{elp}, we obtain
\begin{align*}
&\sqrt{1-\delta_\omega^2}\left|\left\langle \left(\omega(\mathcal{H}_\omega^{(2D)}-\Lambda_\omega)-\frac{1}{\sigma_\omega^2}\partial_\theta^2+\mu_\omega\right)\mathcal{Q}_\omega-(1-\delta_\omega^2)\mathcal{Q}_\omega^3,\mathcal{R}_\omega \right\rangle_{L^2(\sigma_\omega)}\right|\\
&=\sqrt{1-\delta_\omega^2}\delta_\omega^2\big|\langle \mathcal{Q}_\omega^3,\mathcal{R}_\omega\rangle_{L^2(\sigma_\omega)}\big|\leq \delta_\omega^2\|\mathcal{Q}_\omega\|_{L^6(\sigma_\omega)}^3\|\mathcal{R}_\omega\|_{L^2(\sigma_\omega)}\\
&=o_\omega(1)\|\mathcal{R}_\omega\|_{L^2(\sigma_\omega)}^2\quad\textup{(by \eqref{deler})}.
\end{align*}
For the higher-order terms, by Young’s
inequality, Corollary \ref{GN1}, Proposition \ref{existence of a minimizer} (iii) and Proposition \ref{adprop} (i) $(\|\mathcal{R}_\omega\|_{\dot{\Sigma}_{\omega;(s,z)}}=O(\omega^{-\frac34}))$,   we get
\begin{equation}\label{hir}
\begin{aligned}
\|\mathcal{R}_\omega\|_{L^4(\sigma_\omega)}^4&\lesssim  \|\mathcal{R}_\omega\|_{L^2(\sigma_\omega)} \|\mathcal{R}_\omega\|_{\dot{\Sigma}_{\omega;(s,z)}}^2+\sqrt{\omega}\|\mathcal{R}_\omega\|_{L^2(\sigma_\omega)}\|\mathcal{R}_\omega\|_{\dot{\Sigma}_{\omega;(s,z)}}^3 +o_\omega(1)\|\mathcal{R}_\omega \|_{L^2(\sigma_\omega)}^2 \\ 
&=o_\omega(1)(\|\mathcal{R}_\omega \|_{L^2(\sigma_\omega)}^2+\|\mathcal{R}_\omega\|_{\dot{\Sigma}_{\omega;(s,z)}}^2)
\end{aligned}
\end{equation}
and
\begin{align*}
\big|\langle \mathcal{Q}_\omega,\mathcal{R}_\omega^3\rangle_{L^2(\sigma_\omega)}\big|&\leq \|\mathcal{Q}_\omega\|_{L^4(\sigma_\omega)}\|\mathcal{R}_\omega\|_{L^4(\sigma_\omega)}^3\\
&\lesssim \|\mathcal{R}_\omega\|_{L^2(\sigma_\omega)}^\frac34 \|\mathcal{R}_\omega\|_{\dot{\Sigma}_{\omega;(s,z)}}^\frac32+\omega^\frac38\|\mathcal{R}_\omega\|_{L^2(\sigma_\omega)}^\frac34\|\mathcal{R}_\omega\|_{\dot{\Sigma}_{\omega;(s,z)}}^\frac94 + o_\omega(1)\|\mathcal{R}_\omega \|_{L^2(\sigma_\omega)}^2\\
&=o_\omega(1)(\|\mathcal{R}_\omega \|_{L^2(\sigma_\omega)}^2+\|\mathcal{R}_\omega\|_{\dot{\Sigma}_{\omega;(s,z)}}^2).
\end{align*} 
For the second order terms, we extract the linearized operator 
\begin{align*}
&\left\langle \left(\omega(\mathcal{H}_\omega^{(2D)}-\Lambda_\omega)-\frac{1}{\sigma_\omega^2}\partial_\theta^2+\mu_\omega\right)\mathcal{R}_\omega-3(1-\delta_\omega^2)\mathcal{Q}_\omega^2\mathcal{R}_\omega,\mathcal{R}_\omega \right\rangle_{L^2(\sigma_\omega)}\\
&=\langle \mathcal{L}_\omega  \mathcal{R}_\omega, \mathcal{R}_\omega  \rangle_{L^2(\sigma_\omega)}+3\delta_\omega^2\langle \mathcal{Q}_\omega^2\mathcal{R}_\omega,\mathcal{R}_\omega  \rangle_{L^2(\sigma_\omega)},
\end{align*}
and then apply the $L^4$ estimate \eqref{hir} and the coercivity estimate (Proposition \ref{coc1}) to obtain 
\begin{align*}
&\left\langle \left(\omega(\mathcal{H}_\omega^{(2D)}-\Lambda_\omega)-\frac{1}{\sigma_\omega^2}\partial_\theta^2+\mu_\omega\right)\mathcal{R}_\omega-3(1-\delta_\omega^2)\mathcal{Q}_\omega^2\mathcal{R}_\omega,\mathcal{R}_\omega \right\rangle_{L^2(\sigma_\omega)}\\
&\geq\frac{L_\infty}{8}\| \mathcal{R}_{\omega,\parallel}\|_{H^1(\mathbb{S}^1)}^2+\frac{\omega}{4}\|\mathcal{R}_\omega\|_{\dot{\Sigma}_{\omega;(s,z)}}^2-o_\omega(1)\|\mathcal{R}_\omega\|_{L^2(\sigma_\omega)}^2\\
&\geq\frac{L_\infty}{16}\| \mathcal{R}_\omega\|_{L^2(\sigma_\omega)}^2+\frac{\omega}{8}\|\mathcal{R}_\omega\|_{\dot{\Sigma}_{\omega;(s,z)}}^2\quad\textup{(by \eqref{simple L^2 bound})},
\end{align*}
provided that $\omega\geq1$ is large enough.  

Putting all together, we obtain 
$$\mathcal{I}_\omega(\tilde{\mathcal{Q}}_\omega)\geq \mathcal{I}_\omega(\mathcal{Q}_\omega)+\frac{L_\infty}{32}\| \mathcal{R}_\omega\|_{L^2(\sigma_\omega)}^2,$$
which contradicts to our assumptions that $\mathcal{I}_\omega(\tilde{\mathcal{Q}}_\omega)=\mathcal{I}_\omega(\mathcal{Q}_\omega)$ and $\mathcal{R}_\omega\not\equiv 0$.

\appendix

\section{Energy minimization for the 1D periodic NLS}\label{sec: Energy minimization for the cubic NLS on the unit circle}

Following \cite{GLT}, we review characterization of a ground state for the 1D periodic energy minimization
\begin{equation}\label{appendix: circle minimization}
J_\infty^{(1D)}(m)=\inf \left\{E_\infty[w]: w\in H^1(\mathbb{S}^1)\textup{ and }\|w\|_{L^2(\mathbb{S}^1)}^2=m\right\},
\end{equation}
where 
$$E_\infty[w]= \frac{1}{2}\|w'\|_{L^2(\mathbb{S}^1)}^2+\frac{\kappa}{8\pi}\|w\|_{L^4(\mathbb{S}^1)}^4\quad(\kappa=\pm1).$$
Then, we prove a coercivity estimate for the linearized operator at a ground state; it is a key ingredient for the proof of uniqueness of a 3D constrained energy minimizer.

\subsection{Ground states on the unit circle; a review}
To begin with, we summarize basic facts about the Jacobi elliptic functions, because a ground state can be represented by them; we refer to \cite{La} for more details of elliptic functions. Given $k\in(0,1)$, we define the \textit{incomplete elliptic integral of the first kind} (resp., \textit{the second kind}) by 
$$F(\phi;k):=\int_0^\phi\frac{ds}{\sqrt{1-k^2\sin^2(s)}}\quad\left(\textup{resp., }E(\phi;k):=\int_0^\phi  \sqrt{1-k^2\sin^2(s)} ds\right).$$
In particular, when $\phi=\frac{\pi}{2}$, it is called the \textit{complete elliptic integral of the first kind} (resp., \textit{the second kind}), i.e., 
\begin{equation}\label{K and E functions}
K(k):=F(\tfrac{\pi}{2}; k),\quad\left(\textup{resp., }E(k):=E(\tfrac{\pi}{2};k)\right),
\end{equation}
satisfying that $K(k)\to\frac{\pi}{2}$ as $k\to0^+$ and $K(k)\to\infty$ as $k\to1^-$.

Next, using the inverse of $z=F(\phi;k)$, denoted by $\phi(z;k)$, we define the \textit{snoidal} (resp., \textit{cnoidal}, \textit{dnoidal}) function by 
$$\textbf{sn}(z;k):=\sin\phi(z;k)\left(\textup{resp., }\textbf{cn}(z;k):=\cos\phi(z;k),\ \textbf{dn}(z;k):=\sqrt{1-k^2\textbf{sn}^2(z;k)}\right).$$
These three functions are called the \textit{Jacobi elliptic functions}. Note that $\textbf{sn}(z;k)$ and $\textbf{cn}(z;k)$ are $4K(k)$-periodic, but $\textbf{dn}(z;k)$ is $2K(k)$-periodic. Moreover, we have
\begin{equation}\label{sncsdn algebra}
\begin{aligned}
1&=\textbf{sn}^2+\textbf{cn}^2=k^2\textbf{sn}^2+\textbf{dn}^2,\\
\textbf{sn}'=(\textbf{cn})(\textbf{dn})&,\quad  \textbf{cn}'=-(\textbf{sn})(\textbf{dn}), \quad \textbf{dn}'=-k^2(\textbf{cn})(\textbf{sn})
\end{aligned}
\end{equation}
and
\begin{equation}\label{E(k)}
E(k)=\int_0^{K(k)}\textbf{dn}^2(z;k)dz.
\end{equation}
Thus, it follows that $\textbf{dn}(z;k)$ solves the nonlinear elliptic equation on $(0,2K(k))$, 
\begin{equation}\label{dn elliptic equation}
-\textbf{dn}''-2\textbf{dn}^3=-(2-k^2)\textbf{dn}.
\end{equation}

Coming back to the minimization problem \eqref{appendix: circle minimization}, we may assume that 
\begin{equation}\label{circle Q fixing}
Q_\infty(0)=\max_{\theta\in\mathbb{S}^1}|Q_\infty(\theta)|>0
\end{equation}
without loss of generality, since the variational problem is invariant under translation and phase shift. The following proposition asserts that a minimizer is uniquely determined, and that in the focusing case, a ground state is given by the dnoidal function provided that the mass is large enough; otherwise, it is constant. 

\begin{proposition}[Ground state on the unit circle;$\textup{\cite[Proposition 3.2]{GLT}}$]\label{prop: circle ground state}
Suppose that $Q_\infty$ is a minimizer for the problem $J_\infty^{(1D)}(m)$ with \eqref{circle Q fixing}. \begin{enumerate}[$(i)$]
\item If $\kappa=1$ or if $\kappa=-1$ and $0<m\le 2\pi^2$, then $Q_\infty\equiv \sqrt{\frac{m}{2\pi}}$.
\item If $\kappa=-1$ and $m>2\pi^2$, then $Q_\infty=\frac{1}{\alpha}\textup{\textbf{dn}}(\frac{\theta}{\beta};k)$, where $\alpha, \beta$ and $k$ are uniquely determined by 
\begin{equation}\label{alpha beta K relation}
4\pi \alpha^2=\beta^2, \quad K(k)\beta=\pi, \quad \ m=8E(k)K(k).
\end{equation}
\end{enumerate}
\end{proposition}

\begin{remark}\label{1D coercivity; easy case}
The minimizer $Q_\infty$ solves the nonlinear elliptic equation 
\begin{equation}\label{appendix: 1d circle elliptic equation}
-Q_\infty'' +\frac{\kappa}{2\pi}(Q_\infty)^3=-\mu_\infty Q_\infty
\end{equation}
with the Lagrange multiplier\footnote{When $\kappa=-1$ and $m>2\pi^2$, it follows from \eqref{dn elliptic equation} and \eqref{alpha beta K relation}.}
\begin{equation}\label{appendix: 1d circle lagrange multiplier}
\mu_\infty=\left\{\begin{aligned}&-\frac{m}{(2\pi)^2} &&\textup{if }\kappa=1,\\
&\frac{m}{(2\pi)^2} &&\textup{if }\kappa=-1\textup{ and }0<m\le 2\pi^2,\\
&\frac{(2-k^2)K(k)^2}{\pi^2} &&\textup{if }\kappa=-1\textup{ and }m>2\pi^2.
\end{aligned}\right.
\end{equation}
\end{remark}

\subsection{Linearized operator at the ground state $Q_\infty$}
Our next goal is to establish a coercivity estimate for the linearized operator
$$\mathcal{L}_\infty=-\partial_\theta^2+\mu_\infty+\frac{3\kappa}{2\pi}(Q_\infty)^2$$
at the ground state $Q_\infty$, that is, if $\kappa=1$ or    $\kappa=-1$ and $m\in(0,2\pi^2)\cup(2\pi^2,\infty)$, then there exists $L_\infty>0$ such that 
\begin{equation}\label{coercive for}
\langle \mathcal{L}_\infty\varphi, \varphi\rangle_{L^2(\mathbb{S}^1)}\ge L_\infty\|\varphi\|_{H^1(\mathbb{S}^1)}^2
\end{equation}
for all $\varphi\in H^1(\mathbb{S}^1)$ such that $\langle \varphi,   Q_\infty\rangle_{L^2(\mathbb{S}^1)}=\langle \varphi, Q_\infty'\rangle_{L^2(\mathbb{S}^1)}=0$.

\begin{remark}\label{1d linearized operator; constant ground state}
If $Q_\infty$ is constant (see \eqref{appendix: 1d circle lagrange multiplier}), then $\mathcal{L}_\infty$ has the following simple lower bound. If $\kappa=1$, then $\mathcal{L}_\infty=-\partial_\theta^2+\frac{m}{2\pi^2}$ is obviously strictly positive. On the other hand,  if $\kappa=-1$ and $0<m\leq 2\pi^2$, then $\mathcal{L}_\infty=-\partial_\theta^2-\frac{m}{2\pi^2}$ satisfies $\langle \mathcal{L}_\infty \varphi,\varphi\rangle_{L^2(\mathbb{S}^1)}\geq \frac{1}{2}(1-\frac{m}{2\pi^2})\|\varphi\|_{H^1(\mathbb{S}^1)}^2$ for all $\varphi\in H^1(\mathbb{S}^1)$ such that $\langle \varphi, Q_\infty\rangle_{L^2(\mathbb{S}^1)}=0$.
\end{remark}

By Remark \ref{1d linearized operator; constant ground state}, it suffices to consider the case $\kappa=-1$ and $m>2\pi^2$. We prove the following coercive estimate.
\begin{proposition}\label{prop: limit linearized operator coercivity}
If $\kappa=-1$ and $m>2\pi^2$, then there exists $L_\infty>0$ such that 
$$\langle \mathcal{L}_\infty\varphi, \varphi\rangle_{L^2(\mathbb{S}^1)}\ge L_\infty\|\varphi\|_{H^1(\mathbb{S}^1)}^2$$
for all $\varphi\in H^1(\mathbb{S}^1)$ such that $\langle \varphi,   Q_\infty\rangle_{L^2(\mathbb{S}^1)}=\langle \varphi, Q_\infty'\rangle_{L^2(\mathbb{S}^1)}=0$.
\end{proposition}

 We recall  $Q_\infty=\frac{1}{\alpha}\textup{\textbf{dn}}(\frac{\theta}{\beta};k)$ if $\kappa=-1$ and $m>2\pi^2$ (see Proposition \ref{prop: circle ground state} $(ii)$). Thus, for the proof, we need the spectral property for the operator $L^{\textup{\textbf{dn}}}_+:=-\partial_z^2+(2-k^2)-6\textup{\textbf{dn}}^2$.

\begin{lemma} [Spectrum of $L^{\textup{\textbf{dn}}}_+$ \textup{\cite[Section 4.1]{GLT}}]\label{spectra of L_+^dn}
For $k\in(0,1)$, the first three lowest eigenvalues of $L^{\textup{\textbf{dn}}}_+$ are given by 
$$\lambda_0=(k^2-2)-2\sqrt{k^4-k^2+1},\quad  \lambda_1=0, \quad \lambda_2=(k^2-2)+2\sqrt{k^4-k^2+1}$$
with the corresponding eigenstates (in order)
$$\chi_-=1-\big(k^2+1-\sqrt{k^4-k^2+1}\big)\textup{\textbf{sn}}^2, \quad  \textup{\textbf{dn}}', \quad \chi_+=1-\big(k^2+1+\sqrt{k^4-k^2+1}\big)\textup{\textbf{sn}}^2.$$
\end{lemma}

\begin{proof}[Proof of Proposition \ref{prop: limit linearized operator coercivity}]
In order to prove the result, it suffices to show that
\begin{equation}\label{asp1}
\langle \mathcal{L}_\infty\varphi, \varphi\rangle_{L^2(\mathbb{S}^1)}\gtrsim\|\varphi\|_{L^2(\mathbb{S}^1)}^2
\end{equation}
for all $\varphi\in L^2(\mathbb{S}^1)$ such that $\langle \varphi,   Q_\infty\rangle_{L^2(\mathbb{S}^1)}=\langle \varphi, Q_\infty'\rangle_{L^2(\mathbb{S}^1)}=0$. Indeed, if \eqref{asp1} holds but Proposition \ref{prop: limit linearized operator coercivity} is false, there exists a sequence $\{\varphi_n\}_{n=1}^\infty\subset H^1(\mathbb{S}^1)$ such that $\langle \varphi_n,   Q_\infty\rangle_{L^2(\mathbb{S}^1)}=\langle \varphi_n, Q_\infty'\rangle_{L^2(\mathbb{S}^1)}=0$, $\|\varphi_n\|_{H^1(\mathbb{S}^1)}=1$ and $\langle \mathcal{L}_\infty\varphi_n, \varphi_n\rangle_{L^2(\mathbb{S}^1)}\to 0$ as $n\to \infty$. Then, \eqref{asp1} implies that $\|\varphi_n\|_{L^2(\mathbb{S}^1)}\to 0$ as $n\to \infty$, and so is $\|\varphi_n\|_{L^4(\mathbb{S}^1)}\to 0$ by the Gagliardo-Nirenberg inequality \eqref{circle GN ineq}. Thus, it follows that 
$$o_n(1)=\langle \mathcal{L}_\infty\varphi_n, \varphi_n\rangle_{L^2(\mathbb{S}^1)}\ge \min\{1,\mu_\infty\}\|\varphi_n\|_{H^1(\mathbb{S}^1)}+o_n(1)\to \min\{1,\mu_\infty\},$$
which deduces a contradiction.

To prove \eqref{asp1}, for sufficiently small $\epsilon\in \R$ and any $\psi\in H^1(\mathbb{S}^1)$ such that $\|\psi\|_{L^2(\mathbb{S}^1)}=1$ and $\langle Q_\infty, \psi\rangle_{L^2(\mathbb{S}^1)}=0$, we define
$$g(\epsilon):=E_\infty\left(\sqrt{m}\frac{Q_\infty+\epsilon \psi}{\|Q_\infty+\epsilon \psi\|_{L^2(\mathbb{S}^1)}}\right).$$
By direct computations with the identity $\|\partial_\theta Q_\infty\|_{L^2(\mathbb{S}^1)}^2-\frac{1}{2\pi}\|  Q_\infty\|_{L^4(\mathbb{S}^1)}^4=-\mu_\infty m$, it can be written as 
$$g(\epsilon)=g(0)+\frac{\epsilon^2}{2}\langle\mathcal{L}_\infty\psi, \psi\rangle_{L^2(\mathbb{S}^1)}+O(\epsilon^3).$$
However, since $Q_\infty$ is a positive minimizer to the 1D energy minimization problem \eqref{circle minimization}, we must have 
$$\langle\mathcal{L}_\infty\psi, \psi\rangle_{L^2(\mathbb{S}^1)}\ge 0$$
for all $\psi\in H^1(\mathbb{S}^1)$ with $\|\psi\|_{L^2(\mathbb{S}^1)}=1$ and $\langle Q_\infty, \psi\rangle_{L^2(\mathbb{S}^1)}=0$. On the other hand, it is known that $\ker \mathcal{L}_\infty=\textup{span}\{Q_\infty'\}$ and $\mathcal{L}_\infty$ has discrete eigenvalues (see  Proposition \ref{prop: circle ground state} $(ii)$ and  Lemma \ref{spectra of L_+^dn}). Therefore, we prove the claim \eqref{asp1}.
\end{proof}

\section{Defocusing case}\label{sec: defocusing}
In this appendix, we briefly state the results of the existence/uniqueness of an energy minimizer and  its dimension reduction for the defocusing case.  In the defocusing case, we consider the standard  mass constraint energy minimization problem
\begin{equation}\label{variational problem; defocusing}
\boxed{\quad J_\omega^{(3D)}(m):=\min\Big\{E_\omega[u]: \  u\in\Sigma,\ M[u]=m  \Big\}\quad}
\end{equation}
where the mass and the energy are given by \eqref{eq: mass} and \eqref{eq: energy} with $\kappa=1$ respectively. 
The following existence and uniqueness   result was proved in \cite[Theorem 2.1]{LSY}.

\begin{theorem}[Existence and uniqueness of an energy minimizer; defocusing case]\label{thm: existence defocusing}
Let $\kappa=1$ and $\omega\ge1$.
Then the problem \eqref{variational problem; defocusing}
has a positive unique minimizer $Q_\omega$, up to phase shift, and it solves the nonlinear elliptic equation
$$\omega(H_\omega-\Lambda_\omega) Q_\omega+ \sqrt{\omega} Q_\omega^3=-\mu_\omega Q_\omega,$$
where $\mu_\omega\in \R$ is a Lagrange multiplier. 
\end{theorem}

Next, we state the dimension reduction from 3D to the 1D minimizer, and we briefly sketch the ideas of the proof since the arguments are almost same with the focusing case($\kappa=-1$).

\begin{theorem}[Dimension reduction; defocusing case $\kappa=1$]\label{drth; defocusing} 
Let $Q_\omega$ and $\mu_\omega$ be the minimizer and a Lagrange multiplier in Theorem \ref{thm: existence defocusing}. Then  there exists $\{\mathcal{O}_\omega\}_{\omega\gg1}\subset \textup{SO}(2)$  such that 
$$   \left\|Q_\omega(\mathcal{O}_\omega y,z)-\sqrt{\frac{m}{2\pi}} \left(\chi(|y|)\omega^{-\frac14} \Phi_0(|y|-\sqrt{\omega},z)\right) \right\|_{\Sigma}+\left|\mu_\omega+\frac{m}{(2\pi)^2}\right|\to0 \textup{ as } \omega\to \infty,$$
where $\Phi_{0}(s,z)=\frac{1}{\sqrt{\pi}}e^{-\frac{s^2+z^2}{2}}$  and $\chi:[0,\infty)\to[0,1]$ is a smooth function such that $\chi\equiv0$ on $[0,1]$ and $\chi\equiv1$ on $[2,\infty)$.
\end{theorem}
\begin{proof}
 As in Section \ref{sec: reformulation}, we reformulate  the problem with  suitable changes of variables and recall that  the problem $J_\infty^{(1D)}(m)$ has a minimizer  $\sqrt{\frac{m}{2\pi}}$ with  the Lagrange multiplier $\mu_\infty=-\frac{m}{(2\pi)^2}$ for defocusing case $\kappa=1$ (see Proposition \ref{prop: circle ground state}).

\noindent{\bf Step 1.} Energy upper bound ($\mathcal{J}_\omega^{(3D)}(m)\le J_\infty^{(1D)}(m)+O(\omega^{-\frac12})$).
Using the test function $\sqrt{\frac{m}{2\pi}}\chi_\omega(s)\Phi_\omega(s,z) $, where $\sqrt{\frac{m}{2\pi}}$ is the ground state for 1D circle problem $J_\infty^{(1D)}(m)$ in the defocusing case, we can prove the energy upper bound (see Lemma \ref{upper}).

\noindent{\bf Step 2.} Dimension reduction.
Applying the energy bound (Step 1) and arguments  in Proposition \ref{existence of a minimizer} (iii), we can prove the following uniform bound:
$$\sqrt{\omega}\|\mathcal{Q}_\omega\|_{{\dot{\Sigma}}_{\omega;(s,z)}}, \|\partial_\theta\mathcal{Q}_\omega\|_{L^2(\frac{1}{\sigma_\omega})}, \|\mathcal{Q}_\omega\|_{L^6(\sigma_\omega)} \textup{ and } |\mu_\omega|$$ are  bounded uniformly in $\omega$. Then following the arguments in subsections \ref{uniformgaus}-\ref{dimension reduction section}, we can obtain uniform Gaussian decay of an energy minimizer (Proposition \ref{Gaussian L^2 bounds}) and dimension reduction (Proposition \ref{adprop}), and hence we can prove the result.
\end{proof}

\end{document}